\documentclass[reqno]{amsart}
\usepackage{amsmath}
\usepackage{amssymb}
\usepackage{hyperref}
\hypersetup{ colorlinks = true, urlcolor = blue, linkcolor = blue, citecolor = red }
\usepackage{xcolor}
\usepackage{enumitem}
\usepackage{xparse}
\usepackage{textcomp}
\usepackage[utf8]{inputenc}
\usepackage{fontenc}

\let\realItem\item 
\makeatletter
\NewDocumentCommand\myItem{ o }{%
   \IfNoValueTF{#1}%
      {\realItem}
      {\realItem[#1]\def\@currentlabel{#1}}
}
\makeatother  
\setlist[enumerate]{
    before=\let\item\myItem,       
    label=\textnormal{(\arabic*)}, 
    widest=(2')                    
}

\usepackage{esint}
\makeatletter
\def\namedlabel#1#2{\begingroup
	#2%
	\def\@currentlabel{#2}%
	\phantomsection\label{#1}\endgroup
}
\usepackage{varwidth}
\usepackage{tasks}
\DeclareMathOperator{\dv}{div}
\DeclareMathOperator{\loc}{loc}

\newcommand{\RR}{\mathbb{R}}
\newcommand{\mA}{\mathcal{A}}
\newcommand{\Om}{\Omega}
\newcommand{\na}{\nabla}
\newcommand{\pa}{\partial}
\newcommand{\sig}{\sigma}
\newcommand{\La}{\Lambda}
\newcommand{\al}{\alpha}

\newcommand{\de}{\delta}

\newcommand{\ka}{\kappa}
\newcommand{\om}{\omega}
\newcommand{\la}{\lambda}
\newcommand{\g}{g}

\newcommand{\data}{\mathit{data}}
\newcommand{\vh}{v_h}

\newcommand{\vhla}{v_h^{\La}}
\newcommand{\uu}{u_0}
\newcommand{\lai}{\la_i}

\newcommand{\vla}{v^\La}

\newcommand{\vp}{\varphi}
\DeclareMathOperator*{\esssup}{ess\,sup}
\newcommand{\NN}{\mathbb{N}}

\newcommand{\mh}{\widehat{\Upsilon}}

\newcommand{\Pizo}{\Pi_{z_0}}
\newcommand{\sdel}{\delta^*}
\newcommand{\mfc}{c_0}
\newcommand{\mfs}{\gamma}

\newcommand{\mc}{\mathcal{C}}
\newcommand{\co}{c_{rad}}
\newcommand{\Q}{\mathcal{Q}}
\newcommand{\B}{\mathcal{B}}
\newcommand{\ti}{\mathcal{S}}

\theoremstyle{plain}
\newtheorem{theorem}{Theorem}[section]
\newtheorem{lemma}[theorem]{Lemma}

\newtheorem{definition}[theorem]{Definition}
\newtheorem{proposition}[theorem]{Proposition}

\def\Xint#1{\mathchoice
	{\XXint\displaystyle\textstyle{#1}}%
	{\XXint\textstyle\scriptstyle{#1}}%
	{\XXint\scriptstyle\scriptscriptstyle{#1}}%
	{\XXint\scriptstyle\scriptscriptstyle{#1}}%
	\!\int}
\def\XXint#1#2#3{{\setbox0=\hbox{$#1{#2#3}{\int}$}
		\vcenter{\hbox{$#2#3$}}\kern-.5\wd0}}

\def\Yint#1{\mathchoice
	{\YYint\displaystyle\textstyle{#1}}%
	{\YYint\textstyle\scriptstyle{#1}}%
	{\YYint\scriptstyle\scriptscriptstyle{#1}}%
	{\YYint\scriptscriptstyle\scriptscriptstyle{#1}}%
	\!\iint}
\def\YYint#1#2#3{{\setbox0=\hbox{$#1{#2#3}{\iint}$}
		\vcenter{\hbox{$#2#3$}}\kern-.51\wd0}}
\def\longdash{{-}\mkern-3.5mu{-}} 

\def\fiint{\Yint\longdash}

\def\Xint#1{\mathchoice
	{\XXint\displaystyle\textstyle{#1}}%
	{\XXint\textstyle\scriptstyle{#1}}%
	{\XXint\scriptstyle\scriptscriptstyle{#1}}%
	{\XXint\scriptscriptstyle\scriptscriptstyle{#1}}%
	\!\int}
\def\XXint#1#2#3{{\setbox0=\hbox{$#1{#2#3}{\int}$ }
		\vcenter{\hbox{$#2#3$ }}\kern-.6\wd0}}

\def\dashint{\Xint-}

\usepackage{nameref}
\makeatletter
\let\orgdescriptionlabel\descriptionlabel
\renewcommand*{\descriptionlabel}[1]{%
	\let\orglabel\label
	\let\label\@gobble
	\phantomsection
	\edef\@currentlabel{#1}%
	\let\label\orglabel
	\orgdescriptionlabel{#1}%
}
\makeatother
\numberwithin{equation}{section}
\DeclareRobustCommand{\rchi}{{\mathpalette\irchi\relax}}
\newcommand{\irchi}[2]{\raisebox{\depth}{$#1\chi$}} 
\def\Xint#1{\mathchoice
    {\XXint\displaystyle\textstyle{#1}}%
    {\XXint\textstyle\scriptstyle{#1}}%
    {\XXint\scriptstyle\scriptscriptstyle{#1}}%
    {\XXint\scriptscriptstyle\scriptscriptstyle{#1}}%
    \!\int}
\def\XXint#1#2#3{\setbox0=\hbox{$#1{#2#3}{\int}$}
    \vcenter{\hbox{$#2#3$}}\kern-0.5\wd0}
\def\fint{\Xint-}
\def\dashint{\Xint{\raise4pt\hbox to7pt{\hrulefill}}}

\def\XXiint#1#2#3{\setbox0=\hbox{$#1{#2#3}{\iint}$}
    \vcenter{\hbox{$#2#3$}}\kern-0.5\wd0}

\usepackage[normalem]{ulem}

\begin{document}
	
\title[Very weak solutions]{Very weak solutions to \\ degenerate parabolic double-phase systems}

\author{Wontae Kim}
\address[Wontae Kim]{Korea Institute for Advanced Study, 5 Hoegi-ro, Dongdaemun-gu, Seoul 02455, Republic of Korea}
\email{wontae@kias.re.kr}

\author{Lauri Särkiö}
\address[Lauri Särkiö]{Department of Mathematics, Aalto University, P.O. BOX 11100, 00076 Aalto, Finland}
\email{lauri.sarkio@aalto.fi}

\everymath{\displaystyle}

\makeatletter
\@namedef{subjclassname@2020}{\textup{2020} Mathematics Subject Classification}
\makeatother

\begin{abstract}
We prove a local self-improving property for the gradient of very weak solutions to degenerate parabolic double-phase systems. 
The result is based on a reverse Hölder inequality with constants that are independent of the solution. Delicate methods are required to avoid a self-referential argument. In particular, we develop a new phase analysis method.
\end{abstract}

\maketitle

\section{Introduction}
In this paper, we study parabolic double-phase systems with the prototype equation
\[
u_t-\dv(|\na u|^{p-2}\na u + a(z)|\na u|^{q-2}\na u)=-\dv(|F|^{p-2}F+a(z)|F|^{q-2}F),
\]
in $\Omega_T=\Omega\times(0,T)$, where $\Omega$ is a bounded domain in $\mathbb{R}^n$, $n\geq 2$, and $T>0$. Here $1<p<q<\infty$ and $a$ is a nonnegative function that is Hölder continuous in both space and time directions. Before stating our main theorem, we discuss the context of our methods.

In the regularity theory of parabolic PDEs, the classical techniques of DeGiorgi and Moser developed for elliptic equations have been extended to the linear parabolic setting ($p=2$ and $a\equiv 0$) in \cite{MR241822,MR159139}. Although this extension was already nontrivial in the linear case, additional difficulties arise when dealing with nonlinear parabolic equations. These challenges demand analytical tools that can handle nonlinear structures. A major breakthrough in this direction was the introduction of intrinsic geometries for nonlinear parabolic PDEs, first established by DiBenedetto in the study of parabolic $p$-Laplace equations ($a\equiv0$). Intrinsic geometries compensate for the lack of homogeneity in energy estimates and provide scaling invariant structures essential for nonlinear analysis. Based on this framework, various regularity properties have been developed. 
We refer to DiBenedetto’s book \cite{MR1230384} for the proofs of Hölder and gradient Hölder continuity. Furthermore, Harnack inequalities were derived by DiBenedetto, Gianazza and Vespri in \cite{MR2413134,MR2414742,MR2731161,MR2865434}, while gradient higher integrability was proved by Kinnunen and Lewis in \cite{KL_HI}. In addition, Calderón–Zygmund type estimates were obtained by Acerbi and Mingione in \cite{MR2286632} and Schauder estimates were established by Mingione and Kuusi in \cite{MR2968162,MR3273649,MR3187676,MR3191979} and Misawa in \cite{MR1939688}.

This approach with intrinsic geometry has also been applied to other classes of nonlinear parabolic PDEs.
In the setting of porous medium equations, DiBenedetto, Gianazza and Vespri proved Hölder continuity and Harnack inequalities in \cite{MR2413134,MR2414742,MR2731161,MR2865434}, while Gianazza and Schwarzacher obtained gradient higher integrability in \cite{MR3928041,MR4019094}. 
The same methodology has been adapted to equations that are nonlinear in both time and spatial derivative parts. 
For doubly nonlinear equations, higher integrability, Hölder continuity and gradient Hölder continuity have been developed in a sequence of works \cite{bögelein2025schauderestimatesparabolicplaplace, MR4163123,  MR4287785, MR4603642, MR4500278,MR4373238, MR4797292} by Bögelein, Duzaar, Liao, Moring, Schätzler, Scheven and their collaborators.

The elliptic non-standard $(p,q)$-growth condition including the anisotropic equation was first introduced and its regularity was established by Marcellini in \cite{MR969900,MR1094446,MR1401415}. We refer to \cite{MR4779533,MR4665778,MR4965436,MR4917742} for recent progress in this direction.
Regarding the equation considered in this paper, the elliptic double-phase model was studied by Esposito, Leonetti, and Mingione in \cite{MR2076158} within the framework of nonstandard growth conditions introduced by Marcellini. The corresponding weak solution space was analyzed by Fonseca, Malý, and Mingione in \cite{MR2058167}. As the name double-phase indicates, the growth condition distinguishes two phases depending on whether the coefficient $a$ vanishes or not at a given point. The regularity theory for elliptic double-phase models was developed after the phase analysis introduced by Colombo and Mingione in \cite{MR3360738,MR3294408}. 
Moreover, they established higher integrability, Hölder continuity, and gradient Hölder continuity. They also proved Calderón–Zygmund type estimates in \cite{MR3447716}, while delicate borderline cases were later shown by De Filippis and Mingione in \cite{MR3985927}. Furthermore, the Harnack inequality was established by Baroni, Colombo, and Mingione in \cite{MR3348922}. In the spirit of nonlinear parabolic equations, our main result will be proved by combining intrinsic geometry with this phase analysis framework.

The exact definition of the solutions that we study is the following. We consider very weak solutions $u=u(z)=u(x,t)$ to the system \begin{align}\label{11}
	u_t-\dv\mA(z,\na u)=
    -\dv (|F|^{p-2}F+a(z)|F|^{q-2}F) 
    \quad\text{in}\ \Omega_T,
\end{align}
where $\mA(z,\na u):\Omega_T\times \RR^{Nn}\longrightarrow \RR^{Nn}$ with $N\ge1$ is a Carath\'eodory vector field satisfying the structure assumptions
\begin{align}\label{12}
		\mA(z,\xi)\cdot \xi\ge L_1(|\xi|^p+a(z)|\xi|^q)
      \quad\text{and}\quad 
		|\mA(z,\xi)|\le L_2(|\xi|^{p-1}+a(z)|\xi|^{q-1})
\end{align}
for a.e. $z\in \Omega_T$ and every $\xi\in \RR^{Nn}$ with constants $0<L_1\le L_2<\infty$. We assume that the coefficient function $a:\Omega_T\longrightarrow\RR^+$ is non-negative and satisfies $a\in C^{\alpha,\alpha/2}(\Omega_T)$ for some $\al \in (0,1]$. Here $a\in C^{\alpha,\alpha/2}(\Omega_T)$ means that $a\in L^{\infty}(\Omega_T)$ and there exists a constant $[a]_{\alpha,\alpha/2;\Omega_T}=[a]_\alpha>0$, such that
\begin{align}\label{eq_holder_reg}
		|a(x,t)-a(y,s)|\le [a]_{\alpha,\alpha/2;\Omega_T}\max\{ |x-y|^\alpha,|t-s|^\frac{\alpha}{2} \}
\end{align}
for every $(x,y)\in\Omega$ and $(t,s)\in (0,T)$. 
For the source term $F$, we assume that $F:\Om_T\longrightarrow\mathbb{R}^{Nn}$ is a measurable map such that $(|F|^p+a|F|^q)\in L^1(\Om_T)$. Throughout the paper, we denote $H(z,s):\Omega_T\times \RR^{+ }\longrightarrow \RR^+$ as
\begin{align*}
	H(z,s)=s^p+a(z)s^q.
\end{align*}
Regarding the exponents, we assume that
\begin{align}\label{eq_range_q}
   2<p< q < p+\frac{2\alpha}{n+2}. 
\end{align}

\begin{definition}\label{def_weak}
A function $u\in L^1(0,T;W^{1,1}(\Omega,\RR^N))$ satisfying
\[
\iint_{\Omega_T} H(z,|\na u|)^\delta\,dz <\infty
\]
is a very weak solution to \eqref{11} with integrability deficit $\de \in (0,1)$, if
\[
    \iint_{\Omega_T}\left(-u\cdot \varphi_t+\mA(z,\na u)\cdot \na\varphi\right)\,dz
    =
    \iint_{\Om_T} \left( |F|^{p-2}F\cdot \na \varphi+a(z)|F|^{q-2}F\cdot \na\varphi  \right)\,dz
\]
for every $\varphi\in C_0^\infty(\Omega_T,\RR^N)$.
\end{definition}

Our main theorem states that if the integrability deficit is close enough to 1, a very weak solution to \eqref{11} is actually a weak solution, i.e., satisfies Definition~\ref{def_weak} with $\de = 1$. The constants depend on 
\[
\data = (n,N,p,q,L_1,L_2,[a]_\al)
\]
and here, as in the rest of the paper, we denote $c=c(\data)$ for a positive constant depending at most on the constants above.

\begin{theorem}\label{HI_theorem}
Suppose $p,q$ satisfy \eqref{eq_range_q}. Then there exists $\de_0=\de_0(\mathit{data})\in(0,1)$ with $2-q(1-\delta_0)>0$
such that if $u$ is a very weak solution to \eqref{11} with integrability deficit $\de\in(\de_0,1)$ and $u\in C(0,T;L^2(\Omega_T,\mathbb{R}^N))$, we have 
	\begin{align*}
		\begin{split}
			\fiint_{Q_{R} (z_0) }H(z,|\na u|)\,dz 
            &\le c\left(\fiint_{Q_{2R} (z_0) } (  H(z,|\na u|)^\de+ H(z,|F|) ) \,dz+1\right)^{1+\frac{q(1-\de)}{2-q(1-\de)}}
		\end{split}
	\end{align*}
	 for every $Q_{2R}(z_0) \subset\Om_T$ where $c\ge
     1$ depending on $\data,\|a\|_{L^\infty(\Om_T)},\| (H(z,|\na u|)+H(z,|F|)+1 )\|_{L^\delta(\Om_T)}$, and $R\in(0,R_0)$ for $R_0\in(0,1)$ depending on $\data$, $\delta$, $\| H(z,|\na u|) \|_{L^\delta(\Om_T)}$, $\| H(z,|F|) \|_{L^\delta(\Om_T)}$, $\|u\|_{C(0,T;L^2(\Omega,\mathbb{R}^N))}$.
\end{theorem}
For weak solutions to elliptic $p$-Laplace systems, gradient higher integrability results go back to Elcrat and Meyers in \cite{Meyers}. The core principle of the proof is already the same as in our paper. First a Caccioppoli inequality and a Sobolev-Gagliardo-Nirenberg type inequality are used to obtain a reverse Hölder inequality, after which higher integrability follows from an application of Gehring's lemma. Higher integrability result similar to Theorem~\ref{HI_theorem} for elliptic $p$-Laplace systems was proven for very weak solutions by Iwaniec and Sbordone in \cite{Iwaniec_Sbordone} and independently by Lewis in \cite{Lewis_elliptic}. We mention that a conjecture in \cite{Iwaniec_Sbordone} on minimal integrability deficit for elliptic $p$-Laplace problems was  recently answered by Colombo and Tione in \cite{colombo_jems}. There is a difficulty in showing the Caccioppoli inequality for very weak solutions, as a very weak solution is not an admissible test function. This problem was overcome in \cite{Lewis_elliptic} by using a type of Lipschitz truncation based on Whitney decomposition. 

Theorem~\ref{HI_theorem} was proven for parabolic $p$-Laplace systems in \cite{KL_veryweak}. In the parabolic case, intrinsic geometry plays a role both in the Lipschitz truncation and in obtaining the reverse Hölder inequality. Based on \cite{KL_veryweak}, gradient higher integrability of very weak solutions was extended to variable exponent problem and higher order problem of $p$-Laplace type in \cite{MR4074602,verena_thesis,bogelein_li,Li}. Furthermore, corresponding arguments have been implemented in the context of Calder\'on-Zygmund type estimates in \cite{MR3928653,MR3483172} and applied to establish existence of very weak solutions for elliptic $p$-Laplace problems \cite{MR3531368,MR3500290} and linear parabolic problems \cite{MR3985550}. 
The parabolic Lipschitz truncation method has been considered also in relation to caloric approximations in \cite{MR3672391,MR4525732} and existence of solutions to fluid models in \cite{MR2668872}. 

     For parabolic double-phase problems higher integrability for weak solutions was shown in \cite{KKM,KS}. In there, phase analysis is used to distinguish whether the energy estimates of double-phase systems are perturbed into that of $p$-Laplace systems or $(p,q)$-Laplace systems depending on the pointwise value of $a$. Based on the phase division, we construct a corresponding intrinsic geometry.
 Similar phase analysis was used in \cite{KKS} to obtain standard energy estimates for weak solutions. Although a Lipschitz truncation method was applied in \cite{KKS}, the argument is different for very weak solutions.
Other recent result for parabolic double-phase problems include partial regularity in the nondegenerate case in \cite{MR4774133,kim2025boundedsolutionsinterpolativegap,MR4926910,oh2025gradient,ok2025partialregularityparabolicsystems,sen2025lipschitzregularityparabolicdouble}.
In particular, Marcellini's approach has been extended to the parabolic setting in \cite{MR3102165,MR3073153,MR3356846,MR3532237,MR3482400}.

Corresponding result to Theorem~\ref{HI_theorem} for very weak solutions to the elliptic case was proved in \cite{BBK}. Obtaining the higher integrability result for very weak solutions to parabolic double-phase problems presents several new challenges. In order to show that there exists $\de_0$ satisfying Theorem~\ref{HI_theorem}, the higher integrability gained in the argument has to be quantified and independent from $\de$. In \cite{KKM} the phase analysis in the reverse Hölder inequality argument for weak solutions is done with respect to a constant $K$ depending on the integral of the gradient. Utilizing the same argument for very weak solutions would lead to estimates depending on $\de$. To avoid with this problem, we consider cylinders with radii bounded by a constant $R_0$ appearing in Theorem~\ref{HI_theorem} and use a new phase analysis argument.

Also in the proof of the Caccioppoli inequality there are novel challenges. In \cite{KKS} the Lipschitz truncation method is used to prove a Caccioppoli inequality for weak solutions to parabolic double-phase problems. In the argument in \cite{KKS}, the term $\La|E(\La)^c|$ in \eqref{def_E} vanishes as $\La \to \infty$. For very weak solutions this does not happen and finer estimates are required. In particular, we have to compare intrinsic cylinders in the Whitney decomposition with the underlying intrinsic cylinders defined in Section~\ref{sec_setup}. In our context these cylinders have different geometries and to compare them we provide a new argument in Lemma~\ref{lem_theta_2}. Our proof of the Caccioppoli inequality is self-contained. To simplify some of the arguments of previous papers, we have included the term $\tfrac{|u-u_0|}{\rho}$ in the definition of $E(\La)$.              

Finally, we note that the range of $q$ in the parabolic double-phase system is not completely understood, however, the strict inequality for the upper bound of $q$ in \eqref{eq_range_q} seems to be natural. For the elliptic case, the range of $q$ is well understood and its optimality is known. When we consider very weak solutions in the elliptic setting, the natural integrability of the gradient does not allow the borderline case, see \cite{BBK} for the details.

\section{Preliminaries and setup of the proof}

\subsection{Notation} 
Let $\Omega$ be a bounded open subset in $\mathbb{R}^n$ and $\Omega_T=\Omega\times(0,T)$, $T>0$.  
A point in $\RR^{n+1}$ is denoted as $z=(x,t)$ where $x \in \RR^n$ and $t\in \RR$.
A ball with center $x_0\in\RR^n$ and radius $\rho>0$ is denoted as
\[
    B_\rho(x_0)=\{x\in \RR^n:|x-x_0|<\rho\}.
\]
 Let $t_0\in\RR$ and $\la>0$. Parabolic cylinders with quadratic scaling in time are denoted as
\[
    Q_\rho(z_0)=B_\rho(x_0)\times I_\rho(t_0),
\]
where
\[
 I_\rho(t_0)=(t_0-\rho^2,t_0+\rho^2).
\]

For $\la\geq1$, a $p$-intrinsic cylinder centered at $z_0=(x_0,t_0)$ is 
\begin{align}\label{def_Q_cylinder}
	Q_{\rho}^\la(z_0)=
	B_{\rho}(x_0)\times I_{\rho}^\la(t_0),
 \quad I_{\rho}^\la(t_0)=(t_0-\la^{2-p}\rho^2,t_0+\la^{2-p}\rho^2)
\end{align}
and a $(p,q)$-intrinsic cylinders centered at $z_0=(x_0,t_0)$ is
\begin{align}\label{def_G_cylinder}
G_{\rho}^\la(z_0)=B_{\rho}(x_0)\times J_{\rho}^\la(t_0),
 \quad J_\rho^{\la}(t_0)=\left(t_0-\tfrac{\la^2}{H(z_0,\la)}\rho^2,t_0+\tfrac{\la^2}{H(z_0,\la)}\rho^2\right).
\end{align}

Note that if $\la=1$, then $Q_\rho(z_0)=Q^\la_{\rho}(z_0)$. Moreover, we write
\[
    cQ_\rho^\la(z_0)=Q_{c\rho}^\la(z_0)
    \quad\text{and}\quad 
    cG_\rho^\la(z_0)=G_{c\rho}^\la(z_0)
\]
 for $c>0$. We also consider parabolic cylinders with arbitrary scaling in time and denote
 \begin{align}\label{def_ell}
     Q_{r,s}(z_0)=B_r(x_0)\times \ell_{s}(t_0),\quad \ell_s(t_0)=(t_0-s,t_0+s),\quad r,s>0.
 \end{align}
 
 The $(n+1)$-dimensional Lebesgue measure of a set $E\subset\RR^{n+1}$ is denoted as $|E|$.
For $f\in L^1(\Omega_T,\RR^N)$ and a measurable set $E\subset\Om_T$ with $0<|E|<\infty$, we denote the integral average of $f$ over $E$ as
\[
	(f)_{E}=\frac{1}{|E|}\iint_{E}f\,dz=\fiint_{E}f\,dz.
\]

\subsection{Setup of the proof}\label{sec_setup}
The rest of this paper is devoted to proving the higher integrability result in Theorem~\ref{HI_theorem}. The lower bound for the integrability deficit in the theorem, $\de_0$, will be bounded from below throughout the argument. In particular $\de_0$ will be determined in the proofs of Lemma~\ref{lem_caccioppoli} and the reverse Hölder estimates in Lemma~\ref{lem_reverse_Hölder} and Lemma~\ref{lem_reverse_Hölder_pq}. Note that $\de>\de_0$ and we assume at least that $2-q(1-\delta_0)>0$ and $\delta_0>\tfrac{q-1}{p}$.

The first part of the higher integrability proof is to show a reverse Hölder inequality in intrinsic cylinders. For this we fix an intrinsic cylinder and show a reverse Hölder inequality in it in sections \ref{sec_caccioppoli} and \ref{sec_rev_hölder}. After that we show in Section~\ref{sec_HI} that a suitable collection of intrinsic cylinders exists and finalize the higher integrability proof. 

We start by introducing the setting of the reverse Hölder inequality. Let $z_0=(x_0,t_0)\in\Om_T$, with $x_0\in\Omega$ and $t_0\in(0,T)$, be a Lebesgue point of $H(z,|\na u(z)|)^\de$ satisfying
\[
	H(z_0,|\na u(z_0)|)>\Pi
\]
for some
$\Pi>  \mfc $, where $\mfc>1$  is determined in \eqref{def_mfc}. Let 
$\pi>1$
be such that
\[
\Pi = \pi^p+a(z_0) \pi^q.
\]

We prove the reverse Hölder inequality in $p$-intrinsic and $(p,q)$-intrinsic phases. In the $p$-intrinsic case, we consider a cylinder $Q_{16\rho}^\pi(z_0)\subset \Om_T$ defined as in \eqref{def_Q_cylinder} and assume the following: 
 \begin{align}\label{eq_pcase}
     \pi^p \geq a(z_0)\pi^q,
 \end{align}
\begin{align}\label{eq_p1}
    \pi^{\de p}\leq \fiint_{Q_\rho^\pi(z_0)}  ( H(z,|\na u|)+H(z,|F|) )^\de  \,dz
\end{align}
and 
\begin{align}\label{eq_p2}
    \fiint_{Q_{16\rho}^\pi(z_0)} ( H(z,|\na u|)+H(z,|F|) )^\de \,dz\leq \pi^{\de p}.
\end{align}
   The existence of $\rho$ in \eqref{eq_p1} and \eqref{eq_p2} for the $p$-intrinsic case \eqref{eq_pcase} will be proved in Lemma~\ref{p_stopping}. Furthermore as shown in Lemma~\ref{lemma_decay},
we have in the $p$-intrinsic case that
\begin{align}\label{eq_rho_decay}
        \rho^\al \leq \pi^{p-q}. 
    \end{align}

In the $(p,q)$-intrinsic case we consider a cylinder $G_{16\rho}^\pi(z_0)\subset \Om_T$ defined as in \eqref{def_G_cylinder} and assume the following: 
 \begin{align}\label{eq_pqcase}
     \pi^p < a(z_0)\pi^q,
 \end{align}
 \begin{align}\label{eq_pq2}
     \frac{a(z_0)}{2}\leq a(z) \leq 2a(z_0) \text{ for any $z\in Q_{10\rho}(z_0)$,}
 \end{align}
\begin{align}\label{eq_pq3}
    \Pi^{\de}\leq \fiint_{G_\rho^\pi(z_0)} ( H(z,|\na u|)+H(z,|F|) )^\de  \,dz
\end{align}
and 
\begin{align}\label{eq_pq4}
    \fiint_{G_{16\rho}^\pi(z_0)}  ( H(z,|\na u|)+H(z,|F|) )^\de  \,dz\leq \Pi^\de.
\end{align}
The comparability of $a$ in \eqref{eq_pq2} and the existence of $\rho$ in \eqref{eq_pq3} and \eqref{eq_pq4} for the $(p,q)$-intrinsic case \eqref{eq_pqcase} will be verified in Lemma~\ref{p_q_comp} and Lemma~\ref{pq_stopping} respectively.

The reverse H\"older inequality will be derived from the Caccioppoli inequality in Lemma~\ref{lem_caccioppoli} under the condition of the $p$-intrinsic case \eqref{eq_pcase}-\eqref{eq_p2} or the $(p,q)$-intrinsic case \eqref{eq_pqcase}-\eqref{eq_pq4}. For this, we consider the following two cylinders.
 
For $\rho_1, \rho_2 \in [\rho, 4\rho]$  satisfying $\rho_1<\rho_2$, we denote
\[
\Q_1 = Q_{\rho_1,\sig_1}(z_0), \qquad \Q_2 = Q_{\rho_2,\sig_2}(z_0),
\]
where
\begin{align}\label{def_Qi}
    \Q_i = 
    \begin{cases}
        Q_{\rho_i}^\pi(z_0),\qquad \text{ if $\pi^p\geq a(z_0)\pi^q$,}\\ 
        G_{\rho_i}^\pi(z_0),\qquad \text{ if $\pi^p < a(z_0)\pi^q$,}
    \end{cases}
\end{align}
for $i\in\{1,2\}$. Moreover, we denote the bottom ball and the time interval of these cylinders as
\begin{align}\label{def_mathcal}
    \Q_i=\B_i\times \ti_i.
\end{align}
We also use the notation 
\begin{align}\label{def_Pizo}
    \Pizo = \begin{cases}
        \pi^p, \qquad  \text{ if $\pi^p\geq a(z_0)\pi^q$}\\ 
        \Pi, \qquad \text{ if $\pi^p < a(z_0)\pi^q$.}      
    \end{cases}
\end{align}

\subsection{Auxiliary lemmas}
In general, the time derivative of a very weak solution in Definition~\ref{def_weak} exists only in the distributional sense. 
Steklov averages are applied to mollify the function in the time direction. Let $f\in L^1(\Omega_T)$ and $0<h<T$. We define the Steklov average $[f]_h(x,t)$ for all $0<t<T$ as
\begin{align*}
	[f]_h(x,t)=
	\begin{cases}
		\fint_t^{t+h}f(x,s)\,ds,&0<t<T-h,\\
		0,&T-h\le t.
	\end{cases}
\end{align*}
We refer to \cite{MR1230384} for standard properties of the Steklov average. 
In particular, we note that for any $f\in L^2(\RR^{n+1})$ we have $[f]_h\in W^{1,2}(\RR;L^2(\RR^{n}))$ with
\[
	\pa_t[f]_h=\frac{f(x,t+h)-f(x,t)}{h}.
\]

We can derive from Definition~\ref{def_weak} the following Steklov averaged weak formulation on time-slices. For details we refer to \cite{verena_thesis}. Let $u$ be a very weak solution to \eqref{11}. We have
\begin{align}\label{steklov_averaged_wf}
\begin{split}
    &\int_\Omega\pa_t[u]_h(x,t)\cdot\varphi(x)\,dx +\int_\Omega  [\mA(\cdot,\na u)]_h(x,t)\cdot\na\varphi(x)\,dx \\
    &= \int_\Omega  [|F|^{p-2}F+a|F|^{q-2}F]_h(x,t)\cdot \na \varphi(x)   \,dx,
\end{split}
\end{align}
for all $\varphi\in W_0^{1,\infty}(\Omega,\mathbb{R}^N)$ and a.e. $t \in (0,T-h)$.

The following lemma is proved in \cite[Lemma~2.5]{KKM} and the same proof applies for very weak solutions.
\begin{lemma}\label{lem_parabolic_poincare}
    Let $u$ be a very weak solution to \eqref{11}. There exists a constant $c=c(n,N,m,L_2)$ such that
	\begin{align*}
		\begin{split}
			&\fiint_{Q_{r,s}(z_0)}\frac{|u-(u)_{Q_{r,s}(z_0)}|^{\theta m}}{r^{\theta m}}\,dz \\
            &\le c\fiint_{Q_{r,s}(z_0)}|\na u|^{\theta m}\,dz +c\left(\frac{s}{r^2}\fiint_{Q_{r,s}(z_0)} (|\na u|^{p-1}+a(z)|\na u|^{q-1}) \,dz\right)^{\theta m}\\
            &\qquad+c\left(\frac{s}{r^2}\fiint_{Q_{r,s}(z_0)}(|F|^{p-1}+a(z)|F|^{q-1})  \,dz\right)^{\theta m}
		\end{split}
	\end{align*}
 for every $Q_{r,s}(z_0)=B_{r}(x_0)\times (t_0-s,t_0+s)\subset\Omega_T$ with $r,s>0$, $m\in(1,q]$ and $\theta\in(1/m,1]$. 
\end{lemma}

Finally, we provide a standard iteration lemma, see \cite[Lemma 8.3]{MR1962933}.
\begin{lemma}\label{lem_iter}
	Let $0<r<R<\infty$ and $h:[r,R]\longrightarrow\RR$ be a non-negative and bounded function. Suppose there exist $\vartheta\in(0,1)$, $A,B\ge0$ and $\gamma>0$ such that
	\begin{align*}
		h(r_1)\le \vartheta h(r_2)+\frac{A}{(r_2-r_1)^\gamma}+B
		\quad\text{for all}\quad
		0<r\le r_1<r_2\le R.
	\end{align*}
	Then there exists a constant $c=c(\vartheta,\gamma)$ such that
	\begin{align*}
		h(r)\le c\left(\frac{A}{(R-r)^\gamma}+B\right).
	\end{align*}
\end{lemma}

\section{Caccioppoli inequality}\label{sec_caccioppoli} 
In this section we prove a Caccioppoli inequality that is used in the proof of the reverse Hölder inequality. Normally proving the Caccioppoli inequality involves using a solution as a test function. However, for very weak solutions this is not possible due to the integrability deficit. To overcome this problem we use a Lipschitz truncation of the solution as a test function and obtain the Caccioppoli inequality with a self-improving argument.

The following maximal function is essential in constructing the Lipschitz truncation. We define the strong maximal function for $f\in L^1_{\loc}(\RR^{n+1})$ as
\begin{align} \label{strong_M}
	Mf(z)=\sup_{z \in Q }\fiint_{Q} |f|\,dw,
\end{align}
where the supremum is taken over all parabolic cylinders $Q$ as in \eqref{def_ell} satisfying $z\in Q$. 
It follows immediately that
\[
	Mf(z)\le \sup_{t\in (t_1,t_2)} \fint_{t_1}^{t_2} \sup_{x\in B}\fint_{B} |f|\,dy \,ds =M_t(M_x(|f|))(z),
\]
where $M_x$ and $M_t$ are the noncentered Hardy--Littlewood maximal operators with respect to space and time. 
By repeated application of the Hardy--Littlewood--Wiener maximal function theorem we obtain the following bound for the strong maximal function.
\begin{lemma}\label{str}
Let $1<\sig<\infty$ and $f\in L^{\sigma}(\RR^{n+1})$. There exists a constant $c=c(n,\sig)$ such that
\[
	\iint_{\RR^{n+1}}|Mf|^{\sig}\,dz\le c\iint_{\RR^{n+1}}|f|^\sig\,dz.
\]
\end{lemma}

Next, we introduce the setup that we use throughout this section. Let
\begin{align}\label{def_g}
    \g= 
    \left (|\na u|+ \frac{|u-u_0|}{\rho}+|F| \right )\chi_{\Q_2},
    \quad u_0=(u)_{\Q_2},
\end{align}
where $\Q_1$ and $\Q_2$ are defined in \eqref{def_Qi}. 
As the coefficient function $a$ was initially defined only in $\Om_T$, we extend it to $\RR^{n+1}$. Moreover, we truncate the extension by its maximum and minimum in $\Q_2$. This is utilized in the proof of Lemma~\ref{lem_theta_2}.
Since $s\mapsto s^\alpha$ is a concave and subadditive function on $\RR^+$, the extended function 
\[
	\tilde{a}(z)=\inf_{w=(y,s)\in \Om_T}\left\{a(w)\rchi_{\Q_2}+[a]_{\alpha}(|x-y|^\alpha+|t-s|^{\frac{\alpha}{2}})\right\}, 
 \quad z=(x,t)\in\RR^{n+1},
\]
belongs to $C^{\alpha,\alpha/2}(\RR^{n+1})$ (see \cite[Theorem 2.7]{MR4306765} for the details).
With slight abuse of notation we denote 
\begin{align}\label{extended_a}
    a(z) = \max \left\{   \min\left\{  \Tilde{a}(z), \sup_{w\in\mathcal{Q}_2}a(w) \right\},  \inf_{w\in\mathcal{Q}_2}a(w)  \right\}.
\end{align} 
We point out that if $z\in \Om_T\setminus \Q_2$, then $a(z)$ in \eqref{extended_a} may differ from originally given $a(z)$ in \eqref{11}.
With this extension, note that for any $z\in \mathbb{R}^{n+1}$, there exists $w$ belonging to closure of $\mathcal{Q}_2$ such that 
\[
a(z)=a(w),
\]
which may not be true for $a(z)$ defined in \eqref{11}. As in \eqref{def_g}, we have restricted our attention only to $\Q_2$. However, intrinsic cylinders of the Whitney covering below may be placed outside of $\Q_2$ or even $\Om_T$.

Let $\La\geq \theta\Pi$, where $\theta = \theta(\data)\in(0,1)$ is determined later in the proof of Lemma~\ref{lem_caccioppoli}. We denote
\begin{align}\label{def_m}
 \Upsilon (z)=\left(M(\g^{q-1}+(a\g^q)^\frac{q-1}{p})(z)\right)^\frac{p}{q-1},
\end{align}
where $M(\cdot)$ is the strong maximal operator defined in \eqref{strong_M}, and let
\begin{align}\label{def_E}
	E(\La)=\left\{z\in \RR^{n+1}: \Upsilon(z)\le\La\right\}.
\end{align}
Note that by Lemma~\ref{str}, Lemma~\ref{coro_p_phase} or Lemma~\ref{sec4:lem:5}, and \eqref{eq_p2} or \eqref{eq_pq4}, there exists $c=c(\data)$ such that
    \begin{align}\label{eq_est_mz}
        \iint_{\RR^{n+1}} \Upsilon(z)^\de \,dz\leq c\iint_{\Q_2} H(z,\g)^\de\,dz \leq c\Pizo^\de|\Q_2|.
    \end{align}
Therefore, we have by Chebyshev's inequality that
\begin{align}\label{bound_bad_set}
    |E(\La)^c|\leq \frac{1}{\La^\de} \iint_{\RR^{n+1}} \Upsilon(z)^\de\,dz<\infty.
\end{align}
Moreover, since the strong maximal function is lower semicontinuous, $E(\La)^c=\RR^{n+1}\setminus E(\La)$
is an open subset of $\RR^{n+1}$. 

Assuming $E(\La)^c\ne\emptyset$, we construct a Whitney decomposition of $E(\La)^c$. Since $s\mapsto s^p+a(z)s^q$ is a strictly increasing function on $\mathbb{R}^+$, for every $z\in E(\La)^c$  there exists a unique $\la_z>1$ such that $\la_z^p+a(z)\la_z^q=\La$.
Let 
\[
K=\left(\frac{ 1+\|a\|_\infty}{|B_1|}\iint_{\RR^{n+1}} \Upsilon(z)^{\mu}\,dz\right)^\frac{\alpha}{n+2},
\]
where $\mu\in(0,1)$ is chosen to satisfy
\begin{align}\label{def_mu}
    \max\left\{  1-\frac{2}{q}+\frac{(n+2)(q-p)}{q\alpha} , \frac{q-1}{p} \right\}<\mu<1
\end{align}
and
\begin{align}\label{def_mu_2}
    \left( \frac{1}{\mu} -1 \right)\frac{2n}{(n+2)q} +\frac{n}{n+2} \le \frac{p}{q}.
\end{align}
Note that since we have $q<p+\tfrac{2\alpha}{n+2}$, there holds
\[
1-\frac{2}{q}+\frac{(n+2)(q-p)}{q\alpha}<1-\frac{2}{q}+\frac{2}{q}=1.
\]
On the other hand, since we have
\[
\frac{n}{n+2}\le \frac{p}{p+ \frac{2}{n+2}} <\frac{p}{q},
\] 
\eqref{def_mu_2} holds provided $\mu$ is sufficiently close to $1$.

We will take $\mfs\in(0,1)$ to satisfy
\begin{align}\label{def_mfs}
    p-q-\mfs> -\frac{2\alpha}{n+2}+\frac{\alpha q}{n+2}(1-\mu) .
\end{align}
Now, let $\mfc>1$ be sufficiently large to have
\begin{align}\label{def_mfc}
    K\left( \frac{\theta \mfc}{2(1+\|a\|_{\infty})} \right)^{-\frac{\mfs}{p}} \le \frac{1}{800(1+[a]_\alpha
    )}.
\end{align}
In particular, since we have by Lemma~\ref{coro_p_phase} or Lemma~\ref{sec4:lem:5} that
\[
\iint_{\mathbb{R}^{n+1}} \Upsilon^\mu\,dz\le c \iint_{\Om_T} (H(z,|\na u|)+H(z,|F|) )^{\mu}\,dz
\]
for some $c=c(\data)$,
we can take $\mfc$ to satisfy \eqref{def_mfc} depending on $\theta=\theta(\data)$, $\data$ and $\| (H(z,|\na u|)+H(z,|F|)+1 )\|_{L^\delta(\Om_T)}$.

We consider a family of metrics $\{d_z(\cdot,\cdot)\}_{z\in E(\La)^c}$ given by
\begin{align*}
    d_z(z_1,z_2)=
    \begin{cases}
        \max\left\{|x_1-x_2|,\sqrt{\la_z^{p-2}|t_1-t_2|}\right\},&\text{if}\ \la_z^p\ge a(z)\la_z^q,\\
        \max\left\{|x_1-x_2|,\sqrt{\La\la_z^{-2}|t_1-t_2|}\right\},&\text{if}\ \la_z^p< a(z)\la_z^q.
    \end{cases}
\end{align*}
For every $z\in E(\La)^c$ we define a distance from $z$ to $E(\La)$ as
\[
4r_z= d_z(z,E(\La))=\inf_{w\in E(\La)}d_z(z,w).
\]
Let $\{U_z^\La\}_{z\in E(\La)^c}$ be a family of cylinders defined as
\begin{equation*}
	U_z^{\La}=
 \begin{cases}
 Q^{\la_z}_z,&\text{if}\ \la_z^p\ge a(z)\la_z^q,\\
G^{\la_z}_z, &\text{if}\ \la_z^p< a(z)\la_z^q,
\end{cases}
\end{equation*}
where $Q^{\la_z}_z=B_{r_z}(x)\times I^{\la_z}_{r_z}(t)$ is defined as in \eqref{def_Q_cylinder} and $G^{\la_z}_z=B_{r_z}(x)\times J^{\la_z}_{r_z}(t)$ as in \eqref{def_G_cylinder}.
Note that by \eqref{bound_bad_set} we can argue as in \cite{KKS} to claim that with  $\La$ fixed, the radii $r_z$ are uniformly bounded, which is essential for the Vitali covering argument.
Indeed, since $|E(\La)^c|$ is finite by \eqref{bound_bad_set}, and $B_{r_z}(x)\times (t-\La^{-1}r_z^2, t+\La^{-1}r_z^2)\subset U_z^\La$, we have that $2\La^{-1} r_z^{n+2}\le |E(\La)^c|$. As $\La>1$ is fixed, it follows that $r_z$ are uniformly bounded in $z\in E(\La)^c$. Therefore, we are able to apply the Vitali covering argument in \cite{KKS}.

Next, we provide a Whitney  covering of $E(\La)^c$ with cylinders of the above type. To make the notation simpler, we denote here, and in the rest of this section
\begin{equation}\label{def_Ui_37}
U_i =Q_{r_i,s_i}(z_i) = B_i \times I_i = 
 \begin{cases} 
 Q_i,&\text{if}\ \lai^p \geq a(z_i)\lai^q,\\
 G_i,&\text{if}\ \lai^p < a(z_i)\lai^q,
 \end{cases}
\end{equation}
where $i\in\mathbb{N}$ and
\[
	r_i=r_{z_i},\quad s_i=\frac{1}{2}|I_i|,\quad \la_i=\la_{z_i},\quad  Q_i=Q^{\la_i}_{r_i}(z_i),\quad G_i=G_{r_i}^{\la_i}(z_i).
\]
Moreover, we write the following subfamily of $\mathbb{N}$ as
\[
\mathcal{I}_i=\{j\in\mathbb{N}: 2U_i\cap 2U_j\ne\emptyset\}
\]
and the distance from $U_i$ to $E(\La)$ as
\[
d_i(U_i,E(\La))=\inf_{z\in U_i,w\in E(\La)}d_{z_i}(z,w) .
\]

\begin{proposition} \label{prop_whitney} 
Let $E(\La)$ be as in \eqref{def_E}. There exists a collection $\{U_{i}\}_{i\in\mathbb{N}}$ of cylinders defined as in \eqref{def_Ui_37} satisfying the following properties:
	\begin{enumerate}[label=(\roman*),series=theoremconditions]
		\item[(W1)] $\cup_{i\in\mathbb{N}}U_i=E(\La)^c$.
		\item[(W2)] $\tfrac{1}{8}U_i\cap\tfrac{1}{8}U_j= \emptyset$ for every $i,j\in \mathbb{N}$ with $i\ne j$.
		\item[(W3)] $3r_i\le d_i(U_i,E(\La))\le 4r_i$ for every $i\in \mathbb{N}$.
		\item[(W4)] $4U_i\subset E(\La)^c$ and $5U_i\cap E(\La)\ne\emptyset$ for every $i\in \mathbb{N}$.
		\item[(W5)]\label{v} $(12)^{-1}r_j\le r_i\le 12r_j$ for every $i\in \NN$, $j\in \mathcal{I}_i$.
		\item[(W6)] $2^{-\frac{1}{p}}\la_j\le \la_i\le 2^\frac{1}{p}\la_j$ for every $j\in \mathcal{I}_i$.
  \item[(W7)]\label{Uij}There exists a constant $c = c(n)$ such that $|U_i|\leq c|U_j|$, for every $i\in \mathbb{N}$ and $j\in \mathcal{I}_i$.
 \item[(W8)]\label{vii} $2U_j \subset 50U_i$ for every $i\in \NN$, $j\in \mathcal{I}_i$.
 
      \item[(W9)]\label{UUi} If $U_i=Q_i$, there holds $r_i^\alpha\la_i^q\le \tfrac{1}{800(1+[a]_\alpha)}\la_i^p$.

		\item[(W10)]\label{viii} If $U_i=G_{i}$, then 
			$\tfrac{a(z_i)}{2}\le a(z)\le 2a(z_i)$ for every $z\in 200 Q_{r_i}(z_i)$.
		\item[(W11)]\label{ix} For any $i\in\mathbb{N}$, the cardinality of $\mathcal{I}_i$, denoted by $|\mathcal{I}_i|$, is finite. Moreover, there exists a constant $c= c(n)$, such that $|\mathcal{I}_i|\le c$.
	\end{enumerate}
	Additionally, there exists a partition of unity $\{\om_i\}_{i\in\mathbb{N}}$ subordinate to the Whitney decomposition $\{2U_i\}_{i\in \mathbb{N}}$  with the following properties:
	\begin{enumerate}[resume*=theoremconditions]
		\item[(W12)]\label{x} $0\le \om_i\le 1$, $\om_i\in C_0^\infty(2U_i)$ for every $i\in\mathbb{N}$ and $\sum_{i\in\mathbb{N}}\om_i=1$ on $E(\La)^c$.
        \item[(W13)]\label{xi}There exists a constant $c=c(n)$ such that $\lVert\na \om_j\rVert_{\infty}\le cr_i^{-1}$,
   for every $i\in\mathbb{N}$ and $j\in \mathcal{I}_i$.
		\item[(W14)]\label{xii} There exists a constant $c=c(n)$ such that 
  \[
  \lVert \pa_t\om_j\rVert_{\infty}\le cs_i^{-1}
  \]
  for every $i\in\mathbb{N}$ and $j\in \mathcal{I}_i$.
	\end{enumerate}
\end{proposition}
The proof of Proposition~\ref{prop_whitney} is otherwise the same as in \cite[Proposition~3.2]{KKS}, but there is difference in the proofs of \ref{UUi} and \ref{viii}. Thus, we provide only the proofs of \ref{UUi} and \ref{viii}.
\begin{proof}[Proof of \ref{UUi}]
    Since  we have $Q_i\subset E(\La)^c$, there holds
    \[
    \la_i^{p\mu}<\La^\mu            < \fiint_{Q_i}   \Upsilon^\mu\,dz=\frac{\la_i^{p-2}}{2|B_1|r_i^{n+2}} \iint_{\mathbb{R}^{n+1}} \Upsilon^{\mu}\,dz.
    \]
    Multiplying  both sides by $\la_i^{-p\mu}r_i^{n+2}$ and then taking exponent $\tfrac{\alpha}{n+2}$, we get
    \[
    r_i^\alpha < K\la_i^{-\frac{2\alpha}{n+2}+\frac{\alpha p}{n+2} (1-\mu)}.
    \] 
    Note that by the choice of $\mu$ in \eqref{def_mu} and $\mfs$ in \eqref{def_mfs}, it follows that
    \[
        -\frac{2\alpha}{n+2}+\frac{\alpha p}{n+2} (1-\mu)
<-\frac{2\alpha}{n+2}+\frac{\alpha q}{n+2} (1-\mu)<p-q-\mfs
    \]
    and therefore,
    \[
       r_i^\alpha < K \la_i^{-\mfs}\la_i^{p-q}.
    \]
    Moreover, since we have $\theta \mfc<\theta\Pi<\La<2\la_i^p$ and \eqref{def_mfc}, it follows that
    \[
    r_i^\alpha < K \la_i^{-\mfs}\la_i^{p-q}\le K \left(  \frac{\theta \mfc}{2} \right)^{-\frac{\mfs}{p}}\la_i^{p-q}\le \frac{1}{800(1+[a]_\alpha)}\la_i^{p-q}.
    \]
\end{proof}
\begin{proof}[Proof of \ref{viii}]
   We claim that
   \[
   2[a]_\alpha (200r_i)^\alpha < a(z_i).
   \]
   Once the claim holds true, then we have
   \[
   2[a]_\alpha (200r_i)^\alpha < a(z_i)< \inf_{z\in 200 Q_{r_i}(z_i)} a(z)+ [a]_\alpha (200 r_i)^\alpha.
   \]
   and therefore,
   \[
   [a]_\alpha (200 r_i)^\alpha< \inf_{z\in 200 Q_{r_i}(z_i)} a(z)
   \]
   Moreover, we get
   \begin{align*}
       \begin{split}
           \sup_{z\in 200 Q_{r_i}(z_i)} a(z)\le  \inf_{z\in 200 Q_{r_i}(z_i)} a(z)+ [a]_\alpha (200 r_i)^\alpha< 2\inf_{z\in 200 Q_{r_i}(z_i)} a(z)
       \end{split}
   \end{align*}
   and thus, the proof is completed.
    To prove the claim, we argue as in the proof of \ref{UUi}. We have
    \[
    (a(z_i)\la_i^{q})^\mu <  \frac{\La\la_i^{-2}}{2|B_1|r_i^{n+2}}\iint_{\mathbb{R}^{n+1}}\Upsilon^\mu\,dz<\frac{a(z_i)\la_i^{q-2}}{|B_1|r_i^{n+2}}\iint_{\mathbb{R}^{n+1}}\Upsilon^\mu\,dz,
    \]
    where to obtain the last inequality, we used the fact that $\la_i^p<a(z_i)\la_i^q$. Multiplying both sides by $(a(z_i)\la_i^{q})^{-\mu} r_i^{n+2}$ and then taking exponent $\tfrac{\alpha}{n+2}$, we have
        \[
    r_i^\alpha <  \la_i^{-\frac{2\alpha}{n+2}+\frac{\alpha q}{n+2}(1-\mu)} \left(\frac{a(z_i)^{1-\mu}}{|B_1|} \iint_{\mathbb{R}^{n+1}} \Upsilon^\mu \,dz  \right)^\frac{\alpha}{n+2}
    <K \la_i^{-\mfs} \la_i^{p-q}.
    \]
    Again, since $\theta\mfc<\theta\Pi<\La\le 2a(z_i)\lai^q \le 2\|a\|_\infty \lai^q $ and \eqref{def_mfc} holds, we get
    \[
    r_i^\alpha < K\left(  \frac{\theta \mfc}{2\|a\|_\infty} \right)^{-\frac{\mfs}{q}}\lai^{p-q}<\frac{1}{800[a]_\alpha}\lai^{p-q}.
    \]

    Now, if the claim is false, then we have
    \[
    a(z_i)\le 2[a]_\alpha (200 r_i)^\alpha
    \]
    and this leads to a contradiction since the right hand side of the assumption $ \la_i^p<a(z_i)\la_i^q$ is estimated as
    \[
    a(z_i)\la_i^q\le 2[a]_\alpha (200 r_i)^\alpha\la_i^q \le 400[a]_\alpha  r_i^\alpha \la_i^q\le \frac{1}{2} \la_i^p.
    \]
    This completes the proof.
\end{proof}

\subsection{Lipschitz truncation} \label{sec_liptr}
Next, we define the Lipschitz truncated function that will be used as a test function in the proof of the Caccioppoli inequality. Let $0<h_0<\tfrac{\sig_2-\sig_1}{100}$ be small enough and $\eta\in C_0^\infty(B_{\rho_2}(x_0))$ and $\zeta\in C_0^\infty(\ell_{ (\sig_2-h_0 ) }(t_0))$ be cutoff functions satisfying $0\le \eta\le 1$, $0\le \zeta\le 1$, $\eta\equiv 1\ \text{in}\ B_{ (\rho_2+\rho_1)/2 }(x_0)$, $\zeta\equiv 1\ \text{in}\ \ell_{ ((\sig_1+\sig_2)/2)}(t_0)$ and
\begin{align}\label{42_1}
\lVert\na \eta\rVert_{\infty}\le \frac{3}{\rho_2-\rho_1},\quad
\lVert\pa_t\zeta\rVert_{\infty}\le \frac{3}{\sig_2-\sig_1}.
\end{align}
For $0<h<h_0$, let 
\[
    v_h = [u-u_0]_h\eta\zeta.
\]
We define the Lipschitz truncation of $\vh$ as
\begin{align}\label{t_421}
	\vhla(z)=\vh(z)-\sum_{i\in\mathbb{N}}(\vh(z)-v_h^i)\om_i(z),
\end{align}
where $z\in \RR^{n+1}$ and 
\begin{align*}
	v_h^i=
	\begin{cases}
		(\vh)_{2U_i},&\text{if}\  2U_i\subset \Q_2, \\
		0,&\text{otherwise}.
	\end{cases}
\end{align*}
Similarly, we denote
\begin{align}\label{t_421_2}
	v=(u-u_0)\eta\zeta 
 \quad\text{and}\quad
 \vla(z)=v(z)-\sum_{i\in\mathbb{N}}(v-v^i)\om_i,
\end{align}
where
\[
v^i=
	\begin{cases}
		(v)_{2U_i},&\text{if}\ 2U_i\subset \Q_2,  \\
		0,&\text{otherwise}.
	\end{cases}
\]

As $\vh$ vanishes outside $\Q_2$, the definition of $\vh^i$ implies that $\vhla\equiv 0$ in  $\Q_2^c$. Note as well that each term on the right-hand side of \eqref{t_421} belongs to $W^{1,2}(\ell_{\sig_2-h_0}(t_0); \allowbreak L^2(B_{\rho_2}(x_0),\RR^N))$ and by \ref{ix} only finitely many of these terms are nonzero on a neighborhood of any point. Thus 
\[
v_h^\La \in W_0^{1,2}(\ell_{\sig_2-h_0}(t_0);L^2(B_{\rho_2}(x_0),\RR^N)).
\]
We also note that from the properties in Proposition~\ref{prop_whitney} follows that
\begin{align} \label{t_325_1}
	v^\Lambda(z) =
	\begin{cases}
		\sum_{j\in \mathcal{I}_i} v^j\om_j(z), &\text{if}\ z\in 2U_i\ \text{for some}\ i\in \mathbb{N},\\
		v(z),&\text{if}\ z\in E(\La).  
	\end{cases}
\end{align}

\subsection{Lipschitz regularity}
In the following proposition we have collected properties that allow us to use the Lipschitz truncation as a test function. 

\begin{proposition} \label{prop_liptr1}
    Let $E(\La)$ be as in \eqref{def_E}. Then functions $\{v_h^\La \}_{0<h<h_0}$ and $v^\La$ defined in \eqref{t_421} and \eqref{t_421_2} have the following properties:
    \begin{enumerate}[label=(\roman*),series=theoremconditions]
    \item[(L1)]\label{p2-1} $v_h^\La\in W_0^{1,2}(\ell_{\sig_2-h_0}(t_0);L^2(   \B_2  ,\RR^N))\cap L^\infty(\ell_{\sig_2-h_0}(t_0);W^{1,\infty}_{0}(   \B_2  ,\RR^N))$. Moreover, there exist a constant $\co=\co(\data, \La, \rho_1,\rho_2,\sig_1,\sig_2)$ not depending on $h>0$ such that
    \[
    \sup_{z,w\in Q_{r,r^2(w_0)}}|v_h^\La(z)-v_h^\La(w)|\le \co r,
    \]
    for all   arbitrary   $r>0$ and $w_0 \in \RR^{n+1}$.
    \item[(L2)] $v^\La\in L^\infty(  \ti_2  (t_0);W^{1,\infty}_{0}(  \B_2  ,\RR^N))$.
    \item[(L3)] $v_h^\La \to v^\La$ in $L^\infty(  \Q_2  ,\RR^N)$ as $h\to0^+$, taking a subsequence if necessary.
    \item[(L4)]\label{p3-2} $\na v_h^\La\to \na v^\La$ and $\pa_t v_h^\La\to \pa_t v^\La$ a.e. in $E(\La)^c$ as $h\to0^+$.
\end{enumerate}
\end{proposition}

The proofs of \ref{p2-1}-\ref{p3-2} are as in \cite[Proposition 4.11]{KKS}. The only difference in the calculations comes from $\tfrac{|u-u_0|}{\rho}$ in \eqref{def_g}. Indeed, $|u-u_0|$ is used to define $g$ in \cite{KKS}. Note that in the proof of the Caccioppoli inequality for very weak solutions, all constant dependencies remain until the end. For this reason, we provide finer estimates compared to \cite{KKS} on the bad set $E(\La)^c$.

\subsection{Estimates on the bad set}
As we already mentioned in the previous subsection, we prove estimates for terms including $\vla$ that are needed in the proof of the Caccioppoli inequality. We will divide into cases by considering whether $\pa_t\zeta\equiv 0$ holds or not. Recall that $\zeta\equiv 1$ in $(t_0-\tfrac{\sig_1+\sig_2}{2}, t_0+\tfrac{\sig_1+\sig_2}{2} )$.
We set
\begin{align*}
    \begin{split}
        \mc_1&=\mathbb{R}^n\times \left(  t_0-\frac{4\sig_1+\sig_2}{5},t_0+ \frac{4\sig_1+\sig_2}{5}\right),  \\
        \mc_2&=\mathbb{R}^n\times \left(  t_0-\frac{3\sig_1+2\sig_2}{5},t_0+ \frac{3\sig_1+2\sig_2}{5}\right)
    \end{split}
\end{align*}
and among $\{U_i\}_{i\in\mathbb{N}}$, we consider
\begin{align*}
    \begin{split}
        \Theta&=\{i\in \mathbb{N}: 50U_i\cap \mc_1\ne\emptyset \},\\
        \Theta_1&= \{i\in \Theta: 50 U_i\subset \mc_2 \},\\
        \Theta_2&= \Theta\setminus \Theta_1.
    \end{split}
\end{align*}
Note that if $i\in \Theta_2$, then we have
\begin{align}\label{The2_ineq}
    100 s_i> \frac{\sig_2-\sig_1}{5}.
\end{align}
Otherwise, $i\in \Theta_1$ holds. Moreover, with the notions in \eqref{def_mathcal} we have following inclusions
\[
\mathbb{R}^n\times \ti_1 \subset \mc_1 \subset \mc_2    \subset \mathbb{R}^n\times \ti_2.
\]

We start with certain auxiliary estimates involving the very weak solution $u$.
\begin{proposition} \label{prop_3_4}
     There exists a constant $c=c(\data)$ such that
    \begin{enumerate}[label=(\roman*),series=theoremconditions]
\item\label{prop2_subintrinsic} $\begin{aligned}[t]
    \fiint_{200 U_i} \g^s\,dz&\le \lai^s, \quad\text{for any $s\in[1,q-1]$},
    \end{aligned}$
\item\label{prop2_est_nau_p-1} $\begin{aligned}[t]
    \fiint_{50 U_i}(\g^{p-1}+a\g^{q-1})\, dz \leq c \la_i^{-1}\La,
    \end{aligned}$
\item\label{prop2_poincare_u} $\begin{aligned}[t]
  \fiint_{50U_i}|u-(u)_{50U_i}|\,dz \leq cr_i\la.
    \end{aligned}$
\end{enumerate}
\end{proposition}
\begin{proof}
    The proof of \ref{prop2_subintrinsic} is as in \cite[Lemma 4.2.]{KKS}. To show \ref{prop2_est_nau_p-1} we assume first that $U_i =Q_i$. By \eqref{eq_holder_reg}, \ref{prop2_subintrinsic} and \ref{UUi}, we have
    \begin{align*}
    \begin{split}
        &\fiint_{50 Q_i}(\g^{p-1}+a(z)\g^{q-1})\, dz \\
        &\leq \fiint_{50 Q_i} \g^{p-1}\,dz+ a(z_i)\fiint_{50 Q_i} \g^{q-1}\,dz+[a]_\alpha (50 r_i)^\al \fiint_{ 50 Q_i} \g^{q-1}\,dz\\
        &\leq \lai^{p-1}+ a(z_i)\lai^{q-1}+cr_i^\al\lai^{q-1}\\
        &\leq c\lai^{p-1}+a(z_i)\lai^{q-1}\\
        &= c\lai^{-1}\La.
    \end{split}
    \end{align*}
    Then assume that $U_i = G_i$. By \ref{prop2_subintrinsic} and \ref{viii} we have
    \[
    \fiint_{50 G_i}(\g^{p-1}+a(z)\g^{q-1})\, dz \leq c\lai^{p-1}+2a(z_i)\fiint_{50 G_i}\g^{q-1}\, dz\leq c\lai^{-1}\La.
    \]

To show \ref{prop2_poincare_u}, observe that by Lemma~\ref{lem_parabolic_poincare}, \ref{prop2_subintrinsic} and \ref{prop2_est_nau_p-1} we have
    \begin{align*}
        \begin{split}
			&\fiint_{50 U_i}|u-(u)_{50 U_i}|\,dz\\
            &\le cr_i\left(\fiint_{50 U_i}|\na u|\,dz+\left(\frac{s_i}{r_i^2}\fiint_{50 U_i}(|\na u|^{p-1}+a(z)|\na u|^{q-1})\,dz\right)\right)\\
            &\qquad +cr_i\left(\frac{s_i}{r_i^2}\fiint_{50 U_i}(|F|^{p-1}+a(z)|F|^{q-1})\,dz\right)\\
            &\leq cr_i\lai.
		\end{split}
    \end{align*}
The proof is completed.
\end{proof}

\begin{proposition}
There exists $c=c(\data)$ such that the following hold.
    \begin{enumerate}
    \item[(P1)]\label{v_poincare}Suppose $50 B_i\subset \B_2$ and $i\in \Theta_1$, then
    \begin{align*}
        \begin{split}
            \fiint_{50 U_i} |v-(v)_{50 U_i}|\,dz
            &\le c\left( \frac{\rho}{\rho_2-\rho_1} \right) r_i \fiint_{50 U_i} g\,dz\\
            &\qquad+c \left( \frac{\rho}{\rho_2-\rho_1} \right) s_ir_i^{-1}\fiint_{50 U_i} (g^{p-1}+ag^{q-1} )\,dz.
        \end{split}
    \end{align*}
    \item[(P2)]\label{prop2_poincare_internal} 
Suppose $50 B_i\subset \B_2$ and $i\in \Theta_1$, then
    \[
    \fiint_{50 U_i}|v-(v)_{50 U_i}|\,dz \leq c\left( \frac{\rho}{\rho_2-\rho_1} \right) r_i\lai.
    \]
    \item[(P3)]\label{prop2_poincare_boundary} 
    Suppose $50 B_i\cap \B_2^c\ne\emptyset$, then
    \[
    \fiint_{50 U_i}|v|\,dz \leq c\left( \frac{\rho}{\rho_2-\rho_1} \right) r_i\lai.
    \]
\end{enumerate}
\end{proposition}

The proofs of \ref{prop2_poincare_internal} and \ref{prop2_poincare_boundary} follow as in \cite[Lemma 4.5.]{KKS} by applying \ref{v_poincare} and \ref{prop2_subintrinsic} and \ref{prop2_est_nau_p-1} in Proposition~\ref{prop_3_4}. Below we provide the proof of \ref{v_poincare}.

\begin{proof}[Proof of \ref{v_poincare}] 
    Let $\tau_1, \tau_2 \in 50 I_i$ be arbitrary with $\tau_1<\tau_2$ and $\zeta_\delta\in W^{1,\infty}_0(50 I_i)$ be defined as
\begin{align*}
	\zeta_{\delta}(t)=
	\begin{cases}
	    0, &  t \in (-\infty, \tau_1-\delta),\\
		1+\frac{t-\tau_1}{\delta},& t\in [\tau_1-\delta, \tau_1]\\
		1,&t\in (\tau_1,\tau_2),\\
		1-\frac{t-\tau_2}{\delta},&t \in [\tau_2,\tau_2+\delta],\\
		0,& t \in (\tau_2+\delta, \infty).
	\end{cases}
	\end{align*}
Additionally, let $\vp\in C_0^\infty(50 B_i)$ be a nonnegative function such that
    \[
        \fint_{50 B_i} \vp \,dx = 1, \quad \lVert\na\vp\rVert_\infty \le \frac{c(n)}{r_i}, \quad \lVert\vp\rVert_\infty \le c(n). 
    \]
    Recalling $\eta$ and $\zeta$ are defined in \eqref{42_1}, we take $\psi = \vp\eta\zeta \zeta_\delta\in W_0^{1,\infty}(50 U_i)$ as a test function to weak formulation of \eqref{11}. Then,
    \[
        	\iint_{50U_i} -(u-u_0) \cdot\vp\eta\pa_t(\zeta \zeta_\delta) \,dz = \iint_{50U_i} \left(-\mathcal{A}(z,\na u)+|F|^{p-2}F+a|F|^{q-2}F\right) \cdot\na\psi\,dz
    \]
    and we get 
    \begin{align*}
        \begin{split}
            \mathrm{I} 
            &= \left| \iint_{50U_i}-(u-u_0)\cdot\eta \vp \zeta \pa_t\zeta_\delta\,dz\right|\\
            &\leq \left|\iint_{50U_i}(u-u_0)\cdot\eta \vp \zeta_\delta\pa_t\zeta \,dz\right| \\
            &\qquad+(1+L_2) \iint_{50U_i}\left(|\na u|^{p-1}+a|\na u|^{q-1}+|F|^{p-1}+a|F|^{q-1}\right) |\na\psi| \,dz \\
            &\leq \left| \iint_{50U_i}-(u-u_0) \cdot\eta \vp \zeta_\delta\pa_t\zeta \,dz\right|
            +c\left|\iint_{50U_i}\left( g^{p-1}+ag^{q-1}\right) |\na\psi| \,dz\right| \\
            &= \mathrm{II}+\mathrm{III}.
        \end{split}
    \end{align*}
 Note that since we have $i\in \Theta_1$, $\zeta\equiv1$ in $50 U_i$. Therefore, we get $\mathrm{II} =0.$
    For $\mathrm{III}$, we get
    \begin{align*}
        \begin{split}
            \mathrm{III}&\leq c\iint_{50U_i} (g^{p-1}+ag^{q-1} ) |\vp\na\eta+\eta\na\vp| \,dz \\
            & \leq c \iint_{50U_i} (g^{p-1}+ag^{q-1}) (\vp+\eta)(|\na\vp|+|\na\eta|) \,dz\\
            &\leq cs_ir_i^n\left(\frac{1}{\rho_2-\rho_1}+\frac{1}{r_i}\right) \fiint_{50U_i} (g^{p-1}+ag^{q-1})\,dz.
        \end{split}
    \end{align*}
    Moreover, since we have $50B_i\subset \B_2$, we have $cr_i\le \rho$ for some $c=c(\data)$ and thus,
    \[
    \mathrm{III}\le c\left(\frac{\rho}{\rho_2-\rho_1}\right)s_ir_i^{n-1}\fiint_{50U_i} (g^{p-1}+ag^{q-1})\,dz.
    \]
    Meanwhile, by the one-dimensional Lebesgue differentiation theorem, we have
\begin{align*}
\begin{split}
    \lim_{\delta \to 0^+} \mathrm{I}
    &=\left|\lim_{\delta \to 0^+}\iint_{50U_i}(u-u_0)\eta \vp \zeta \pa_t\zeta_\delta\,dz \right|\\
    &=\left| \int_{50 B_i\times \{\tau_1\}} v\vp\,dz - \int_{50 B_i\times \{\tau_2\}} v\vp\,dz \right|\\
    &=r_i^n\left|(v\vp)_{50 B_i}(\tau_1)-(v\vp)_{50 B_i}(\tau_2)\right|
    \end{split}
\end{align*}
and thus, we get
\begin{equation} \label{t_320_1}
\begin{split}
    &\esssup_{\tau_1,\tau_2 \in 50 I_i}|(v\vp)_{B_i}(\tau_1)-(v\vp)_{B_i}(\tau_2)| \\
    &\leq c\left(\frac{\rho}{\rho_2-\rho_1}\right) s_ir_i^{-1} \fiint_{50 Q_i} (g^{p-1}+ag^{q-1})\,dz .
\end{split}
\end{equation}

To complete the proof we estimate the left-hand side. We use the Poincar\'e inequality in the spatial direction. Then,
	\begin{align*} 
	\begin{split}
		\fiint_{50U_i}|v-(v)_{50U_i}|\,dz
		&\le \fiint_{50U_i}|v-(v)_{50B_i}(t)|\,dz+ \fiint_{50U_i}|(v)_{50B_i}(t)-(v)_{50U_i}|\,dz\\
		&\le c r_i\fiint_{50U_i}|\na v|\,dz+\fiint_{50U_i}|(v)_{50B_i}-(v)_{50U_i}|\,dz.
		\end{split}
	\end{align*}
 By using \eqref{42_1}, the first term on the right hand side is estimated as 
	\begin{align}\label{426_5}
			\begin{split}
			    \fiint_{50U_i}|\na v|\,dz 
                &= \fiint_{50U_i}\left|\na u \eta\zeta+(u-\uu)\na \eta \zeta\right| \,dz \\
                &\le c\left( \frac{\rho}{\rho_2-\rho_1}\right)\fiint_{50U_i}g\,dz.
			\end{split}
	\end{align}
For the second term on the right-hand side, we have
\begin{align*}
    \begin{split}
        &\fiint_{50U_i}|(v)_{50B_i}(t)-(v)_{50U_i}|\,dz\\
        &= \fint_{50I_i}\left|\fint_{50I_i}(v)_{50B_i}(t)-(v)_{50B_i}(\tau)\,dt\right|\,d\tau\\
        &\leq \fint_{50I_i}2|(v)_{50B_i}(t)-(v\vp)_{50B_i}(t)|\,dt + \esssup_{t,\tau\in 50I_i}|(v\vp)_{50B_i}(t)-(v\vp)_{50B_i}(\tau)|.
    \end{split}
\end{align*}
The second term above can be estimated by using \eqref{t_320_1}. In order to estimate the remaining term, we recall that $(\vp)_{B_r}=1$. The Poincar\'e inequality in the spatial direction and \eqref{426_5} give
\begin{align*} 
   \begin{split}
    \fint_{50I_i}|(v)_{50B_i}(t)-(v\vp)_{50B_i}(t)|\,dt
       &=\fint_{50I_i}\left|\fint_{50B_i\times \{t\}} v\vp\,dx - (v)_{50B_i}(t)\fint_{50B_i}\vp\,dx \right|\,dt\\
       &=\fint_{50I_i}\left|\fint_{50B_i\times \{t\}}\vp(v-(v)_{50B_i})\,dx \right|\,dt\\
        &\leq \lVert\vp\rVert_\infty \fint_{50I_i}\fint_{50B_i}|v-(v)_{50B_i}|\,dx\,dt\\
        &\leq c\left( \frac{\rho}{\rho_2-\rho_1}\right)r_i\fiint_{50U_i}g\,dz.
    \end{split}
\end{align*}   
\end{proof}

Next, we provide estimates for $|\vla|$ and its derivatives.
\begin{lemma}\label{pre_The2}
    There exists $c=c(\data)$ such that for any $z\in E(\La)^c$, we have
    \[
      |\vla(z)|\le c\rho\lai.
    \]
\end{lemma}

\begin{proof}
    Since $z\in E(\La)^c$, there exists $i\in\mathbb{N}$ such that $z\in U_i$. Moreover, we have by \eqref{t_325_1} that 
    \[
    \vla(z)=\sum_{j\in\mathcal{I}_i}v^j\om_j(z).
    \]
    since $|\om_j|\le 1$ and the cardinality of $\mathcal{I}_i$ is bounded by \ref{ix}, we are enough to estimate $|v^j|$ when $2U_j\subset \Q_2$. Recalling $v=(u-u_0)\eta\zeta$ with $|\eta|,|\zeta|\le 1$, it follows that
    \[
    |v^j|\le \fiint_{2U_j} |v|\,dz \le \fiint_{2U_j} |u-u_0|\,dz \le c\rho\fiint_{50U_i} g\,dz\le c\rho\lai,
    \]
    where we used \ref{vii} and Proposition~\ref{prop_3_4}\ref{prop2_subintrinsic}.
    
\end{proof}

Next, we state estimates of $|\na \vla|$ and $|\pa_t \vla|$ for special cases. Estimates for the remaining case will be presented in Lemma~\ref{lem_ts_2}.

\begin{lemma}\label{lem_410}
Suppose $z\in U_i$ and $i\in \Theta_1$. Then, there exists $c=c(\data)$ such that the following hold.
	\begin{align}  \label{t_382}
  |\na \vla(z)|\le c\left( \frac{\rho}{\rho_2-\rho_1} \right)\la_i,\quad |\pa_t\vla(z)|\le c\left( \frac{\rho}{\rho_2-\rho_1} \right) s_i^{-1}r_i\lai.
	\end{align}
\end{lemma}

\begin{proof}
By \eqref{t_325_1} we have 
\[
    \na\vla(z)=\na\biggl(\sum_{j\in \mathcal{I}_i}v^j\om_j(z)\biggr) = \sum_{j\in \mathcal{I}_i}v^j\na\om_j(z).
\]
 Since 
\[
0 = \na\biggl(\sum_{j\in \mathbb{N}}\om_j(z)\biggr)= \na\biggl(\sum_{j\in \mathcal{I}_i}\om_j(z)\biggr)=\sum_{j\in \mathcal{I}_i}\na\om_j
\]
 by \ref{x}, we get
	\[
		\na\vla(z)=\sum_{j\in \mathcal{I}_i}v^j\na\om_j(z)=\sum_{j\in \mathcal{I}_i}v^j\na\om_j(z)-v^i\sum_{j\in \mathcal{I}_i}\na\om_j(z)=\sum_{j\in \mathcal{I}_i}(v^j-v^i)\na\om_j(z).
	\]
Since the cardinality of $|\mathcal{I}_i|$ is finite and $\|\na \om_j\|_{\infty}\le cr_i^{-1}$ for some $c=c(\data)$, it is enough to estimate $|v^j-v^i|$. 
If $2U_i \not\subset \Q_2$, we have by $i\in \Theta_1$ that $50B_i\cap \B_2^c\ne\emptyset$. Therefore, as $2U_j\subset 50U_i$ by \ref{vii}, it follows from \ref{prop2_poincare_boundary} that
\[
|v^j-v^i| = |v^j| \leq c\fiint_{50U_i}|v|\,dz \leq c\left( \frac{\rho}{\rho_2-\rho_1} \right) r_i\lai.
\]
On the other hand, if $2U_i \subset \Q_2$ we have by \ref{vii} that
\[
|v^j-v^i|
        \le |v^j-(v)_{50U_i}|+|(v)_{50U_i}-v^i|\le c\fiint_{50U_i} |v-(v)_{50U_i}|\,dz,
\]
where to obtain the last inequality, we applied the inequality $|U_j|\approx|U_i|$ by \ref{Uij}. 
Finally, applying \ref{prop2_poincare_internal} the first inequality in \eqref{t_382} follows. 

To show the second inequality, we follow the same argument as above and get
\[
    \pa_t\vla(z)=\sum_{j\in \mathcal{I}}\vh^j\pa_t\om_j(z)=\sum_{j\in \mathcal{I}}(v^j-v^i)\pa_t\om_j(z)
\]
and therefore, the conclusion follows along with \ref{xii}.
\end{proof}

Next, we provide estimates in the case $i\in\Theta_2$. In this case we use the following lemma.
\begin{lemma}\label{lem_theta_2}
	Suppose $i\in \Theta_2$ and $50U_i\cap \Q_2\ne\emptyset$. Then there exists $c=c(\data)$ such that 
    \[
    \rho_2-\rho_1 \leq c\theta^{\frac
    {2-q}{p}}r_i.
    \]
\end{lemma}
\begin{proof}
For the proof, we claim
\begin{align}\label{eq_pi_lai}
    \pi\leq c\theta^{-\frac{1}{p}}\lai,
\end{align}
where $c = c(\data)$. We divide into cases.

\textit{Case $1$}. First we consider the case that $\mathcal{Q}_2$ is $p$-intrinsic. Recalling the extension of $a$ in \eqref{extended_a}, there exists $w_i$ in the closure of $\mathcal{Q}_2$ such that $a(w_i)=a(z_i)$. Therefore, it follows from \eqref{eq_holder_reg} that
\[
|a(z_0)-a(z_i)|\le [a]_\alpha (4\rho)^\alpha.
\]
Since we have
\[
\theta(\pi^p+a(z_0)\pi^q)=\theta\Pi<\La=\la_i^p+a(z_i)\la_i^q,
\]
it follows that
\[
\pi^p+a(z_0)\pi^q<  \theta^{-1}(\la_i^p+a(z_0)\la_i^q+[a]_\alpha\rho^\alpha\lambda_i^q).
\]
On the other hand, we employ \eqref{eq_rho_decay} to have
\[
   \frac{1}{2}\pi^p+c^{-1}\rho^\alpha\pi^q \le \frac{1}{2}\pi^p+\frac{1}{2}\pi^p=\pi^p,
\]
where $c=c(\data)$.
Recall $\theta\in(0,1)$.
Estimating the left hand side of the previous inequality by using the above inequality, we get
\[
\pi^p+a(z_0)\pi^q+\rho^\al\pi^q\leq c\theta^{-1}(\lai^p+a(z_0)\lai^q+\rho^\alpha\lai^q).
\]
This leads to \eqref{eq_pi_lai}.

\textit{Case $2$}. Next, we consider the case when $\mathcal{Q}_2$ is $(p,q)$-intrinsic. Again, we recall that there exists $w_i$ as in the previous case. Moreover, we deduce from \eqref{eq_pq2} that
\[
\frac{a(z_0)}{2}\le a(z_i)=a(w_i) \le 2a(z_0).
\]
Therefore, we obtain
\[
\theta(\pi^p+a(z_0)\pi^q)=\theta\Pi<\La=\la_i^p+a(z_i)\la_i^q\le 2(\lai^p+a(z_0)\lai^q ).
\]
Thus, we get $ \pi\le 2^\frac{1}{p}\theta^{-\frac{1}{p}}\lai $ and \eqref{eq_pi_lai} follows.

Finally, to prove the lemma, assume for contradiction that
\[
\rho_2-\rho_1>100 c^\frac{q-2}{2}\theta^\frac{2-q}{2p}r_i,
\]
where $c$ is in \eqref{eq_pi_lai}. We consider two cases. In the above inequality, we have multiplied $\tfrac{1}{2}$ in the exponent of $\theta$ to make calculations below simpler.
When $\Q_2 $ is $p$-intrinsic, we have by \eqref{eq_pi_lai} and the counterassumption that
\[
\sig_2-\sig_1 = (\rho_2^2-\rho_1^2) \pi^{2-p}  > c^{2-p} \theta^\frac{p-2}{p}(\rho_2-\rho_1)^2 \lai^{2-p}\geq 100^2 s_i,
\]
which is a contradiction with $i\in\Theta_2$ as \eqref{The2_ineq} is not valid. Thus, the estimate in the lemma holds.
On the other hand, if $\Q_2$ is $(p,q)$-intrinsic, we again have
\[
\frac{a(z_0)}{2}\le a(z_i) \le 2a(z_0).
\]
Note that by \eqref{eq_pq2} and \ref{viii}, $a(\cdot)$ is comparable in $\Q_2$ and $200U_i$. Since we have $200 U_i\cap \Q_2\ne \emptyset$, it follows from the counterassumption that
\begin{align*}
\begin{split}
    \sig_2-\sig_1&=(\rho_2^2-\rho_1^2)\frac{\pi^2}{\pi^p+a(z_0)\pi^q}\\
    & > c^{2-q}\theta^\frac{q-2}{p} (\rho_2-\rho_1)^2 \frac{\lai^2}{\lai^p+2a(z_i)\lai^q}\\
    &\geq 50^2 s_i,
\end{split}
\end{align*}
where to obtain the last inequality, we have used
\[
\frac{\lai^2}{\lai^p+2a(z_i)\lai^q}\ge \frac{\lai^2}{2\La}
\]
and if $U_i$ is $p$-intrinsic, 
\[
\frac{\lai^2}{\lai^p+2a(z_i)\lai^q}\ge \frac{1}{4}\lai^{2-p}.
\]
Again, this is a contradiction and the proof is completed.
\end{proof}

\begin{lemma}\label{lem_ts_2}
    Let $z\in U_i$ and $i\in\Theta_2$. There exists $c=c(\data)$ such that
    \[
    |\vla|(z) \le c\left( \frac{\rho}{\rho_2-\rho_1} \right) \theta^\frac{2-q}{p}r_i\lai
    \]
    and
    \[
    |\na \vla|(z)\le c\left( \frac{\rho}{\rho_2-\rho_1} \right) \theta^\frac{2-q}{p}\lai.
    \]
\end{lemma}
\begin{proof}
For the first inequality, we modify the proof of Lemma~\ref{pre_The2}.
We have
\[
|\vla|(z) = \biggl|\sum_{j\in \mathcal{I}_i}\omega_j(z)v^j \biggr|\le \sum_{j\in \mathcal{I}_i}|v^j|.
\]
We may assume that $50U_j\cap \Q_2\ne\emptyset$, as otherwise $v^j=0$ since $2U_j\subset 50 U_j$. As in the proof of Lemma~\ref{pre_The2}, we get
\[
|v^j|\le c\rho\la_i
\]
for some $c=c(\data)$. Now, we apply Lemma~\ref{lem_theta_2} to have
\[
|v^j|\le c\rho\la_i=c \left( \frac{\rho}{\rho_2-\rho_1} \right)(\rho_2-\rho_1) \la_i\le c \left( \frac{\rho}{\rho_2-\rho_1} \right) \theta^\frac{2-q}{p}r_j\lai.
\]
Finally, using \ref{v}, we get
\[
|v^j|\le c \left( \frac{\rho}{\rho_2-\rho_1} \right) \theta^\frac{2-q}{p}r_i\lai.
\]
The conclusion follows from the fact that the cardinality of $\mathcal{I}_i$ is finite by \ref{xi}. For the second inequality, note that 
\[
|\na \vla|(z)=  \biggl|\sum_{j\in \mathcal{I}_i} \na\om_j(z) v^j \biggr| .
\]
Therefore, using the same argument and \ref{xi}, we get the conclusion.
\end{proof}

We close this subsection with the estimate.

\begin{lemma}\label{time_deri_est}
    There exists a constant $c=c(\data)$ such that
    \begin{align*}
            \iint_{\mc_1} | \pa_t\vla\cdot (\vla-v) |\,dz
            \le c \left( \frac{\rho}{\rho_2-\rho_1} \right) \La |E(\La)^c|+ c\left(\frac{\rho}{\rho_2-\rho_1}\right)^2\iint_{ \Q_2 } \frac{|v |^2}{\sig_2}\,dz.
    \end{align*}
\end{lemma}

\begin{proof}
    Recall $\vla=v$ in $E(\La)$ and $v^\La$ is supported in $\Q_2$.
    Since we have $E(\La)^c=\cup_{i\in\mathbb{N}} U_i$, there holds
    \begin{align*}
        \begin{split}
            \iint_{\mc_1} | \pa_t\vla\cdot (\vla-v) |\,dz
            &=\iint_{\mc_1\cap \Q_2\cap E(\La)^c } | \pa_t\vla\cdot (\vla-v) |\,dz\\
            &\le \sum_{i\in \Theta}\iint_{\mc_1\cap \Q_2\cap U_i } | \pa_t\vla\cdot (\vla-v) |\,dz\\
            &\le \sum_{i\in \Theta_1}\iint_{ \mc_1\cap \Q_2\cap U_i  } | \pa_t\vla\cdot (\vla-v) |\,dz\\
            &\qquad+ \sum_{i\in \Theta_2}\iint_{ \mc_1\cap \Q_2\cap U_i  } | \pa_t\vla\cdot (\vla-v) |\,dz.
        \end{split}
    \end{align*}
    By using the definition of $\vla$ in \eqref{t_421_2}, for the case $i\in \Theta_1$ we apply \eqref{t_382} to have
    \begin{align*}
        \begin{split}
            \iint_{ U_i } | \pa_t\vla\cdot (\vla-v) |\,dz
            &\le\sum_{j\in \mathcal{I}_i }\iint_{ U_i } | \pa_t\vla| |v-v^j |\,dz\\
            &\le c\left( \frac{\rho}{\rho_2-\rho_1} \right) s_i^{-1}r_i\lai \sum_{j\in \mathcal{I}_i }\iint_{ U_i }  |v-v^j |\,dz.
        \end{split}
    \end{align*}
If $2U_j\not\subset \Q_2$, we have by \ref{vii} that $50U_i\cap \Q_2 \neq \emptyset$. As $i \in \Theta_1$, it follows that $50B_i\cap \B_2^c\ne\emptyset$. Therefore, we get from \ref{prop2_poincare_boundary}, along with \ref{ix}, that
\begin{align*}
\begin{split}
     \iint_{ U_i } | \pa_t\vla\cdot (\vla-v) |\,dz &\leq c\left( \frac{\rho}{\rho_2-\rho_1} \right) s_i^{-1}r_i\lai \iint_{ 50 U_i } |v|\,dz\\
     &\leq c\left( \frac{\rho}{\rho_2-\rho_1} \right)^2 s_i^{-1}r_i^2\lai^2|U_i|.
\end{split}
\end{align*}
On the other hand, if $2U_j\subset \Q_2$, using adding and subtracting $(v)_{50U_i}$ and then \ref{vii} and \ref{ix}, it follows that
\begin{align*}
    \begin{split}
        \iint_{ U_i } | \pa_t\vla\cdot (\vla-v) |\,dz
        &\le c \left( \frac{\rho}{\rho_2-\rho_1} \right) s_i^{-1}r_i\lai \iint_{ 50U_i }  |v-(v)_{50U_i} |\,dz.
    \end{split}
\end{align*}
Employing \ref{prop2_poincare_internal}, we conclude that
\begin{align*}
    \begin{split}
        \sum_{i\in \Theta_1}\iint_{ U_i } | \pa_t\vla\cdot (\vla-v) |\,dz
        &\le c \left( \frac{\rho}{\rho_2-\rho_1} \right)\sum_{i\in \Theta_1} s_i^{-1}r_i^2\lai^2 | U_i|\\
        &\le c \left( \frac{\rho}{\rho_2-\rho_1} \right)\sum_{i\in \Theta_1} \La | U_i|.
    \end{split}
\end{align*}
Since $|U_i|\le 8^{n+2}|\tfrac{1}{8}U_i|$ and $\tfrac{1}{8}U_i$ are disjoint from each other, we get
\[
\sum_{i\in \Theta_1}\iint_{ U_i } | \pa_t\vla\cdot (\vla-v) |\,dz\le c\left( \frac{\rho}{\rho_2-\rho_1} \right)\La|E(\La)^c|.
\]

On the other hand, to estimate the term when $i\in \Theta_2$, note that for $z\in U_i$ we have by \ref{xii}
\[
|\pa_t\vla(z)|=\biggl| \sum_{j\in\mathcal{I}_i } v^j\pa_t\om_j (z) \biggr|\le  cs_i^{-1} \sum_{j\in\mathcal{I}_i } |v^j|\le cs_i^{-1} \sum_{j\in\mathcal{I}_i } \fiint_{2U_j} |v|\,dz.
\]
Similarly, we use \ref{Uij} to have $|U_i|\approx |U_j|$ and therefore
\begin{align*}
    \begin{split}
        \iint_{ U_i } | \vla-v |\,dz
        &\le \sum_{j\in\mathcal{I}_i } \iint_{ U_i } | v-v^j |\,dz\\
        &\le c|U_i|\sum_{j\in\mathcal{I}_i} \left( \fiint_{ 2 U_i } |v |\,dz + \fiint_{ 2U_j } |v |\,dz  \right)\\
        &\le c|U_i|\left( \fiint_{ 2 U_i } |v |\,dz + \sum_{j\in\mathcal{I}_i}\fiint_{ 2U_j } |v |\,dz  \right).
    \end{split}
\end{align*}
Therefore, using the fact that the cardinality of $\mathcal{I}_i$ is uniformly bounded, $|U_i|\approx|U_j|$ and \eqref{The2_ineq}, we obtain
\begin{align*}
    \begin{split}
        &\sum_{i\in\Theta_2}\iint_{ U_i } | \pa_t\vla\cdot (\vla-v) |\,dz\\
        &\le c\sum_{i\in\Theta_2} s_i^{-1} |U_i|\left( \fiint_{ 2 U_i } |v |\,dz + \sum_{j\in\mathcal{I}_i}\fiint_{2 U_j } |v |\,dz  \right)^2\\
        &\le c\sum_{i\in\Theta_2} (\sig_2-\sig_1)^{-1} |U_i|\left( \fiint_{ 2 U_i } |v |^2\,dz + \sum_{j\in\mathcal{I}_i}\fiint_{2 U_j } |v |^2\,dz  \right)\\
        &\le c\sum_{i\in\mathbb{N}} (\sig_2-\sig_1)^{-1} \iint_{ 2 U_i } |v |^2\,dz\\
        &\le c\left(\frac{\sig_2}{\sig_2-\sig_1}\right)\iint_{ \mathbb{R}^{n+1} } \frac{|v |^2}{\sig_2}\,dz\\
        &=c\left(\frac{\sig_2}{\sig_2-\sig_1}\right)\iint_{ \Q_2 } \frac{|v |^2}{\sig_2} \,dz.
    \end{split}
\end{align*}
Moreover, canceling scaling factors in numerator and denominator of the fraction, the conclusion follows.
\end{proof}

\subsection{Time-slice estimate}
To prove the Caccioppoli inequality in the next subsection, we need the following estimate in time slices. Recall the notation $\Q_i=\B_i\times \ti_i$ in \eqref{def_mathcal}.

\begin{lemma}\label{lem_timeslice}
For a.e. $t\in \ti_1$ there holds
    \begin{align*}
    \begin{split}
        \int_{(\B_2\times\{t\})\cap E(\La)^c}(|v|^2-|\vla-v|^2)\,dx
        &\ge - c\theta^{\frac
    {2-q}{p}} \left ( \frac{\rho}{\rho_2-\rho_1} \right)^2 \La|E(\La)^c|\\
     &\qquad -c \left( \frac{\rho}{\rho_2-\rho_1}  \right)^2 \iint_{\Q_2}\frac{|u-u_0|^2}{\sig_2}\,dz.
    \end{split}
\end{align*}
\end{lemma}

Before proving the lemma, we state two auxiliary lemmas.

\begin{lemma}\label{lem_ts_3}
    Let $z\in U_i$ and $i\in\Theta_1$. For any $\varepsilon\in (0,1)$ there exists $c=c(\data)$ such that we have
    \[
    |\vla(z)|\leq cr_i\lai+|v^i|  \leq \frac{c}{\varepsilon}r_i\la_i+ \frac{\varepsilon}{r_i\la_i}|v^i|^2
    \]
    and
    \[
    |\na \vla(z)|\leq \frac{c}{\varepsilon}\la_i+ \frac{\varepsilon}{\la_i}|v^i|^2
    \]
\end{lemma}
\begin{proof}
From \ref{x} and the triangle inequality we get
\[
|\vla(z)| = \biggl|\sum_{j\in \mathcal{I}_i}\omega_j(z)v^j \biggr| \leq \sum_{j\in \mathcal{I}_i}\omega_j(z)|v^j-v^i|+|v^i|.
\]
With Young's inequality we obtain
\[
|v^i|=|v^i|(r_i\la_i)^\frac{1}{2}(r_i\la_i)^\frac{-1}{2}\leq c\frac{r_i\la_i}{\varepsilon} + \frac{\varepsilon}{r_i\la_i}|v^i|^2
\]
    and as in the proof of Lemma~\ref{lem_410}, we have
    \[
    \sum_{j\in \mathcal{I}_i}\omega_j(z)|v^j-v^i| \leq cr_i\la_i.
    \]
This proves the first inequality. The second inequality follows from the same argument and \ref{xi}.
\end{proof}

\begin{lemma}\label{lem_ts_theta1}
Suppose $i\in \Theta_1$. For any $\varepsilon > 0$ there exists $c=c(\data)$ such that 
\[
\int_{\B_2\times\{t\}}\om_i\vla\cdot(v-v^i) \,dx \leq  \frac{c}{\varepsilon}  \left(\frac{\rho}{\rho_2-\rho_1}\right)^2\La|U_i|+ \varepsilon|B_i||v^i|^2
\]
for a.e. $t\in \ti_1$.
\end{lemma}
\begin{proof} 
Throughout the proof, note that $\om_i(\cdot,\tau)>0$ implies $\zeta(\tau)=1$  since $i\in\Theta_1$. 
Also note that it is enough to show the claim for $i\in\Theta_1$ such that $t>t_i-4s_i$, as otherwise the left hand side of the claim is zero. By Proposition~\ref{prop_liptr1} we have that $\varphi(\cdot,t) = \eta\om_i\vhla(\cdot,t)\in W_0^{1,\infty}(\B_2)$ is an admissible test function in the Steklov averaged weak formulation \eqref{steklov_averaged_wf}. Taking $\varphi(\cdot,t)$ as a test function and integrating over $(t_i-4s_i,t)$ in time, we obtain for any constant vector $\xi\in \RR^N$ that
\begin{align*}
\begin{split}
        &\int_{t_i-4s_i}^t\int_{\B_2}\pa_t[u-\xi]_h\cdot\eta \om_i\vhla\,dx\,d\tau \\
        &= -\int_{t_i-4s_i}^t\int_{\B_2}[\mA(\cdot,\na u)]_h \cdot\na(\eta \om_i\vhla)\,dx\,d\tau\\
        &\qquad+\int_{t_i-4s_i}^t\int_{\B_2}[|F|^{p-2}F+a|F|^{q-2}F]_h \cdot\na(\eta \om_i\vhla)\,dx\,d\tau
\end{split}
\end{align*}
 Using the product rule, fundamental theorem of calculus and the fact that $\om_i(\cdot,t_i-4s_i)= 0$, we obtain
\begin{align*}
\begin{split}
    &\int_{t_i-4s_i}^t\int_{\B_2}\pa_t[u-\xi]_h \cdot\eta \om_i\vhla\,dx\,d\tau \\
    &= \int_{t_i-4s_i}^t\int_{\B_2}\pa_t([u-\xi]_h \cdot\eta \om_i\vhla)\,dx\,d\tau -\int_{t_i-4s_i}^t\int_{\B_2}[u-\xi]_h \cdot \pa_t(\eta \om_i\vhla)\,dx\,d\tau\\
    &= \int_{\B_2\times\{t\}} [u-\xi]_h\cdot\eta \om_i\vhla  \,dx - \int_{t_i-4s_i}^t\int_{\B_2}[u-\xi]_h \cdot \pa_t(\eta \om_i\vhla)\,dx\,d\tau. 
\end{split}
\end{align*}
We let $h\to 0^+$ to get 
\begin{align}\label{ts_518}
\begin{split}
    &\int_{\B_2\times\{t\}} \om_i\eta\vla \cdot (u-\xi)  \,dx \\
    &= -\int_{t_i-4s_i}^t\int_{\B_2}\mA(z,\na u) \cdot\na(\eta \om_i\vla)\,dx\,d\tau\\
    &\qquad+\int_{t_i-4s_i}^t\int_{\B_2}  (|F|^{p-2}F+a|F|^{q-2}F) \cdot\na(\eta \om_i\vla)\,dx\,d\tau\\
    &\qquad +\int_{t_i-4s_i}^t\int_{\B_2}(u-\xi)\cdot \pa_t(\eta \om_i\vla)\,dx\,d\tau,
\end{split}
\end{align}
where we denoted $z=(x,\tau)$ for simplicity.
We consider first the case when $2B_i\not\subset \B_2$. In this case $v^i=0$ and by choosing $\xi=u_0$ the previous display becomes
\begin{align}\label{ts_case1_1}
    \begin{split}
            &\int_{\B_2\times\{t\}}(\om_i\vla \cdot v) \,dx \\
            &= -\int_{t_i-4s_i}^t\int_{\B_2}\mA(z,\na u) \cdot\na(\eta \om_i\vla)\,dx\,d\tau\\
            &\qquad+\int_{t_i-4s_i}^t\int_{\B_2}  (|F|^{p-2}F+a|F|^{q-2}F) \cdot\na(\eta \om_i\vla)\,dx\,d\tau\\
    &\qquad +\int_{t_i-4s_i}^t\int_{\B_2}v\cdot \pa_t( \om_i\vla)\,dx\,d\tau,
    \end{split}
\end{align}
where the left hand side is the same as in the claim of the lemma.
We estimate the first term on the right hand side first.
Using the growth condition of $\mA$, we get that
\begin{align*}
\begin{split}
    &-\int_{t_i-4s_i}^t\int_{\B_2}\mA(z,\na u) \cdot\na(\eta \om_i\vla)\,dx\,d\tau \\
    &\leq \int_{t_i-4s_i}^t\int_{\B_2}|\mA(z,\na u)| |\na(\eta \om_i\vla)|\,dx\,d\tau\\
    &\leq c\iint_{\Q_2\cap 2U_i}(|\na u|^{p-1}+a(z)|\na u|^{q-1})|\na(\eta \om_i\vla)|\,dz.
\end{split}
\end{align*}
By the product rule we have in $U_i$ that
\begin{align*}
\begin{split}
    |\na(\eta \om_i\vla)| &\leq |\vla||\na \eta|+|\vla||\na\om_i|+|\na\vla| \\
    &\leq c\max\{r_i^{-1},(\rho_2-\rho_1)^{-1}\}|\vla|+|\na\vla|.
\end{split}
\end{align*}
Assume first that $\rho_2-\rho_1 \leq r_i$. Then we get by 
Lemma~\ref{pre_The2}, \eqref{t_382} and Proposition~\ref{prop_3_4}\ref{prop2_est_nau_p-1}
the desired estimate
\begin{align*}
	\begin{split}
		&\iint_{\Q_2\cap 2U_i}(|\na u|^{p-1}+a(z)|\na u|^{q-1})|\na(\eta \om_i\vla)|\,dz\\
        &\leq c (\rho_2-\rho_1)^{-1}\iint_{\Q_2\cap 2U_i}(|\na u|^{p-1}+a(z)|\na u|^{q-1})|\vla|\,dz\\
        &\qquad+ \iint_{\Q_2\cap 2U_i}(|\na u|^{p-1}+a(z)|\na u|^{q-1})|\na\vla|\,dz\\
		&\leq c\left(\frac{\rho}{\rho_2-\rho_1}\right)\lai\iint_{\Q_2\cap 2U_i}(|\na u|^{p-1}+a(z)|\na u|^{q-1})\,dz\\
		&\leq c\left(\frac{\rho}{\rho_2-\rho_1}\right)\La|U_i|.
	\end{split}
\end{align*}
Now assume that $r_i\leq \rho_2-\rho_1$. Then we have by 
Lemma~\ref{lem_ts_3} and Proposition~\ref{prop_3_4}\ref{prop2_est_nau_p-1}
that
\begin{align*}
    \begin{split}
        &\iint_{\Q_2\cap 2U_i}(|\na u|^{p-1}+a(z)|\na u|^{q-1})|\na(\eta \om_i\vla)|\,dz \\
        &\leq c r_i^{-1}\iint_{\Q_2\cap 2U_i}(|\na u|^{p-1}+a(z)|\na u|^{q-1})|\vla|\,dz\\
        &\qquad+ \iint_{\Q_2\cap 2U_i}(|\na u|^{p-1}+a(z)|\na u|^{q-1})|\na\vla|\,dz\\
        &\leq \left(\frac{c\lai}{\varepsilon}+\frac{\varepsilon}{r_i^2\lai}|v^i|^2\right)\lai^{-1}\La|U_i|,
    \end{split}
\end{align*}
for any $\varepsilon >0$.
Note that if $U_i=G_i$, we have $(\lai r_i)^{-2}\La|U_i|=2|B_i| $ and if $U_i = Q_i$, we have $(\lai r_i)^{-2}\La|U_i|\leq 2\lai^{p-2}r_i^{-2}|Q_i| \leq 4|B_i| $. Therefore we conclude from the above estimate that
\[
\iint_{\Q_2\cap 2U_i}(|\na u|^{p-1}+a(z)|\na u|^{q-1})|\na(\eta \om_i\vla)|\,dz
       \leq \frac{c}{\varepsilon}\La|U_i|
       +\varepsilon|B_i||v^i|^2.
\]
The same arguments holds for the second term on the right hand side of \eqref{ts_case1_1} and thus
\[
\iint_{\Q_2\cap 2U_i}(|F|^{p-1}+a(z)|F|^{q-1})|\na(\eta \om_i\vla)|\,dz
       \leq \frac{c}{\varepsilon}\left(\frac{\rho}{\rho_2-\rho_1}\right)\La|U_i|
       +\varepsilon|B_i||v^i|^2.
\]

Then we estimate the third term on the right hand side of \eqref{ts_case1_1}. By 
\ref{xii} and \eqref{t_382}
we have  in $U_i$ that
\begin{align}\label{ts_616}
     |\pa_t( \om_i\vla)| \leq  |\pa_t\om_i||\vla| + |\pa_t\vla| \leq \frac{c|\vla|}{s_i}+\frac{cr_i\lai}{s_i}.
\end{align} 
Therefore we have by 
\eqref{ts_616}, Lemma~\ref{lem_ts_3} and \ref{prop2_poincare_boundary}
that
\begin{align*}
    \begin{split}
       \iint_{\Q_2\cap 2U_i}|v||\pa_t( \om_i\vla)|\,dz 
       &\leq \left(\frac{c r_i\lai}{\varepsilon s_i}+\frac{\varepsilon|v^i|^2}{r_is_i\lai}\right)|U_i|\fiint_{2U_i}|v|\,dz\\
       &\leq \frac{c}{\varepsilon}\La|U_i| + \varepsilon|v^i|^2|B_i|.
    \end{split}
\end{align*}
This finishes the proof in the case $2B_i\not\subset \B_2$. 

Then we consider the case $2B_i\subset \B_2$. In this case the last term on the left hand side of \eqref{ts_518} is not exactly what we need. Therefore, we start by observing that by adding and subtracting we have
\begin{align}\label{ts_lem3_1}
\begin{split}
    &\int_{\B_2\times\{t\}}\om_i\vla\cdot(v-v^i) \,dx \\
    &= \int_{\B_2\times\{t\}}\om_i\vla\cdot((u-u_0)\eta-v^i)  \,dx \\
    & =\int_{\B_2\times\{t\}}\om_i\eta\vla\cdot(u-(u)_{50U_i})  \,dx \\
    &\qquad+  \int_{(\B_2\cap 2B_i) \times \{t\}}(((u)_{50U_i}-u_0)\eta-v^i)\cdot\om_i\vla  \,dx.
    \end{split}
\end{align}
The first term on the right hand side of \eqref{ts_lem3_1} is estimated as in the case $2B_i \not\subset \B_2$ by using \eqref{ts_518} with $\xi=(u)_{50U_i}$. The only difference is that we use Proposition~\ref{prop_3_4}\ref{prop2_poincare_u} instead of \ref{prop2_poincare_boundary} at the end of the argument.

It remains to estimate the second term on the right hand side of \eqref{ts_lem3_1}. For any $x\in \B_2\cap 2B_i$ we have
\begin{align}\label{eq_ts_lem3_3}
\begin{split}
    |((u)_{50U_i}-u_0)\eta(x)-v^i| 
    &\leq \left| ((u)_{50U_i}-u_0)\fint_{2B_i} ( \eta(y)-\eta(x)  ) \,dy \right| \\
    &\qquad + \left|\fint_{2B_i}  ( v^i-((u)_{50U_i}-u_0)\eta(y)  )  \,dy \right|.
\end{split}
\end{align}
For the first term in the right hand side of \eqref{eq_ts_lem3_3} we get with Proposition~\ref{prop_3_4}\ref{prop2_subintrinsic} that 
\begin{align*}
\begin{split}
    \left|((u)_{50U_i}-u_0)\fint_{2B_i}\eta(y)-\eta(x)\,dy \right| &\leq cr_i\sup_{y\in 2B_i}|\na \eta(y)| \fiint_{50U_i}|u-u_0|\,dz \\
    &\leq c\left(\frac{\rho}{\rho_2-\rho_1}\right)r_i\fiint_{50U_i}\g\,dz \\
    &\leq c\left(\frac{\rho}{\rho_2-\rho_1}\right)r_i\lai.
\end{split}
\end{align*}
For the second term in the right hand side of \eqref{eq_ts_lem3_3}, we obtain by Fubini's theorem and $v = (u-u_0)\eta$ in $2U_i$, that
\begin{align*}
    \begin{split}
        & \left|\fint_{2B_i} ( v^i-((u)_{50U_i}-u_0)\eta(y) ) \,dy \right| \\
        &= \left|\fiint_{2U_i}\left(v(z)-\fint_{2B_i}((u)_{50U_i}-u_0)\eta(y)\,dy\right)\,dz\right|\\
        &= \left|\fint_{2I_i}\left(\fint_{2B_i}v(y,\tau)\,dy-\fint_{2B_i}((u)_{50U_i}-u_0)\eta(y)\,dy\right)\,d\tau \right|\\
        &= \left|\fiint_{2U_i}(u-(u)_{50U_i})\eta\,dz \right|.
    \end{split}
\end{align*}
We have by Proposition~\ref{prop_3_4}\ref{prop2_poincare_u} that
\[
\left|\fiint_{2U_i}(u-(u)_{50U_i})\eta\,dz\right|\le \fiint_{2U_i}|u-(u)_{50U_i}|\,dz\le cr_i\lai.
\]

With this and Lemma~\ref{lem_ts_3} we get for the second term on the right hand side \eqref{ts_lem3_1} the desired estimate
\begin{align*}
\begin{split}
    &\int_{(\B_2\cap 2B_i) \times \{t\}}(((u)_{50U_i}-u_0)\eta-v^i)\cdot\om_i\vla \,dx \\
    &\leq c\left(\frac{\rho}{\rho_2-\rho_1}\right)r_i\lai\int_{2B_i}|\vla  |\,dx\\
    &\leq c\left(\frac{\rho}{\rho_2-\rho_1}\right) r_i\lai ( r_i\lai +|v^i|)|B_i| \\
    &\leq  \frac{c}{\varepsilon}  \left(\frac{\rho}{\rho_2-\rho_1}\right)^2 \La|U_i|+ \varepsilon|B_i||v_i|^2,
\end{split}
\end{align*}
where to obtain the last inequality, Young's inequality is used.
This completes the proof.
\end{proof}

Now we are ready to prove the time-slice estimate.

\begin{proof}[Proof of Lemma~\ref{lem_timeslice}]
Recalling the properties of $\om_i$ in Proposition~\ref{prop_whitney} and the definitions of $\Theta_1$ and $\Theta_2$, we have
\begin{align*}
	\begin{split}
	&\int_{ (\B_2\times\{t\} ) \cap E(\La)^c}(|v|^2-|\vla-v|^2)\,dx \\
    &= \sum_{i\in\NN}\int_{\B_2\times\{t\}}\om_i(|v|^2-|\vla-v|^2)\,dx\\
	&= \sum_{i\in\Theta_1}\int_{\B_2\times\{t\}}\om_i(|v|^2-|\vla-v|^2)\,dx
    +\sum_{i\in\Theta_2}\int_{\B_2\times\{t\}}\om_i(|v|^2-|\vla-v|^2)\,dx\\
	&=\mathrm{I}+\mathrm{II}.
	\end{split}
\end{align*}
Note that we may assume further that $50U_i\cap \Q_2\ne\emptyset$. Otherwise, $\vla\equiv0$ in $2U_i$.
We estimate $\mathrm{I}$ first. Note that by adding and subtracting we have 
\begin{align*}
	\begin{split}
		\mathrm{I}
        &=\sum_{i\in\Theta_1}\int_{\B_2\times\{t\}}\om_i(|v|^2-|\vla-v|^2)\,dx\\
		&= \sum_{i\in\Theta_1}\int_{\B_2\times\{t\}}\om_i(|\vla-v^i|^2-|\vla-v^i|^2+|v|^2-|\vla-v|^2)\,dx \\
		&= \sum_{i\in\Theta_1}\int_{\B_2\times\{t\}}\om_i(|v^i|^2+2\vla \cdot (v-v^i))\,dx
        -\sum_{i\in\Theta_1}\int_{\B_2\times\{t\}}\om_i|\vla-v^i|^2\,dx \\
		&= \mathrm{I}_1 - \mathrm{I}_2.
	\end{split}
\end{align*}
We have by \ref{ix}  and \ref{prop2_poincare_internal} or \ref{prop2_poincare_boundary} as in the proof of Lemma~\ref{lem_410},
that for $z=(x,t)$ with $x\in 2B_i$ and $t\in \ti_1$
\begin{align*}
	|\vla-v^i| = \biggl|\sum_{j\in \mathcal{I}_i}\om_jv^j-\sum_{j
		\in \mathcal{I}_i}\om_jv^i\biggr|\leq \sum_{j\in\mathcal{I}_i}\om_j|v^j-v^i|\leq c\left(  \frac{\rho}{\rho_2-\rho_1}\right)r_i\lai,
\end{align*}
and therefore we have by \ref{ix} that
\begin{align*}
	\begin{split}
		\mathrm{I}_2 
        &\leq  c\left(  \frac{\rho}{\rho_2-\rho_1}\right)^2\sum_{i\in\Theta_1}|B_i|r_i^2\lai^2 \\
        &\leq c\left(  \frac{\rho}{\rho_2-\rho_1}\right)^2\sum_{i\in\Theta_1}\La|U_i| \\
        &\leq c\left(  \frac{\rho}{\rho_2-\rho_1}\right)^2\La|E(\La)^c|.
	\end{split}
\end{align*}

Then we estimate $\mathrm{I}_1$. By Lemma~\ref{lem_ts_theta1} we have for any $\varepsilon \in (0,1)$ that
\begin{align*}
	\begin{split}
	    &\int_{\B_2\times\{t\}}\om_i(|v^i|^2+2\vla \cdot (v-v^i))\,dx \\
        &\geq \int_{\B_2\times\{t\}}\om_i|v^i|^2\,dx-\varepsilon|B_i||v^i|^2-\frac{c}{\varepsilon}\left(  \frac{\rho}{\rho_2-\rho_1}\right)^2\La|U_i|.
	\end{split}
\end{align*}
We will first estimate
\[
\sum_{i\in\Theta_1}\int_{\B_2\times\{t\}}\om_i|v^i|^2\,dx-\varepsilon\sum_{i\in\Theta_1}|B_i||v^i|
	^2
\]
To estimate the last term, we consider a subset of $\Theta_1$ and its complement defined as
\begin{align*}
    \begin{split}
        \Theta_{1,1}=\{ i\in\mathbb{N}: j\in \Theta_1\quad\text{for all}\quad j\in\mathcal{I}_i \},\quad\Theta_{1,2}=\Theta_1\setminus \Theta_{1,2}.
    \end{split}
\end{align*}
Then as in the proof of Lemma~\ref{lem_410}, we have
\begin{align*}
    \begin{split}
     \sum_{i\in\Theta_{1,1}}|B_i||v^i|^2
     &=\sum_{i\in\Theta_{1,1}}\sum_{j\in \mathcal{I}_i} \fint_{2\B_i\times \{ t \} }\om_j|B_i||v^i|^2 \,dx\\
     &\le   \sum_{i\in\Theta_{1,1}}\sum_{j\in \mathcal{I}_i} \fint_{2\B_i\times \{ t \} } 2\om_j ( |B_i||v^i-v^j|^2+|B_i||v^j|^2) \,dx\\
     &\le\sum_{i\in\Theta_{1,1}}\sum_{j\in \mathcal{I}_i} \fint_{2\B_i\times \{ t \} }\om_j \left( c \left( \frac{\rho}{\rho_2-\rho_1} \right)^2 |B_i|r_i^2\lai^2 +|B_i||v^j|^2  \right) \,dx\\
     &\le c\left( \frac{\rho}{\rho_2-\rho_1} \right)^2\sum_{i\in\Theta_{1,1}}\La|U_i|+c\sum_{i\in\Theta_{1,1}}\sum_{j\in \mathcal{I}_i} \int_{2\B_j\times \{t \} } \om_j|v^j|^2 \,dx
    \end{split}
\end{align*}
where $c=c(\data)$ and to get the last inequality, we used \ref{v}. Next, we apply \ref{ix} to both terms on the right hand side in order to obtain
\begin{align*}
    \begin{split}
        \sum_{i\in\Theta_{1,1}}|B_i||v^i|^2
        &\le c\left( \frac{\rho}{\rho_2-\rho_1} \right)^2 \La|E(\La)^c|+c\sum_{i\in\Theta_1}\int_{\B_2 \times \{t \} } \om_i|v^i|^2 \,dx.
    \end{split}
\end{align*}
On the other hand, if $i\in\Theta_{1,2}$, then there exists $j\in \mathcal{I}_i\cap  \Theta_2$. It follows from \ref{v}, \ref{Uij} and \eqref{The2_ineq} that
\[
cs_i> \sig_2-\sig_1
\]
for some $c=c(\data)$. Therefore, we have
\begin{align*}
    \begin{split}
        \sum_{i\in\Theta_{1,2}}|B_i||v^i|^2
        &\le \sum_{i\in\Theta_{1,2}}|B_i|\fiint_{2U_i} |v|^2\, dz\\
        &\le \sum_{i\in\Theta_{1,2}} \frac{1}{2s_i}\iint_{2U_i} |v|^2\, dz\\
        &\le c\sum_{i\in \Theta_1} \left(\frac{\sig_2}{\sig_2-\sig_1} \right)\iint_{2U_i} \frac{|v|^2}{\sig_2} \, dz.
    \end{split}
\end{align*}
Again using \ref{ix} and canceling the scaling factors in the fractions, we get
\[
\sum_{i\in\Theta_{1,2}}|B_i||v^i|^2\le \left(\frac{\rho}{\rho_2-\rho_1} \right)^2 \iint_{\Q_2} \frac{|v|^2}{\sig_2} \, dz.
\]
Combining these estimates, we obtain
\begin{align*}
    \begin{split}
        &\sum_{i\in\Theta_1}\int_{\B_2\times\{t\}}\om_i|v^i|^2\,dx-\varepsilon\sum_{i\in\Theta_1}|B_i||v^i|^2 \\
    &\ge \sum_{i\in\Theta_1}\int_{\B_2\times\{t\}}\om_i|v^i|^2\,dx -c\varepsilon\sum_{i\in\Theta_1}\int_{\B_2 \times \{t \} } \om_i|v^i|^2 \,dz\\
    &\qquad -c\varepsilon\left( \frac{\rho}{\rho_2-\rho_1} \right)^2\La|E(\La)^c|-c\varepsilon\left(\frac{\rho}{\rho_2-\rho_1} \right)^2 \iint_{\Q_2} \frac{|v|^2}{\sig_2} \, dz.
    \end{split}
\end{align*}
Taking $\varepsilon=c^{-1}$, we get
\begin{align*}
    \begin{split}
       \mathrm{I}_1
       \ge-c\left( \frac{\rho}{\rho_2-\rho_1} \right)^2\La|E(\La)^c|-c \left(\frac{\rho}{\rho_2-\rho_1} \right)^2 \iint_{\Q_2} \frac{|v|^2}{\sig_2} \, dz.
    \end{split}
\end{align*}
We conclude the desired estimate
\[
\mathrm{I} \ge-c\left( \frac{\rho}{\rho_2-\rho_1} \right)^2\La|E(\La)^c|-c \left(\frac{\rho}{\rho_2-\rho_1} \right)^2 \iint_{\Q_2} \frac{|v|^2}{\sig_2} \, dz.
\]

Then we estimate $\mathrm{II}$. Using $|\vla-v|^2 =|\vla|^2-2\vla\cdot v + |v|^2$, we have
\begin{align}\label{ts1}
\begin{split}
    \mathrm{II} 
    &= \sum_{i\in\Theta_2}\int_{\B_2\times\{t\}}\om_i(|v|^2-|\vla-v|^2)\,dx\\
    &= -\sum_{i\in\Theta_2}\int_{\B_2\times\{t\}} \om_i |\vla|^2  \,dx+2\sum_{i\in\Theta_2}\int_{\B_2\times\{t\}}   \om_i v\cdot \vla  \,dx.
\end{split}
\end{align}
First we consider the second term on the right hand side. Taking $\varphi = \eta \zeta\om_i\vhla$ as the test function in the Steklov averaged weak formulation and integrating over $(t_i-4s_i,t)$ in time (again notice that the integral is zero if $t<t_i-4s_i$), we obtain
\begin{align*}
    \begin{split}
        &\int_{t_i-4s_i}^t\int_{\B_2}\pa_t[u-\uu]_h\cdot\eta \zeta\om_i\vhla\,dx\,d\tau \\
        &= -\int_{t_i-4s_i}^t\int_{\B_2}[\mA(\cdot,\na u)]_h \cdot\na(\eta \zeta\om_i\vhla)\,dx\,d\tau\\
        &\qquad+\int_{t_i-4s_i}^t\int_{\B_2}[|F|^{p-2}F+a|F|^{q-2}F]_h \cdot\na(\eta \zeta\om_i\vhla)\,dx\,d\tau.
    \end{split}
\end{align*}
Using the product rule, fundamental theorem of calculus and the fact that $\om_i(\cdot,t_i-4s_i)= 0$, we obtain
\begin{align*}
\begin{split}
    &\int_{t_i-4s_i}^t\int_{\B_2}\pa_t[u-\uu]_h \cdot\eta \zeta\om_i\vhla\,dx\,d\tau \\
    &= \int_{t_i-4s_i}^t\int_{\B_2}\pa_t([u-\uu]_h \cdot\eta \zeta\om_i\vhla)\,dx\,d\tau -\int_{t_i-4s_i}^t\int_{\B_2}[u-\uu]_h \pa_t(\eta \zeta\om_i\vhla)\,dx\,d\tau\\
    &= \int_{\B_2\times\{t\}}[u-\uu]_h\cdot\eta \zeta\om_i\vhla  \,dx - \int_{t_i-4s_i}^t\int_{\B_2}[u-\uu]_h \cdot \pa_t(\eta \zeta\om_i\vhla)\,dx\,d\tau. 
\end{split}
\end{align*}
Taking the limit $h \to 0$ and recalling that $v = (u-\uu)\eta\zeta$, we conclude that 
\begin{align*}
\begin{split}
    \int_{\B_2\times\{t\}}  \om_iv\cdot \vla  \,dx 
    &= -\int_{t_i-4s_i}^t\int_{\B_2}\mA(z,\na u) \cdot\na(\eta \zeta\om_i\vla)\,dx\,d\tau\\
    &\qquad+\int_{t_i-4s_i}^t\int_{\B_2} (|F|^{p-2}F+a|F|^{q-2}F ) \cdot\na(\eta \zeta\om_i\vhla)\,dx\,d\tau \\
    &\qquad +\int_{t_i-4s_i}^t\int_{\B_2}(u-u_0) \cdot \pa_t(\eta \zeta\om_i\vla)\,dx\,d\tau.
\end{split}
\end{align*}
 Taking absolute values of both sides and using the growth condition of $\mA$, we get that
\begin{align}\label{ts5}
\begin{split}
        \left|\int_{\B_2\times\{t\}}   \om_iv\cdot \vla \,dx\right| 
        &\leq c\iint_{\Q_2}(|\na u|^{p-1}+a(z)|\na u|^{q-1})|\na(\eta \zeta\om_i\vla)|\,dz\\
        &\qquad+ c\iint_{\Q_2}(|F|^{p-1}+a(z)|F|^{q-1})|\na(\eta \zeta\om_i\vla)|\,dz\\
        &\qquad+\iint_{\Q_2}|(u-u_0)  \cdot \pa_t(\eta \zeta\om_i\vla)|\,dx\,d\tau.
\end{split}
\end{align}

We start by estimating $|\na(\eta \zeta\om_i\vla)|$ where $i\in\Theta_2$. 
There holds
\[
|\na(\eta \zeta\om_i\vla)| \leq |\vla\na \eta|+|\vla\na\om_i|+|\na\vla|.
\]
To estimate the first and the second term on the right hand side, we apply Lemma~\ref{lem_theta_2} and Lemma~\ref{pre_The2} to get
\[
|\vla\na \eta|+|\vla\na\om_i|\le c\left ( \frac{\rho}{\rho_2-\rho_1} \right)\lai + c\frac{\rho}{r_i}\la_i\le c\theta^\frac{2-q}{p}\left ( \frac{\rho}{\rho_2-\rho_1} \right)\lai,
\]
while for the third term on the right hand side the same estimate holds by Lemma~\ref{lem_ts_2}.
Then with Proposition~\ref{prop_3_4}\ref{prop2_est_nau_p-1} and the fact that $\om_i$ vanishes outside $2U_i$, we have
\begin{align*}
    \iint_{\Q_2}(|\na u|^{p-1}+a(z)|\na u|^{q-1})|\na(\eta \zeta\om_i\vla)|\,dz \leq c\theta^{\frac
    {2-q}{p}}\left ( \frac{\rho}{\rho_2-\rho_1} \right) \La|2U_i|.
\end{align*}
Similarly, we get
\[
 \iint_{\Q_2}(|F|^{p-1}+a(z)|F|^{q-1})|\na(\eta \zeta\om_i\vla)|\,dz \leq c\theta^{\frac
    {2-q}{p}}\left ( \frac{\rho}{\rho_2-\rho_1} \right) \La|2U_i|.
\]

Next, we consider the third term in \eqref{ts5}. 
As $i\in\Theta_2$, we have $\sig_2-\sig_1<c s_i$. 
Thus, from the definition of $\vla$ in \eqref{t_325_1}, the product rule and \ref{xii}, we obtain that in $2U_i$
\begin{align*}
     |\pa_t(\zeta\om_i\vla)| = |\sum_{j\in \mathcal{I}_i} v^j\pa_t(\zeta\om_i\om_j)|\leq c\sum_{j\in \mathcal{I}_i}|v^j| \left( \frac{1}{\sig_2-\sig_1}+\frac{1}{s_i} \right )\leq \frac{c}{\sig_2-\sig_1}\sum_{j\in \mathcal{I}_i}|v^j|.
\end{align*}
Recall that the cardinality of $\mathcal{I}_i$ is bounded by \ref{ix}. With the above estimate and Young's inequality, we get for $i\in \Theta_2$ that
\begin{align*}
    \begin{split}
        \iint_{\Q_2}|(u-u_0)\pa_t(\zeta\om_i\vla)|\,dz&\leq \frac{c}{\sig_2-\sig_1}\iint_{\Q_2\cap 2U_i}\sum_{j\in\mathcal{I}_i}|u-u_0||v^j|\,dz\\
        &\leq \frac{c}{\sig_2-\sig_1}\iint_{\Q_2\cap 2U_i} \biggl( \sum_{j\in\mathcal{I}_i}|u-u_0|^2+|v^j|^2 \biggr) \,dz \\
        &\leq \frac{c}{\sig_2-\sig_1} \iint_{\Q_2\cap 2U_i}|u-u_0|^2\,dz  \\
        &\qquad  + \frac{c}{\sig_2-\sig_1}\sum_{j\in\mathcal{I}_i}\iint_{2U_j}|v|^2\,dz,
    \end{split}
\end{align*}
where to obtain the last inequality, we applied H\"older's inequality and the fact that $|U_j|\approx |U_i|$.
Combining these estimates, \eqref{ts5} becomes
\begin{align*}
\begin{split}
    &\sum_{i\in\Theta_2}\int_{\B_2\times\{t\}}   \om_i v\cdot \vla \,dx\\
    &\leq c\theta^{\frac
    {2-q}{p}}\left ( \frac{\rho}{\rho_2-\rho_1} \right)\sum_{i\in\Theta_2}\La|2U_i| \\
    &\qquad +\sum_{i\in\Theta_2} \frac{1}{\sig_2-\sig_1} \left( \iint_{\Q_2\cap 2U_i}|u-u_0|^2\,dz + \sum_{j\in\mathcal{I}_i}\iint_{2U_j}|v|^2\,dz \right)\\
    &\leq c\theta^{\frac
    {2-q}{p}} \left ( \frac{\rho}{\rho_2-\rho_1} \right)\La|E(\La)^c| + \frac{c}{\sig_2-\sig_1}\iint_{\Q_2}|u-u_0|^2\,dz.
\end{split}
\end{align*}

Finally we estimate the first term in \eqref{ts1}.  
From the definition of $\vla$, Hölder's inequality, \ref{Uij} and \ref{ix} we obtain
\begin{align*}
\begin{split}
    \int_{\B_2\times\{t\}}  \om_i |\vla|^2 \,dx 
    &= \int_{\B_2\times\{t\}}    \om_i \biggl|\sum_{j\in\mathcal{I}_i}v^j\om_j  \biggr|^2    \,dx\\
    &\leq c\int_{ 2B_i} \sum_{j\in\mathcal{I}_i}|v^j|^2\,dx\\
    &\leq c\sum_{j\in\mathcal{I}_i}\frac{|2B_i|}{|2U_j|}\iint_{2U_j}|v|^2\,dz\\
    &\leq \frac{c}{s_i} \sum_{j\in\mathcal{I}_i} \iint_{2U_j}|v|^2\,dz.
\end{split}
\end{align*}
As we have $\sig_2-\sig_1 \leq c s_i$ for $i\in\Theta_2$ and \ref{ix}, we conclude from the above inequality that 
\begin{align*}
\begin{split}
\sum_{i\in\Theta_2}\int_{\B_2\times\{t\}} \om_i |\vla|^2  \,dx &\leq  c \sum_{i\in\Theta_2}\left( \frac{c}{s_i} \sum_{j\in\mathcal{I}_i} \iint_{2U_j}|v|^2\,dz\right) \\
&\leq \frac{c}{\sig_2-\sig_1}\sum_{i\in \mathbb{N}}\iint_{2U_i}|v|^2\,dz \\
&\leq \frac{c}{\sig_2-\sig_1}\iint_{\Q_2}|v|^2\,dz.
\end{split}
\end{align*}
This completes the proof.
\end{proof}

\subsection{Proof of the Caccioppoli inequality}

Now we are ready to prove the Caccioppoli inequality that is used to prove the reverse Hölder inequalities in Lemma~\ref{lem_reverse_Hölder} and Lemma~\ref{lem_reverse_Hölder_pq}. For technical reasons we get an unwanted truncated maximal function as a weight in the Caccioppoli inequality. We denote
\begin{align}\label{def_mh}
    \mh(z)=\max\{\Upsilon(z),\theta\Pi\},
\end{align}
where $\theta=\theta(\data)>0$ is determined in the proof of the lemma and $\Upsilon(z)$ is defined in \eqref{def_m}. Also recall the definition of $\Pizo$ in \eqref{def_Pizo}.
Note that the technical assumption $\rho_2-\rho_1\geq \rho$ is used only at the end of the proof. 
We recall $\mu\in(0,1)$ is chosen to satisfy \eqref{def_mu} and \eqref{def_mu_2} depending on $n,p,q$ and $\alpha$.

\begin{lemma}\label{lem_caccioppoli}
Let $u$ be a very weak solution to \eqref{11}. Suppose $\rho\le \rho_1< \rho_2\le 4\rho$ with $\rho_2-\rho_1\ge\rho$. Then there exists $\delta_1=\delta_1(\data)\in(\mu,1)$ and $c=c(\data)$ such that
\begin{align}\label{claim_cac}
    \begin{split}
     &   \sup_{t\in\ti_1}
    \fint_{\B_1\times \{t\}} \frac{|u-u_{\Q_1}|^2}{\sig_1\mh^{1-\de}}\,dx+\Pizo^\de \\
    &\leq        c\fiint_{\Q_2}\frac{|u-u_{\Q_2}|^2}{\Pi^{1-\delta}\sig_2}\,dz  + c\fiint_{\Q_2}H\left(z,\frac{|u-u_{\Q_2}|}{\rho_2}\right)^\de\,dz \\
    &\qquad +c  \fiint_{\Q_2} H(z,|F|)^\delta\,dz,
    \end{split}
\end{align}
when $\de \in (\de_1,1)$
\end{lemma}

\begin{proof}
Let $\La \geq \theta\Pi$. Let $h>0$ be small enough and let $\vhla$ be the Lipschitz truncation of $\vh$, defined as in Section~\ref{sec_liptr}. Taking $\eta \vhla$ as a test function in \eqref{steklov_averaged_wf} and integrating with respect to time over $(t_1,t)\subset (t_0-\tfrac{4\sig_1+\sig_2}{5},t_0+\sig_1) $, where $t_1\in (t_0-\tfrac{4\sig_1+\sig_2}{5},t_0-\sig_1)  \subset \ti_2$ is chosen later and $t \in \ti_1$, we get
\begin{align*}
\begin{split}
    0 &= \int_{t_1}^t\int_{\B_2} \pa_\tau\vh \cdot \vhla \, \,dx \,d\tau + \int_{t_1}^t\int_{\B_2} [\mA(\cdot,\na u)]_h\cdot\na(\eta \vhla) \, \,dx\, d\tau\\
    &  \qquad + \int_{t_1}^t\int_{\B_2} [|F|^{p-2}F+a|F|^{q-2}F]_h\cdot\na(\eta \vhla) \, \,dx \,d\tau  \\
    & = \mathrm{I}_h+\mathrm{II}_h+\mathrm{III}_h,
\end{split}
\end{align*}
where we used the fact that $\zeta\equiv1$ in $(t_1,t)$.
Next, we estimate both terms of the above inequality. First we consider $\mathrm{I}_h$. After adding and subtracting and using the chain rule, we get
\begin{align*}
    \begin{split}
        \mathrm{I}_h &= \int_{t_1}^t\int_{\B_2} \pa_\tau\vh \cdot \vhla \, dx\,d\tau \\
        &= \int_{t_1}^t\int_{\B_2} (\pa_\tau\vh \cdot \vhla - \pa_\tau\vh\cdot\vh+\pa_\tau\vh\cdot\vh )  \,dx\,d\tau\\
        &=\int_{t_1}^t\int_{\B_2} \left( \pa_\tau\vh \cdot (\vhla - \vh)+\frac{1}{2}\pa_\tau|\vh|^2 \right) \,dx\,d\tau.
    \end{split}
\end{align*}
Again, by adding and subtracting and using the chain rule, we get 
\begin{align*}
    \begin{split}
        \mathrm{I}_h &= \int_{t_1}^t\int_{\B_2} \left( \pa_\tau\vh \cdot (\vhla - \vh)+\frac{1}{2}\pa_\tau|\vh|^2 \right)\,dx\,d\tau\\
        &\qquad + \int_{t_1}^t\int_{\B_2} \left( \pa_\tau\vhla\cdot(\vhla-\vh) - \pa_\tau\vhla\cdot(\vhla-\vh)  \right) \,dx\,d\tau\\
        & = \int_{t_1}^t\int_{\B_2} \left( \frac{1}{2}\pa_\tau|\vh|^2 - \frac{1}{2}\pa_\tau |\vhla-\vh|^2 + \pa_\tau\vhla\cdot(\vhla-\vh)  \right) \,dx\,d\tau  \\
        & = \frac{1}{2} \int_{t_1}^t\int_{\B_2} \pa_\tau(|\vh|^2-|\vhla-\vh|^2) \,dx\,d\tau+ \int_{t_1}^t\int_{\B_2\cap E(\La)^c} \pa_\tau\vhla\cdot(\vhla-\vh) \,dx\,d\tau,
    \end{split}
\end{align*}
where we shorten $(\B_2 \times\{\tau\}) \cap E(\La)^c$ as $\B_2\cap E(\La)^c$.
As $\vh$ and $\vhla$ are differentiable in time, we have by the fundamental theorem of calculus that
\begin{align*}
      \mathrm{I}_h = \frac{1}{2} \int_{\B_2\times \{\tau\}} \Biggr|_{t_1}^t(|\vh|^2-|\vhla-\vh|^2) \, dx+ \int_{t_1}^t\int_{\B_2\cap E(\La)^c} \pa_\tau\vhla\cdot(\vhla-\vh) \,dx\,d\tau.
\end{align*}

We let $h\to 0^+$ and get
\begin{align*}
      \lim_{h\to 0}\mathrm{I}_h = \frac{1}{2} \int_{\B_2\times \{\tau\}} \Biggr|_{t_1}^t(|v|^2-|\vla-v|^2) \, dx + \int_{t_1}^t\int_{\B_2\cap E(\La)^c} \pa_\tau\vla\cdot(\vla-v) \,dx\,d\tau.
\end{align*}
After estimating the second term with  Lemma~\ref{time_deri_est},
we obtain
\begin{align}\label{cac_first_est_Ih}
      \begin{split}
          \lim_{h\to 0}\mathrm{I}_h 
          &\geq -\frac{1}{2} \int_{\B_2\times \{t_1\}}(|v|^2-|\vla-v|^2) \, dx +\frac{1}{2} \int_{\B_2\times \{t\}}(|v|^2-|\vla-v|^2) \, dx\\
          &\qquad- c \left( \frac{\rho}{\rho_2-\rho_1} \right) \La |E(\La)^c|- c\left(\frac{\rho}{\rho_2-\rho_1}\right)^2\iint_{ \Q_2 } \frac{|v |^2}{\sig_2}\,dz.
      \end{split}
\end{align}
To estimate the first term above, we choose $t_1 \in  (t_0-\tfrac{4\sig_1+\sig_2}{5},t_0-\sig_1) $ such that
\begin{align*}
\begin{split}
    \int_{\B_2 \times \{t_1\} }(|v|^2-|\vla-v|^2) \, dx 
    &\leq \fint_{t_0-\tfrac{4\sig_1+\sig_2}{5}}^{t_0-\sig_1}\int_{\B_2}(|v|^2-|\vla-v|^2) \, dx\,d\tau \\
    &\leq \frac{5}{\sig_2-\sig_1}\iint_{\Q_2}(|v|^2+|\vla-v|^2)\,dz\\
    &\leq \frac{c}{\sig_2-\sig_1}\left(\iint_{\Q_2}|v|^2\,dz+\iint_{\Q_2\cap E(\La)^c}|\vla|^2\,dz\right).
\end{split}
\end{align*}
From the definition of $\vla$, Hölder's inequality and \ref{ix}, we obtain that
\begin{align*}
    \begin{split}
        \iint_{\Q_2\cap E(\La)^c}|\vla|^2\,dz 
        &= \iint_{\Q_2\cap E(\La)^c} \biggl|\sum_{i\in\NN}v^i\om_i \biggr|^2\,dz \\
        &\leq c\sum_{i\in\NN}|\Q_2\cap 2U_i||v^i|^2 \\
        &\leq c\sum_{i\in\NN}\frac{|\Q_2\cap 2U_i|}{|2U_i|}\iint_{2U_i}|v|^2\,dz \\
        &\leq c\iint_{\Q_2}|v|^2\,dz,
    \end{split}
\end{align*}
and therefore
\begin{align*}
    \int_{\B_2\times \{t_1\}}(|v|^2-|\vla-v|^2)\, dx \leq c\left(\frac{\rho}{\rho_2-\rho_1}\right)^2\iint_{ \Q_2 } \frac{|v |^2}{\sig_2}\,dz.
\end{align*}
Finally, we split the second term on the right hand side of \eqref{cac_first_est_Ih} into sets $E(\La)$ and $E(\La)^c$ and apply Lemma~\ref{lem_timeslice} to obtain
\begin{align*}
    \begin{split}
        \int_{\B_2\times \{t\}}(|v|^2-|\vla-v|^2) \, dx 
        &\geq \int_{  (\B_2\times \{t\} )\cap E(\La)}|v|^2 \, dx - c\theta^{\frac
    {2-q}{p}} \left ( \frac{\rho}{\rho_2-\rho_1} \right)^2 \La|E(\La)^c|\\
    &\qquad -c \left( \frac{\rho}{\rho_2-\rho_1}  \right)^2 \iint_{\Q_2}\frac{|u-u_0|^2}{\sig_2}\,dz.
    \end{split}
\end{align*}
We conclude that
\begin{align}\label{est_for_Ih}
          \begin{split}
              \lim_{h\to 0^+}\mathrm{I}_h &\geq \int_{ (\B_2 \times \{t\} )\cap E(\La)}|v|^2 \, dx - c\theta^{\frac
    {2-q}{p}} \left ( \frac{\rho}{\rho_2-\rho_1} \right)^2 \La|E(\La)^c|\\
    &\qquad -c \left( \frac{\rho}{\rho_2-\rho_1}  \right)^2 \iint_{\Q_2}\frac{|u-u_0|^2}{\sig_2}\,dz.
          \end{split}
\end{align}

Next, we consider the term $\mathrm{II}_h$. Letting $h\to 0^+$ and again dividing into the  sets $E(\La)$ and $E(\La)^c$, we get
\begin{align*}
\begin{split}
    \lim_{h\to 0^+} \mathrm{II}_h &= \int_{t_1}^t\int_{\B_2} \mA(z,\na u)\cdot\na(\eta \vla) \,d\tau \,dx \\
    &= \int_{t_1}^t\int_{\B_2\cap E(\La)} \mA(z,\na u)\cdot\na(\eta v) \,d\tau \,dx\\
    &\qquad + \int_{t_1}^t\int_{\B_2\cap E(\La)^c} \mA(z,\na u)\cdot\na(\eta \vla) \,d\tau \,dx.
\end{split}
\end{align*}
Leaving the first term, we estimate the second term. Using \eqref{12}, we have
\begin{align*}
\begin{split}
    &\left|\int_{t_1}^t\int_{\B_2\cap E(\La)^c} \mA(z,\na u)\cdot\na(\eta \vla) \,dz\right| \\
    &\leq 
    \sum_{i\in\NN}\iint_{\Q_2\cap 2U_i} |\mA(z,\na u)||\na(\eta \vla)| \,dz \\
    &\leq c\sum_{i\in\NN}\iint_{\Q_2\cap 2U_i} (|\na u|^{p-1}+a(z) |\na u|^{q-1})|\na(\eta \vla)| \,dz\\
    &\leq c\sum_{i\in\NN}\iint_{\Q_2\cap 2U_i} (|\na u|^{p-1}+a(z)|\na u|^{q-1})\left(\frac{|\vla|}{\rho_2-\rho_1}+|\na\vla|\right) \,dz. 
\end{split}
\end{align*}
We estimate $|\vla|$ using Lemma~\ref{pre_The2} and the term $(|\na u|^{p-1}+a|\na u|^{q-1})$ using Proposition~\ref{prop_3_4}\ref{prop2_est_nau_p-1}. Regarding $|\na\vla|$, we apply Lemma~\ref{lem_410} if $i\in\Theta_1$ and Lemma~\ref{lem_ts_2} if $i\in \Theta_2$.
Then we get
\begin{align*}
    \begin{split}
        \left|\int_{t_1}^t\int_{\B_2\cap E(\La)^c} \mA(z,\na u)\cdot\na(\eta \vla) \,dz\right|
        &\le c\theta^{\frac
    {2-q}{p}}\left ( \frac{\rho}{\rho_2-\rho_1} \right) \sum_{i\in\mathbb{N}} \La|2U_i|\\
    &\le c\theta^{\frac{2-q}{p}}   \left ( \frac{\rho}{\rho_2-\rho_1} \right) \La|E(\La)^c|.
    \end{split}
\end{align*}
We conclude that 
\begin{align}\label{est_for_IIh}
\begin{split}
    \lim_{h\to 0^+}\mathrm{II}_h 
    &\geq \int_{t_1}^t\int_{\B_2\cap E(\La)}\mA(z,\na u)\cdot\na (v\eta)\,d\tau\,dx\\
    &\qquad-c\theta^{\frac{2-q}{p}}   \left ( \frac{\rho}{\rho_2-\rho_1} \right) \La|E(\La)^c|.
\end{split}
\end{align}
Similarly, we have
\begin{align*}
    \begin{split}
        \lim_{h\to 0^+}\mathrm{III}_h 
    &\geq \int_{t_1}^t\int_{\B_2\cap E(\La)} (|F|^{p-2}F+a|F|^{q-2}F)\cdot\na (v\eta)\,d\tau\,dx\\
    &\qquad-c\theta^{\frac{2-q}{p}}   \left ( \frac{\rho}{\rho_2-\rho_1} \right) \La|E(\La)^c|.
    \end{split}
\end{align*}

From \eqref{est_for_Ih}, \eqref{est_for_IIh} and the above display, we conclude that
\begin{align*}
    \begin{split}
        &\int_{(\B_2\times \{t\})\cap E(\La)}|v|^2 \, dx+\int_{t_1}^t\int_{\B_2\cap E(\La)}\mA(z,\na u)\cdot\na (v\eta)\,dx\,d\tau\\
        &\qquad+ \int_{t_1}^t\int_{\B_2\cap E(\La)} (|F|^{p-2}F+a|F|^{q-2}F)\cdot\na (v\eta)\,d\tau\,dx\\
        &\le  c\theta^{\frac
    {2-q}{p}} \left ( \frac{\rho}{\rho_2-\rho_1} \right)^2 \La|E(\La)^c| +c \left( \frac{\rho}{\rho_2-\rho_1}  \right)^2 \iint_{\Q_2}\frac{|u-u_0|^2}{\sig_2}\,dz.
    \end{split}
\end{align*}
 
Now we multiply both sides by $\La^{-1-(1-\de)}$ and integrate with respect to $\La$ over $(\theta\Pi,\infty)$. This gives
\begin{align}\label{integrated_la}
\begin{split}
    &\int_{\theta\Pi}^\infty \La^{-2+\de} \int_{(\B_2\times \{t\})\cap E(\La)} |v |^2\,dx\,d\La \\
    &\qquad+\int_{\theta\Pi}^\infty \La^{-2+\de} \int_{t_1}^t\int_{\B_2\cap E(\La)}\mA(z,\na u)\cdot\na (v\eta)\,dx\,d\tau\,d\La \\
    &\qquad+\int_{\theta\Pi}^\infty \La^{-2+\de} \int_{t_1}^t\int_{\B_2\cap E(\La)} (|F|^{p-2}F+a|F|^{q-2}F)\cdot\na (v\eta)\,dx\,d\tau\,d\La\\
    &\leq c\theta^{\frac
    {2-q}{p}} \left ( \frac{\rho}{\rho_2-\rho_1} \right)^2 \int_{\theta\Pi}^\infty\La^{-1+\delta}|E(\La)^c| \,d\La \\
    &\qquad +  c \left( \frac{\rho}{\rho_2-\rho_1}  \right)^2 \iint_{\Q_2}\frac{|u-u_0|^2}{\sig_2}\,dz\int_{\theta\Pi}^\infty \La^{-2+\de} \,d\La.
\end{split}
\end{align}
For the first term in \eqref{integrated_la} we have by Fubini's theorem and the definition of $\mh$ in \eqref{def_mh}, that
\begin{align*}
\begin{split}
    \int_{\theta\Pi}^\infty\La^{-2+\de} \int_{(\B_2\times \{t\})\cap E(\La)}|v|^2\,dx\,d\La 
    &= \int_{\B_2\times \{t\}} |v|^2 \int_{\theta\Pi}^\infty\La^{-2+\de}\chi_{[\Upsilon(x,t),\infty)} \,d\La \,dx\\
    &=\int_{\B_2\times \{t\}} |v|^2 \int_{\mh(x,t)}^\infty\La^{-2+\de}\,d\La \,dx\\
    &=\frac{1}{1-\de}\int_{\B_2\times \{t\}} \frac{|v|^2 }{\mh^{1-\de}}\,dx.
\end{split}
\end{align*}
Similarly, we have for the second term in \eqref{integrated_la} that
\begin{align*}
    \begin{split}
        &\int_{\theta\Pi}^\infty\La^{-2+\de} \int_{t_1}^t\int_{\B_2\cap E(\La)}\mA(z,\na u)\cdot\na (v\eta)\,dx\,d\tau\,d\La \\ 
        & = \frac{1}{1-\de}\int_{t_1}^t\int_{\B_2} \frac{\mA(z,\na u)\cdot\na (v\eta)}{\mh^{1-\de}}\,dx\,d\tau
    \end{split}
\end{align*}
and for the third term in \eqref{integrated_la} that
\begin{align*}
    \begin{split}
        &\int_{\theta\Pi}^\infty\La^{-2+\de} \int_{t_1}^t\int_{\B_2\cap E(\La)} (|F|^{p-2}F+a|F|^{q-2}F)\cdot\na (v\eta)\,dx\,d\tau\,d\La \\ 
        & = \frac{1}{1-\de}\int_{t_1}^t\int_{\B_2} \frac{(|F|^{p-2}F+a|F|^{q-2}F)\cdot\na (v\eta)}{\mh^{1-\de}}\,dx\,d\tau.
    \end{split}
\end{align*}

Next, we estimate the right hand side of \eqref{integrated_la}, that is
\begin{align*}
    \begin{split}
    &c\theta^{\frac
    {2-q}{p}} \left ( \frac{\rho}{\rho_2-\rho_1} \right)^2 \int_{\theta\Pi}^\infty\La^{-1+\delta}|E(\La)^c| \,d\La\\
    &\qquad +  c\left( \frac{\rho}{\rho_2-\rho_1}  \right)^2 \iint_{\Q_2}\frac{|u-u_0|^2}{\sig_2}\,dz\int_{\theta\Pi}^\infty \La^{-2+\de} \,d\La.
    \end{split}
\end{align*}
For the first term we have by Fubini's theorem and \eqref{eq_est_mz} that
\begin{align*}
\begin{split}
    \int_{\theta \Pi}^\infty \La^{-1+\de}|E(\La)^c| \,d\La 
    &= \int_{\theta \Pi}^\infty \La^{-1+\de}\iint_{\RR^{n+1}}\chi_{\{z\in \RR^{n+1}: \Upsilon(z) >\La \}}\,dz \,d\La\\
    &=\iint_{\RR^{n+1}}\int_{\theta \Pi}^{\Upsilon(z)} \La^{-1+\de} \,d\La\,dz\\
    &\leq \iint_{\RR^{n+1}}\int_{0}^{\Upsilon(z)} \La^{-1+\de} \,d\La\,dz\\
    &= \frac{1}{\de}\iint_{\RR^{n+1}} \Upsilon(z)^{\de}\,dz \\
    &\leq \frac{c}{\de} \Pizo^\de|\Q_2|,
\end{split}
\end{align*}
where $c=c(n,p,q)$. 
For the second term we have again by Fubini's theorem that 
\begin{align*}
    \int_{\theta\Pi}^\infty\La^{-2+\de} \iint_{\Q_2}  \frac{|u-u_0|^2}{\sig_2}   \,dz \,d\La = \frac{1}{(1-\de)  (\theta\Pi)^{1-\de}}\iint_{Q_2}\frac{|u-u_0|^2}{\sig_2}\,dz.
\end{align*}

Inserting the previously obtained estimates into \eqref{integrated_la} and multiplying both sides with $1-\de$, we obtain
\begin{align}\label{after_fubini}
\begin{split}
    &\int_{\B_2\times\{t\} } \frac{|v |^2 }{\mh^{1-\de}}\,dx + \int_{t_1}^t\int_{\B_2} \frac{\mA(z,\na u)\cdot\na (v\eta)}{\mh^{1-\de}}\,dx\,d\tau \\
    &\qquad+ \int_{t_1}^t\int_{\B_2} \frac{(|F|^{p-2}F+a|F|^{q-2}F)\cdot\na (v\eta)}{\mh ^{1-\de}}\,dx\,d\tau\\
    &\leq c\left( \frac{\rho}{\rho_2-\rho_1} \right)^2   \left( \frac{1-\de}{\delta} \right)  \theta^\frac{2-q}{p}  \Pi_{z_0}^{\de}  |\Q_2| \\
    &\qquad+ c\left( \frac{\rho}{\rho_2-\rho_1} \right)^2 \frac{1}{ (\theta\Pi)^{1-\de}}\iint_{\Q_2}\frac{|u-u_0|^2}{\sig_2}\,dz,
\end{split}
\end{align}
where $c=c(\data)$. Note that the right hand side is finite.

Next, we estimate the term containing $\mA$. By the product rule and \eqref{12}, we get
\begin{align*}
\begin{split}
    &\int_{t_1}^t\int_{\B_2} \frac{\mA(z,\na u)\cdot\na (v\eta)}{\mh^{1-\de}}\,dx\,d\tau \\
    &= \int_{t_1}^t\int_{\B_2} \frac{\mA(z,\na u)\cdot\na ((u-u_0)\eta^2)}{\mh^{1-\de}}\,dx\,d\tau \\
    &\geq c \int_{t_1}^t\int_{\B_2}\frac{ (|\na u|^p+a(z)|\na u|^q) \eta^2  }{\mh^{1-\de}}\,dx\,d\tau \\
    &\qquad- c\int_{t_1}^t\int_{\B_2}\frac{(|\na u|^{p-1}+a(z)|\na u|^{q-1})\eta|u-u_0|}{(\rho_2-\rho_1)\mh^{1-\de}}\,dx\,d\tau
\end{split}
\end{align*}
We apply Young's inequality to the second term on the right hand side to get
\begin{align*}
\begin{split}
    \int_{t_1}^t\int_{\B_2} \frac{\mA(z,\na u)\cdot\na (v\eta)}{\mh^{1-\de}}\,dx\,d\tau
    &\geq c \int_{t_1}^t\int_{\B_2}\frac{ (|\na u|^p+a(z)|\na u|^q) \eta^2  }{\mh^{1-\de}} \,dx\,d\tau\\
    &\qquad- c\int_{t_1}^t\int_{\B_2} \frac{|u-u_0|^p}{(\rho_2-\rho_1)^p\mh^{1-\de}} \,dx\,d\tau\\
    &\qquad- c\int_{t_1}^t\int_{\B_2} \frac{a|u-u_0|^q}{(\rho_2-\rho_1)^q\mh^{1-\de}} \,dx\,d\tau
\end{split}
\end{align*}

We estimate further the second term on the right hand side of the above inequality. As 
\[
\mh\geq \frac{|u-u_0|^p}{\rho^p}+a\frac{|u-u_0|^q}{\rho^q}+|F|^p+a|F|^q
\]
we get 
\begin{align*}
    \begin{split}
        \int_{t_1}^t\int_{\B_2} \frac{\mA(\cdot,\na u)\cdot\na (v\eta)}{\mh ^{1-\de}}\,dz
    &\geq c\int_{t_1}^t\int_{\B_2}   \frac{H(z,|\na u|)}{\mh^{1-\delta}} \eta^2 \,dx\,d\tau \\
    &\qquad-c\left( \frac{\rho}{\rho_2-\rho_1} \right)^q \iint_{\Q_2}H\left(z,\frac{ |u-u_0|  }{\rho}\right)^\de\,dz.
    \end{split}
\end{align*}
Similarly, we have
\begin{align*}
    \begin{split}
        &\int_{t_1}^t\int_{\B_2} \frac{(|F|^{p-2}F+a|F|^{q-2}F)\cdot\na (v\eta)}{\mh ^{1-\de}}\,dx\,d\tau  \\
        &\ge  -\frac{c}{2}\int_{t_1}^t\int_{\B_2}   \frac{H(z,|\na u|)}{\mh^{1-\delta}} \eta^2 \,dx\,d\tau  -c_2  \iint_{\Q_2} H(z,|F|)^\delta\,dz  \\
        &\qquad- c_2\left( \frac{\rho}{\rho_2-\rho_1} \right)^q \iint_{\Q_2}H\left(z, \frac{ |u-u_0| }{\rho}\right)^\de\,dz,
    \end{split}
\end{align*}
where the constant $c = c(\data)$ is the same as in the previous display in order to absorb the first term on the right hand side of the above display.

Inserting the above estimates to \eqref{after_fubini}, we get 
\begin{align*}
    \begin{split}
        &\int_{\B_2 \times \{t\}} \frac{|v|^2}{\mh^{1-\de}}\,dx + \int_{t_1}^t\int_{\B_1}\frac{H(z,|\na u|)}{\mh^{1-\de}}\,dx\,d\tau\\ 
        &\leq c\left( \frac{\rho}{\rho_2-\rho_1} \right)^2   \left( \frac{1-\de}{\delta} \right)  \theta^\frac{2-q}{p}  \Pi_{z_0}^{\de}  |\Q_2| \\
         &\qquad +   c\left( \frac{\rho}{\rho_2-\rho_1} \right)^2 \frac{1}{ (\theta\Pi)^{1-\de}}\iint_{\Q_2}\frac{|u-u_0|^2}{\sig_2}\,dz \\
        &\qquad+ c\left( \frac{\rho}{\rho_2-\rho_1} \right)^q \iint_{\Q_2}H\left(z,\frac{ |u-u_0| }{\rho}  \right)^\de\,dz +c  \iint_{\Q_2} H(z,|F|)^\delta\,dz.
    \end{split}
\end{align*}
As the above inequality holds for any $t\in \ti_1$ and $t_1\leq t_0-\sigma_1$, we have
\begin{align}\label{after_A}
     \begin{split}
        &\sup_{t\in\ti_1}\int_{\B_1 \times \{t\}} \frac{|u-u_0|^2}{\mh^{1-\de}}\,dx + \iint_{\Q_1}\frac{H(z,|\na u|)}{\mh^{1-\de}}\,dz\\ 
        &\leq c\left( \frac{\rho}{\rho_2-\rho_1} \right)^2   \left( \frac{1-\de}{\delta} \right)  \theta^\frac{2-q}{p}  \Pi_{z_0}^{\de}  |\Q_2| \\
         &\qquad +   c\left( \frac{\rho}{\rho_2-\rho_1} \right)^2 \frac{1}{ (\theta\Pi)^{1-\de}}\iint_{\Q_2}\frac{|u-u_0|^2}{\sig_2}\,dz \\
        &\qquad+ c\left( \frac{\rho}{\rho_2-\rho_1} \right)^q \iint_{\Q_2}H\left(z,\frac{|u-u_0|}{\rho}\right)^\de\,dz +c  \iint_{\Q_2} H(z,|F|)^\delta\,dz.
    \end{split}
\end{align}

Next, we get rid of the weight $\mh^{-1+\de}$ in the second term on the left-hand side of \eqref{after_A}. Let
\begin{align*}
E=\{z\in \Q_1 : H(z,|\na u|) \geq (1-\de) \Upsilon(z)\;\text{and}\; \Upsilon(z)\geq \theta\Pi\}.
\end{align*}
Note that in $E$ we have $\Upsilon=\mh$, while for a.e. $z \in \Q_1/E$, either $H(z,|\na u|)< (1-\de) \Upsilon (z)$ or $H(z,|\na u|)\leq \Upsilon(z)< \theta \Pi$. Therefore
\begin{align*}
    \begin{split}
        \iint_{\Q_1}H(z,|\na u|)^\de\,dz &= \iint_EH(z,|\na u|)^\de\,dz+\iint_{\Q_1\setminus E}H(z,|\na u|)^\de\,dz\\
        &= \iint_E\frac{H(z,|\na u|)}{H(z,|\na u|)^{1-\de}}\,dz+\iint_{\Q_1\setminus E}H(z,|\na u|)^\de\,dz\\
        &\leq (1-\de)^{-1+\de}\iint_E \frac{H(z,|\na u|)}{\mh ^{1-\de}}\,dz+ \iint_{\Q_1\setminus E}((1-\de) \Upsilon)^{\de}\,dz\\
        &\qquad+\iint_{\Q_1\setminus E}(\theta\Pi)^\de\,dz\\
        &\leq 2\iint_{\Q_1} \frac{H(z,|\na u|)}{\mh^{1-\de}}\,dz+ (1-\de)^\delta\iint_{\RR^{n+1}} \Upsilon ^\de\,dz+(\theta\Pi)^\de|\Q_2|,
    \end{split}
\end{align*}
where we also used the elementary fact that $(1-\de)^{-1+\de}\leq 2$ for any $\de \in (0,1)$. From \eqref{eq_est_mz} we get
\[
(1-\de)^\delta\iint_{\RR^{n+1}} \Upsilon(z)^\de\,dz \leq c (1-\de)^\de \Pizo^\de|\Q_2|.
\]
From \eqref{eq_p1} and the two inequalities above we obtain
\begin{align*}
\begin{split}
    \Pizo^\de|\Q_2| 
    &\leq c\iint_{\Q_1}H(z,|\na u|)^\de\,dz \\
    &\leq c\iint_{\Q_1} \frac{H(z,|\na u|)}{\mh(z)^{1-\de}}\,dz+ c(\theta^\de+(1-\de)^\de) \Pizo^\de|\Q_2|.
\end{split}
\end{align*}

Now we are ready to estimate the left hand side of the original claim in \eqref{claim_cac}. By the previous estimate and \eqref{after_A} we have 
\begin{align*}
    \begin{split}
        &\sup_{t\in\ti_1}\int_{\B_1\times \{t\}} \frac{|u-u_0|^2}{\mh^{1-\de}}\,dx + \Pizo^\de|\Q_2|\\
        &\leq        c\left( \frac{\rho}{\rho_2-\rho_1} \right)^2 \frac{1}{ (\theta\Pi)^{1-\de}}\iint_{\Q_2}\frac{|u-u_0|^2}{\sig_2}\,dz \\
        &\qquad+ c\left( \frac{\rho}{\rho_2-\rho_1} \right)^q \iint_{\Q_2}H\left(z,\frac{ |u-u_0|  }{\rho}\right)^\de\,dz +c  \iint_{\Q_2} H(z,|F|)^\delta\,dz\\
        &\qquad+c\left( \frac{\rho}{\rho_2-\rho_1} \right)^2   \left( \frac{1-\de}{\delta} \right)  \theta^\frac{2-q}{p}  \Pi_{z_0}^{\de}  |\Q_2| + c(\theta^\de+(1-\de)^\de) \Pizo^\de|\Q_2|.
    \end{split}
\end{align*}
Now, we use the assumption that $\rho_2-\rho_1\ge \rho$ so that 
\begin{align*}
    \begin{split}
        &\sup_{t\in\ti_1}\int_{\B_1\times \{t\}} \frac{|u-u_0|^2}{\mh^{1-\de}}\,dx + \Pizo^\de|\Q_2|\\
        &\leq        c\frac{1}{ (\theta\Pi)^{1-\de}}\iint_{\Q_2}\frac{|u-u_0|^2}{\sig_2}\,dz  + c\iint_{\Q_2}H\left(z, \frac{ | u-u_0|  }{\rho} \right)^\de\,dz +c  \iint_{\Q_2} H(z,|F|)^\delta\,dz\\
        &\qquad+c \left( \left( \frac{1-\de}{\delta} \right) \theta^\frac{2-q}{p} + \theta^\de+(1-\de)^\de \right) \Pizo^\de|\Q_2|.
    \end{split}
\end{align*}
Then, we take $\delta_1\in (\mu,1)$ to have for $\delta\in(\delta_1,1)$
\[
\left( \frac{1-\de}{\delta} \right) \theta^\frac{2-q}{p} + \theta^\de+(1-\de)^\de
\le   \mu^{-1}(1-\delta)  \theta^\frac{2-q}{p} + \theta^\mu+(1-\de)^\mu.
\]
Now, take $\theta$ small enough and then $\delta_1$ sufficiently close to $1$ depending on $\data$ to obtain 
\[
c\left( \left( \frac{1-\de}{\delta} \right) \theta^\frac{2-q}{p} + \theta^\de+(1-\de)^\de\right)\le \frac{1}{2}
\]
for any $\delta\in(\delta_1,1)$. Then we have
\begin{align*}
    \begin{split}
        &\sup_{t\in\ti_1}\int_{\B_1\times \{t\}} \frac{|u-u_0|^2}{\mh^{1-\de}}\,dx + \Pizo^\de|\Q_2|\\
        &\leq        c\iint_{\Q_2}\frac{|u-u_0|^2}{\Pi^{1-\delta}\sig_2}\,dz  + c\iint_{\Q_2}H\left(z, \frac{|u-u_0|}{\rho}\right)^\de\,dz +c  \iint_{\Q_2} H(z,|F|)^\delta\,dz.
    \end{split}
\end{align*}
Finally, recall $u_0=u_{\Q_2}$ and $\theta \Pi\le \mh$ and observe that
\begin{align*}
    \begin{split}
        \sup_{t\in\ti_1}\int_{\B_1\times \{t\}} \frac{|u-u_{\Q_1}|^2}{\mh^{1-\delta}}\,dx
        &\le 2\sup_{t\in\ti_1}\int_{\B_1\times \{t\}} \frac{|u-u_{\Q_2}|^2}{\mh^{1-\delta}}\,dx+ c|\B_1|\frac{|u_{\Q_1} - u_{\Q_2}|}{ \Pi^{1-\delta}}\\
        &\le 2\sup_{t\in\ti_1}\int_{\B_1\times \{t\}} \frac{|u-u_{\Q_2}|^2}{\mh^{1-\delta}}\,dx + c\iint_{\Q_1} \frac{|u-u_{\Q_2}|}{\Pi^{1-\delta}\sig_1}\,dz\\
        &\le 2\sup_{t\in\ti_1}\int_{\B_1\times \{t\}} \frac{|u-u_{\Q_2}|^2}{\mh^{1-\delta}}\,dx + c\iint_{\Q_2} \frac{|u-u_{\Q_2}|}{\Pi^{1-\delta}\sig_2}\,dz.
    \end{split}
\end{align*}
Substituting this estimate to the previous one and using the fact that $|\Q_1| \approx |\Q_2|$ up to
$c = c(n)$ finishes the proof.
\end{proof}

\section{Reverse Hölder inequality}\label{sec_rev_hölder}
In this section, we prove the reverse Hölder inequality.
The proof is based on parabolic Poincar\'e inequalities in intrinsic cylinders, the Caccioppoli inequality in Lemma~\ref{lem_caccioppoli}, and the following standard Sobolev-Gagliardo-Nirenberg inequality.
\begin{lemma}\label{lem_sobo}
	Let $B_{\rho}(x_0)\subset\RR^n$, $\sig,s,r\in[1,\infty)$ and $\vartheta\in(0,1)$ be such that 
	\begin{align*}
		-\frac{n}{\sig}\le \vartheta\left(1-\frac{n}{s}\right)-(1-\vartheta)\frac{n}{r}.
	\end{align*}
	Then there exists a constant $c=c(n,\sig)$ such that
	\begin{align*}
		\fint_{B_{\rho}(x_0)}\frac{|v|^\sig}{\rho^\sig}\,dx
		\le c\left(\fint_{B_{\rho}(x_0)}\left(\frac{|v|^s}{\rho^s}+|\na v|^s\right)\,dx\right)^\frac{\vartheta \sig}{s}\left(\fint_{B_{\rho}(x_0)}\frac{|v|^r}{\rho^r}\,dx\right)^\frac{(1-\vartheta)\sig}{r}
	\end{align*}
  for every $v\in W^{1,s}(B_{\rho}(x_0))$.
\end{lemma}
 
 For the simplicity of calculations, we redefine $\delta_1=\max\{\delta_1,\tfrac{2}{p}\}$ so that for any $\delta\in(\delta_1,1)$, we have
\[
\delta p>2.
\]
Indeed, the calculations for the case $\delta p\le 2$ is similar to the singular case in this context. Moreover, we will assume that $\rho\in(0,R_0)$, where  $R_0\in (0,1)$ depending on $\data$, $\delta$, $\| H(z,|\na u|) \|_{L^\delta(\Om_T)}$, $\| H(z,|F|) \|_{L^\delta(\Om_T)}$ and $\|u\|_{C(0,T;L^2(\Omega,\mathbb{R}^N))}$ will be chosen in Lemma~\ref{lem_KKM_4_3}, Lemma~\ref{lemma_decay} and Lemma~\ref{p_q_comp}.

\subsection{The $p$-intrinsic case}
We begin by proving a parabolic Poincar\'e inequality in the $p$-intrinsic case.
First we estimate the last term in Lemma~\ref{lem_parabolic_poincare}.
\begin{lemma}\label{sec4:lem:1}
	Let $u$ be a very weak solution to \eqref{11}. Then, for $\theta\in((q-1)/p,\de]$ and $s\in[2\rho,4\rho]$, 
	there exists a constant $c=c(\data)$ such that
	\begin{align*}
		\begin{split}
			&\fiint_{Q_{s}^\pi(z_0)}(|\na u|^{p-1}+a(z)|\na u|^{q-1}+|F|^{p-1} + a(z) |F|^{q-1} )\,dz\\
			&\le c\fiint_{Q_{s}^\pi(z_0)} (|\na u| + |F| )^{p-1}\,dz\\
            &\qquad +c\pi^{-1+\frac{p}{q}}\fiint_{Q_{s}^\pi(z_0)}a(z)^\frac{q-1}{q} ( |\na u|^{q-1} + |F|^{q-1} )\,dz\\
		&\qquad+c \left(\fiint_{Q_{s}^\pi(z_0)}  (  |\na u| + |F|   )^{\theta p}\,dz\right)^{\frac{p-1}{\theta p}},
		\end{split}
	\end{align*}
 whenever $Q_{16\rho}^{\pi}(z_0)\subset\Omega_T$ satisfies \eqref{eq_pcase}-\eqref{eq_p2}.
\end{lemma}

\begin{proof}
	We estimate the term involving $|\na u|$ first.
    It follows from \eqref{eq_range_q} that $q-1<p$. By \eqref{eq_holder_reg} there exists a  constant $c=c([a]_{\alpha})$ such that
	\begin{equation}\label{sec4:41}
		\begin{split}
			&\fiint_{Q_{s}^\pi(z_0)}(|\na u|^{p-1}+a(z)|\na u|^{q-1})\,dz\\
			&\le c\fiint_{Q_{s}^\pi(z_0)}|\na u|^{p-1}\,dz+c\fiint_{Q_{s}^\pi(z_0)}\inf_{w\in Q_{s}^\pi(z_0)}a(w)|\na u|^{q-1}\,dz\\
			&\qquad+cs^\alpha\fiint_{Q_{s}^\pi(z_0)}|\na u|^{q-1}\,dz.
		\end{split}
	\end{equation}
	We apply \eqref{eq_pcase} to estimate the second term on the right-hand side of \eqref{sec4:41} and obtain
	\begin{align*}
		\begin{split}
			&\fiint_{Q_{s}^\pi(z_0)}\inf_{w\in Q_{s}^\pi(z_0)}a(w)(|\na u|+|F|)^{q-1}\,dz\\
			&\le \pi^{-1+\frac{p}{q}}\fiint_{Q_{s}^\pi(z_0)}\inf_{w\in Q_{s}^\pi(z_0)}a(w)^\frac{q-1}{q}(|\na u|+|F|)^{q-1}\,dz\\
			&\le \pi^{-1+\frac{p}{q}}\fiint_{Q_{s}^\pi(z_0)}a(z)^\frac{q-1}{q}(|\na u|+|F|)^{q-1}\,dz.
		\end{split}
	\end{align*}
	In order to estimate the last term on the right-hand side of \eqref{sec4:41}, we have by \eqref{eq_rho_decay} and Hölder's inequality that
    \begin{align*}
        \begin{split}
            s^\alpha\fiint_{Q_{s}^\pi(z_0)}|\na u|^{q-1}\,dz 
            &\leq \pi^{p-q} \left(\fiint_{Q_{s}^\pi(z_0)}|\na u|^{\theta p}\,dz\right)^\frac{q-1}{\theta p}\\
            &\leq \pi^{p-q} \left(\fiint_{Q_{s}^\pi(z_0)}|\na u|^{\delta p}\,dz\right)^{\frac{q-p}{\delta p}}  \left(\fiint_{Q_{s}^\pi(z_0)}|\na u|^{\theta p}\,dz\right)^{\frac{p-1}{\theta p}}\\
            &\le c \left(\fiint_{Q_{s}^\pi(z_0)}|\na u|^{\theta p}\,dz\right)^\frac{p-1}{\theta p}.
        \end{split}
    \end{align*}
    The same calculations hold for the term involving $|F|$.
This completes the proof.
\end{proof}

\begin{lemma}\label{sec4:lem:2}
	Let $u$ be a very weak solution to \eqref{11}. Then, for $\theta\in((q-1)/p,\de]$ and $s\in[2\rho,4\rho]$, 
	there exists a constant $c=c(\data)$ such that 
	\begin{equation*}
		\begin{split}
			\fiint_{Q_{s}^\pi(z_0)}\frac{|u-u_{Q_{s}^\pi(z_0)}|^{\theta p}}{s^{\theta p}}\,dz
			\le c\fiint_{Q_{s}^\pi(z_0)} ( H(z,|\na u|)+ H(z,|F|) )^\theta\,dz
		\end{split}
	\end{equation*}
whenever $Q_{16\rho}^{\pi}(z_0)\subset\Omega_T$ satisfies \eqref{eq_pcase}-\eqref{eq_p2}.
\end{lemma}

\begin{proof}
	By Lemma~\ref{lem_parabolic_poincare} and Lemma~\ref{sec4:lem:1}, there exists a constant $c=c(\data)$ such that
	\begin{align}\label{sec4:42}
		\begin{split}
			&\fiint_{Q_{s}^\pi(z_0)}\frac{|u-u_{Q_{s}^\pi(z_0)}|^{\theta p}}{s^{\theta p}}\,dz \\
			&\le c\fiint_{Q_{s}^\pi(z_0)}|\na u|^{\theta p}\,dz +c\left(\pi^{2-p}\fiint_{Q_{s}^\pi(z_0)}|\na u|^{p-1}  + |F|^{p-1}  \,dz\right)^{\theta p}\\
			&\qquad+c\left(\pi^{1-p+\frac{p}{q}}\fiint_{Q_{s}^\pi(z_0)}a(z)^\frac{q-1}{q} (|\na u|  +  |F|  )^{q-1}\,dz\right)^{\theta p}\\
			&\qquad+c\pi^{ (2-p )\theta p}\left(\fiint_{Q_{s}^\pi(z_0)}  ( |\na u| +  |F|   )^{\theta p}\,dz\right)^{p-1}.
		\end{split}
	\end{align}

    We estimate the terms involving $|\na u|$ first.
	To estimate the second term on the right-hand side of \eqref{sec4:42}, we use \eqref{eq_p2} and Hölder's inequality to obtain
	\begin{align*}
		\begin{split}
		    \pi^{(2-p)\theta p}\left(\fiint_{Q_{s}^\pi(z_0)}|\na u|^{p-1}\,dz\right)^{\theta p}
            &\le \pi^{(2-p)\theta p} \left( \fiint_{Q_{s}^\pi(z_0)} |\na u|^{\theta p}\,dz \right)^{p-1} \\
            &\le c\fiint_{Q_{s}^\pi(z_0)}H(z,|\na u|)^\theta\,dz.
		\end{split}
	\end{align*}
where $c=c(n,p)$.
Similarly, the third term on the right-hand side of \eqref{sec4:42} is estimated with \eqref{eq_p2} and Hölder's inequality to obtain
\begin{align*}
		& \pi^{\left(1-p+\frac{p}{q}\right)\theta p}\left(\fiint_{Q_{s}^\pi(z_0)}a(z)^\frac{q-1}{q}|\na u|^{q-1}\,dz\right)^{\theta p}\\
        &\le \pi^{\left(1-p+\frac{p}{q}\right)\theta p}\left(\fiint_{Q_{s}^\pi(z_0)} a(z)^{\theta}|\na u|^{\theta q}\,dz\right)^{\frac{(q-1)p}{q}} \\
        & \leq \pi^{\left(1-p+\frac{p}{q}\right)\theta p}\left(\fiint_{Q_{s}^\pi(z_0)}a(z)^\delta |\na u|^{\delta q}\,dz\right)^{ \left( \frac{(q-1)p}{q}-1 \right) \frac{\theta}{\delta} }  \left(\fiint_{Q_{s}^\pi(z_0)}a(z)^\theta|\na u|^{\theta q}\,dz\right)\\
		& \le c\fiint_{Q_{s}^\pi(z_0)}H(z,|\na u|)^\theta\,dz,
\end{align*}
where $c=c(n,p,q)$. The same calculations work with the terms involving $|F|$.
This finishes the proof.
\end{proof}

\begin{lemma}\label{sec4:lem:3}
	Let $u$ be a very weak solution to \eqref{11}. Then, for $\theta\in((q-1)/p,\de]$ and $s\in[2\rho,4\rho]$ and , 
	there exists a constant $c=c(\data)$ such that 
	\begin{equation*}
		\begin{split}
			\fiint_{Q_{s}^\pi(z_0)}\inf_{w\in Q_{s}^\pi(z_0)}a(w)^{\theta}\frac{|u-u_{Q_{s}^\pi(z_0)}|^{\theta q}}{s^{\theta q}}\,dz
			&\le c\fiint_{Q_{s}^\pi(z_0)}   (H(z,|\na u|)+H(z,|F|)  )^\theta\,dz.
		\end{split}
	\end{equation*}
    whenever $Q_{16\rho}^{\pi}(z_0)\subset\Omega_T$ satisfies \eqref{eq_pcase}-\eqref{eq_p2}.
\end{lemma}

\begin{proof}
		By Lemma~\ref{lem_parabolic_poincare} and Lemma~\ref{sec4:lem:1} there exists a constant $c=c(\data)$ such that
	\begin{equation}\label{sec4:43}
		\begin{split}
			&\fiint_{Q_{s}^\pi(z_0)}\inf_{w\in Q_{s}^{\pi}(z_0)}a(w)^{\theta}\frac{|u-u_{Q_{s}^\pi(z_0)}|^{\theta q}}{s^{\theta q}}\,dz \\
			& \le c\fiint_{Q_{s}^\pi(z_0)}\inf_{w\in Q_{s}^\pi(z_0)}a(w)^{\theta}|\na u|^{\theta q}\,dz\\
			&\qquad+ c\inf_{w\in Q_{s}^{\pi}(z_0)}a(w)^{\theta}\left(\pi^{2-p}\fiint_{Q_{s}^\pi(z_0)}  ( |\na u|  + |F|   )^{p-1}\,dz\right)^{\theta q}\\
			&\qquad+c\inf_{w\in Q_{s}^{\pi}(z_0)}a(w)^{\theta}\left(\pi^{1-p+\frac{p}{q}}\fiint_{Q_{s}^\pi(z_0)}a(z)^\frac{q-1}{q} ( |\na u|  + |F|   )^{q-1}\,dz\right)^{\theta q}\\
			&\qquad+c\inf_{w\in Q_{s}^\pi(z_0)}a(w)^{\theta}\pi^{ ( 2-p )\theta q}\left(\fiint_{Q_{s}^\pi(z_0)}  ( |\na u|  + |F|   )^{\theta p}\,dz\right)^{ \frac{q(p-1)}{p} }.
		\end{split}
	\end{equation}
     It is easy to see that
     \begin{align*}
         \begin{split}
             \fiint_{Q_{s}^\pi(z_0)}\inf_{w\in Q_{s}^\pi(z_0)}a(w)^{\theta}|\na u|^{\theta q}\,dz
             \le \fiint_{Q_{s}^\pi(z_0)} H(z,|\na u|)^\theta \,dz
         \end{split}
     \end{align*}
For the rest of terms, we estimate the integrals with $|\na u|$ first. 
	By \eqref{eq_pcase} we have $a(z_0)\leq \pi^{p-q}$ and we obtain with Hölder's inequality and \eqref{eq_p2} that
\begin{align*}
	\begin{split}
    &\inf_{w\in Q_{s}^{\pi}(z_0)}a(w)^{\theta}\left(\pi^{2-p}\fiint_{Q_{s}^\pi(z_0)}|\na u|^{p-1}\,dz\right)^{\theta q}\\
		& \leq \inf_{w\in Q_{s}^\pi(z_0)}a(w)^{\theta}\pi^{(2-p)\theta q}\left(\fiint_{Q_{s}^\pi(z_0)}|\na u|^{\theta p}\,dz\right)^{\frac{(p-1)q}{p}}\\
		&\le \pi^{(p-q)\theta}\pi^{(2-p)\theta q+(p-1)q\theta-\theta p}\fiint_{Q_{s}^\pi(z_0)}|\na u|^{\theta p}\,dz\\
		&\le c\fiint_{Q_{s}^\pi(z_0)}|\na u|^{\theta p}\,dz,
	\end{split}
\end{align*}
where we used the fact that
\[
\frac{(p-1)q}{p}> p-1>0.
\]
The same calculations hold for the fourth term  on the right-hand side  of \eqref{sec4:43}.
The third one can be estimated similarly as
\begin{equation*}
	\begin{split}
        &\inf_{w\in Q_{s}^{\pi}(z_0)}a(w)^{\theta}\left(\pi^{1-p+\frac{p}{q}}\fiint_{Q_{s}^\pi(z_0)}a(z)^\frac{q-1}{q}|\na u|^{q-1}\,dz\right)^{\theta q}\\
		&\leq\inf_{w\in Q_{s}^\pi(z_0)}a(w)^{\theta}\pi^{(p+q-pq)\theta }\left(\fiint_{Q_{s}^\pi(z_0)}(a(z)|\na u|^q)^{\theta}\,dz\right)^{q-1}\\
		&\le c\fiint_{Q_{s}^\pi(z_0)}(a(z)|\na u|^q)^{\theta}\,dz.
	\end{split}
\end{equation*}
The analogous calculations work for the terms with $|F|$.
This finishes the proof.
\end{proof}

Recall from \eqref{def_mu} and \eqref{def_mu_2} that $\mu\in(0,1)$ satisfies
\[
\max\left\{  1-\frac{2}{q}+\frac{(n+2)(q-p)}{q\alpha} , \frac{q-1}{p} \right\}<\mu<1
\]
and
\[
\left( \frac{1}{\mu} -1 \right)\frac{2n}{(n+2)q} +\frac{n}{n+2} \le \frac{p}{q}.
\]

\begin{lemma}\label{lem_KKM_4_3}
    Let $u$ be a very weak solution to \eqref{11}. There exist constants $c=c(\data)$ and $R_0=R_0(\data, \delta ,\| u\|_{C_{\loc}(0,T;L_{\loc}^2(\Om,\mathbb{R}^N))} )\in (0,1)$ such that
    \begin{align*}
		S(u,Q_{2\rho}^\pi(z_0))=\sup_{I_{2\rho}^\pi(t_0)}\fint_{B_{2\rho}(x_0)}\frac{|u-u_{Q_{2\rho}^\pi}|^2}{(2\rho)^2\mh^{1-\de}}\,dx\le c\pi^{2-p(1-\de)},
	\end{align*}
    whenever $Q_{16\rho}^{\pi}(z_0)\subset\Omega_T$ satisfies \eqref{eq_pcase}-\eqref{eq_p2} with $\rho\in(0,R_0)$. Here  $\mh$ is defined as in \eqref{def_mh} with $\Q_1 = Q_{2\rho}^\pi(z_0)$ and $\Q_2 = Q_{4\rho}^\pi(z_0)$.
\end{lemma}

\begin{proof}
    By Lemma~\ref{lem_caccioppoli} there exists a constant $c=c(\data)$ such that 
	\begin{align}\label{sec5:3}
		\begin{split}
			&\pi^{p-2}S(u,Q_{2\rho}^\pi(z_0))\\
			&\le c\fiint_{Q_{4\rho}^\pi(z_0)} \frac{|u-u_{Q_{4\rho}^\pi(z_0)}|^2}{\Pi^{1-\de}\pi^{2-p}\rho^2 } \,dz\\
             &\qquad + c\fiint_{Q_{4\rho}^\pi(z_0)}\left( \frac{|u-u_{Q_{4\rho}^\pi(z_0)}|^{\de p}}{\rho^{\de p}}+a(z)^\de \frac{|u-u_{Q_{4\rho}^\pi(z_0)}|^{\de q}}{\rho^{\de q}}\right)\,dz\\
             &\qquad + c\fiint_{Q_{4\rho}^\pi(z_0)} H(z,|F|)^\delta\,dz.
		\end{split}
	\end{align}
  We have by Lemma~\ref{sec4:lem:2} and \eqref{eq_p2}, that
	\begin{align*}
    \begin{split}
        \fiint_{Q_{4\rho}^\pi(z_0)}\frac{|u-u_{Q_{4\rho}^\pi(z_0)}|^{\de p}}{\rho^{\de p}}\,dz &\leq c\fiint_{Q_{4\rho}^\pi(z_0)} H(z,|\na u|)^\delta\,dz \\
        &\qquad + c\fiint_{Q_{4\rho}^\pi(z_0)} H(z,|F|)^\delta\,dz\\
        &\le c\pi^{\delta p},
    \end{split}
	\end{align*}
 where $c=c(data)$. Moreover, as we have $\de p> 2$, it follows from Hölder's inequality and the above estimate that 
    \[
    \fiint_{Q_{4\rho}^\pi(z_0)}\frac{|u-u_{Q_{4\rho}^\pi(z_0)}|^2}{\Pi^{1-\de}\pi^{2-p}\rho^2}\,dz\le \pi^{\de p-2}\left(\fiint_{Q_{4\rho}^\pi(z_0)}\frac{|u-u_{Q_{4\rho}^\pi(z_0)}|^{\de p}}{ \rho^{\de p}}\,dz\right)^\frac{2}{\de p} \le c\pi^{\de p},
    \]
where we used the fact that
\[
 \frac{1}{\Pi^{1-\delta}\pi^{2-p}}\le \frac{1}{\pi^{p(1-\delta)}\pi^{2-p}}=\pi^{\delta p-2}.
\]
For the third term on the right hand side of \eqref{sec5:3}, we apply \eqref{eq_holder_reg} and triangle inequality to obtain
    \begin{align*}
    \begin{split}
        \fiint_{Q_{4\rho}^\pi(z_0)}a(z)^\de \frac{|u-u_{Q_{4\rho}^\pi(z_0)}|^{\de q}}{ \rho^{\de q}}\,dz   &\leq c\fiint_{Q_{4\rho}^\pi(z_0)}\inf_{w\in Q_{4\rho}^\pi(z_0)}a(w)^\de \frac{|u-u_{Q_{4\rho}^\pi(z_0)}|^{\de q}}{ \rho^{\de q}}\,dz \\
        &\qquad +c\fiint_{Q_{4\rho}^\pi(z_0)}\rho^{\de \al} \frac{|u-u_{Q_{4\rho}^\pi(z_0)}|^{\de q}}{ \rho^{\de q}}\,dz.
    \end{split}
    \end{align*}
     By Lemma~\ref{sec4:lem:3} and \eqref{eq_p2} we have that
     \[
     \fiint_{Q_{4\rho}^\pi(z_0)}\inf_{w\in Q_{4\rho}^\pi(z_0)}a(w)^\de \frac{|u-u_{Q_{4\rho}^\pi(z_0)}|^{\de q}}{(4\rho)^{\de q}}\,dz \leq  c\pi^{\de p}.
     \]
    To estimate the second term, we use Lemma~\ref{lem_sobo} with $\sigma =  q$, $s = \de p$, $r=2$ and $\vartheta = p/ q$. Indeed, we have
    \begin{align*}
        \begin{split}
            -\frac{n}{q}\le \frac{ p}{q}\left( 1-\frac{n}{\delta p} \right)- \left( 1-\frac{ p}{q} \right)\frac{n}{2} 
            &\iff \left( \frac{1}{\delta} -1 \right)\frac{n}{q} +\frac{n}{2} \le \frac{(n+2)p}{2q}\\
            &\iff \left( \frac{1}{\delta} -1 \right)\frac{2n}{(n+2)q} +\frac{n}{n+2} \le \frac{p}{q},
        \end{split}
    \end{align*}
    where we used \eqref{def_mu_2} and the fact that $\mu<\delta$.
    This gives 
    \begin{align*}
        \begin{split}
            \rho^{\de \al}\fiint_{Q_{4\rho}^\pi(z_0)} \frac{|u-u_{Q_{4\rho}^\pi(z_0)}|^{\de q}}{\rho^{\de q}}\,dz
            &\le \rho^{\de \al} \fint_{I_{4\rho}^\pi(t_0) }\left( \fint_{B_{4\rho}(x_0)} \frac{|u-u_{Q_{4\rho}^\pi(z_0)}|^{ q}}{\rho^{ q}}\,dx \right)^\delta \,dt \\
            &\leq c \rho^{\de \al}\fiint_{Q_{4\rho}^\pi(z_0)} \left(\frac{|u-u_{Q_{4\rho}^\pi(z_0)}|^{\de p}}{\rho^{\de p}}+|\na u|^{\de p}\right)\,dz \\
            &\qquad \times \left(\sup_{I^\pi_{4\rho}(t_0)}\fint_{B_{4\rho}(x_0)} \frac{|u-u_{Q_{4\rho}^\pi(z_0)}|^2}{\rho^2}\,dx\right)^\frac{\de(q-p)}{2}.
        \end{split}
    \end{align*}
    Observe that by the triangle inequality we have
    \begin{align*}
        \begin{split}
            &\rho^{\de\al}\left(\sup_{I^\pi_{4\rho}(t_0)}\fint_{B_{4\rho}(x_0)} \frac{|u-u_{Q_{4\rho}^\pi(z_0)}|^2}{\rho^2}\,dx\right)^\frac{\de(q-p)}{2} \\
        &\leq c\rho^{\de\al}\left(\sup_{I^\pi_{4\rho}(t_0)}\fint_{B_{4\rho}(x_0)} \frac{|u|^2}{\rho^2}\,dx\right)^\frac{\de(q-p)}{2} \\
        &= c\rho^{\delta \left( \alpha-\frac{(q-p)(n+2)}{2} \right) } \left(\sup_{I^\pi_{4\rho}(t_0)}\int_{B_{4\rho}(x_0)} |u|^2 \,dx\right)^\frac{\de(q-p)}{2}.
        \end{split}
    \end{align*}
Note that since we have by \eqref{eq_range_q} that
\[
\alpha -\frac{(q-p)(n+2)}{2} >0
\]
and $\rho\le R_0$, we may take $R_0\in(0,1)$ small enough so that
\[
R_0^{\mu \left( \alpha-\frac{(q-p)(n+2)}{2} \right) } \|u\|_{C_{\loc}(0,T;L^2_{\loc}(\Om,\mathbb{R}^N))}^{\de(q-p)}\le \frac{1}{8c}.
\]
Therefore, we get with Lemma~\ref{sec4:lem:2} and \eqref{eq_p2} and the above inequalities that 
     \begin{align}\label{radius_absorb_1}
         c\rho^{\de \al}\fiint_{Q_{4\rho}^\pi(z_0)} \frac{|u-u_{Q_{4\rho}^\pi(z_0)}|^{\de q}}{(4\rho)^{\de q}}\,dz \leq \frac{1}{8}\pi^{\de p }.
     \end{align}
     Finally, \eqref{eq_p2} gives
     \[
     \fiint_{Q_{4\rho}^\pi(z_0)} H(z,|F|)^\delta\,dz\le \pi^{\delta p}
     \]
We conclude from \eqref{sec5:3} that
\[
S(u,Q_{2\rho}^\pi(z_0)) \leq c\pi^{2-p(1-\de))},
\]
which finishes the proof.
\end{proof}

We remark that we have calculated the following in the proof of the previous lemma.
\begin{lemma}\label{coro_p_phase}
    Let $u$ be a very weak solution to \eqref{11}. Then for $\theta\in(\mu,\de]$ and $s\in[2\rho,4\rho]$, there exist constants $c=c(\data)$ and $R_0=R_0(\data, \delta ,\| u\|_{C_{\loc}(0,T;L_{\loc}^2(\Om,\mathbb{R}^N))} )\in (0,1)$ such that
	\begin{equation*}
		\begin{split}
			\fiint_{Q_{s}^\pi(z_0)} H \left(  z, \frac{|u-u_{Q_{s}^\pi(z_0)}| }{s}  \right)^\theta \,dz
			&\le c\fiint_{Q_{s}^\pi(z_0)}H(z,|\na u|)^\theta\,dz\\
			&\qquad+ c\left( \fiint_{Q_{s}^\pi(z_0)}  H(z,|F|)^\delta  \,dz  \right)^\frac{\theta}{\delta}
		\end{split}
	\end{equation*}
    whenever $Q_{16\rho}^{\pi}(z_0)\subset\Omega_T$ satisfies \eqref{eq_pcase}-\eqref{eq_p2} with $\rho\in(0,R_0)$.
\end{lemma}

Now we are ready to prove the reverse Hölder inequality. We denote upper level sets with respect to $|\na u|$ and $|F|$ as
 \begin{align}\label{sec2:2}
\begin{split}
    \Psi(\La)&=\{ z\in \Omega_T: H(z,|\na u(z)|)>\La\},\\
    \Phi(\La)&=\{ z\in \Omega_T: H(z,| F (z)|)>\La\}.
\end{split}
\end{align}

\begin{lemma}\label{lem_reverse_Hölder}
Let $u$ be a very weak solution to \eqref{11}.
There exist constants $c=c(\data)$, $\delta_2=\delta_2(\data)\in(\delta_1,1)$ with $2-q(1-\delta_2)>0$, $\theta_0=\theta_0(\data)\in(0,\delta_2) $, $R_0=R_0(\data, \delta ,\| u\|_{C_{\loc}(0,T;L_{\loc}^2(\Om,\mathbb{R}^N))} )\in (0,1)$ such that
	\begin{align*}
		\begin{split}
			\pi^{\de p}
            \le c\left(\fiint_{Q_{2\rho}^\pi(z_0)}H(z,|\na u|)^{\theta_0}\,dz\right)^\frac{\delta}{\theta_0}+c\fiint_{Q_{2\rho}^\pi(z_0)} H(z,|F|)^\delta  \,dz,
		\end{split}
	\end{align*}	
	whenever $Q_{16\rho}^{\pi}(z_0)\subset\Omega_T$ satisfies \eqref{eq_pcase}-\eqref{eq_p2} with $\rho\in(0,R_0)$ and $\delta\in (\delta_2,1)$.
    Moreover, we have 
    \begin{align}\label{eq_revH_Psi}
		\begin{split}
			\iint_{Q_{16\rho}^\pi(z_0)}H(z,|\na u|)^\delta\,dz
			&\le c\Pi^{\delta-\theta_0}\iint_{Q_{2\rho}^\pi(z_0)\cap \Psi(c^{-1}\Pi)}H(z,|\na u|)^{\theta_0}\,dz\\
            &\qquad +c \iint_{Q_{2\rho}^\pi(z_0)\cap \Phi(c^{-1}\Pi)}H(z,|F|)^\delta \,dz,
		\end{split}
	\end{align}
 where $\Psi(\cdot)$ and $\Phi(\cdot)$ are defined in \eqref{sec2:2}.
\end{lemma}

\begin{proof}
For simplicity, we write $Q=Q_{2\rho}^\pi(z_0)$, $B=B_{2\rho}(x_0)$, $I=I^\pi_{2\rho}(t_0)$.
To show the first statement, we use the Caccioppoli inequality in Lemma~\ref{lem_caccioppoli} and  \eqref{eq_p1} to get
\begin{align}\label{eq_915}
	\begin{split}
	\pi^{\de p} 
    &\leq c\fiint_{Q}  \frac{|u-u_{Q}|^2}{\pi^{2-p}\Pi^{1-\de}\rho^2 }+c\fiint_{Q}\left( \frac{ |u-u_{Q}|^{\de p} }{\rho^{\de p}}+a(z)^\de\frac{ |u-u_{Q }|^{\de q}  }{\rho^{\de q}}\right)\,dz\\
    &\qquad +c\fiint_{Q} H(z,|F|)^\delta  \,dz.
	\end{split}
\end{align}
We will estimate the second term on the right hand side 
first. Before proceeding, we take $\sdel=\sdel(n,p)\in(\tfrac{n}{p+n},1)$ to satisfy
\[
\frac{n}{2}< \frac{(\sdel)^2}{1-\sdel}.
\]
 Note that for $\sig=p$, $s=\sdel p$, $r^*=2\sdel$, there exists $\vartheta\in(0,1)$ sufficiently close to 1 satisfying
 \[
 -\frac{n}{p}\le \vartheta \left( 1-\frac{n}{\sdel p}\right) -(1-\vartheta) \frac{n}{2\sdel}.
 \]
 Indeed, the above display is equivalent to
 \begin{align}\label{vartheta_choice}
     \frac{n}{2\sdel}-\frac{n}{p} \le \vartheta \left(   1-\frac{n}{\sdel p} +\frac{n}{2\sdel}  \right)
 \end{align}
 and we have
 \[
 \frac{n}{2\sdel}-\frac{n}{p}<    1-\frac{n}{\sdel p} +\frac{n}{2\sdel},
 \]
  since we have taken $\sdel \in ( \tfrac{n}{n+p},1 )$. Therefore, we can find $\vartheta\in(0,1)$ satisfying \eqref{vartheta_choice}. We will take $\vartheta=\tfrac{\sdel}{\delta}$ for $\delta\in(\delta_2,1)$ where $\delta_2\in(\sdel,1)$ will be chosen later. For this, we consider the extreme case in \eqref{vartheta_choice} when $\vartheta$ is sufficiently close to $\sdel$. We observe 
  \begin{align*}
  \frac{n}{2\sdel}-\frac{n}{p} < \sdel \left(   1-\frac{n}{\sdel p} +\frac{n}{2\sdel}  \right) \Longleftrightarrow \frac{n}{2} < \frac{(\sdel)^2}{1-\sdel},
  \end{align*}
  where we have chosen $\sdel$ so that the second inequality is true. Since $\tfrac{\sdel}{\delta}$ is decreasing function in $\delta$, we take $\delta_2\in(\max\{\delta_1,\sdel\},1)$ so that
  \[
  \frac{n}{2\sdel}-\frac{n}{p} \le \frac{\sdel}{\delta_2} \left(   1-\frac{n}{\sdel p} +\frac{n}{2\sdel}  \right),\qquad 0< 2-q(1-\delta_2).
  \]
  Furthermore, for our purpose later, we also take $\delta_2\in(0,1)$ sufficiently close to $1$ to satisfy
  \begin{align}\label{def_del_2_2}
      \frac{2\delta^*}{\delta_2}< 2-q(1-\delta^*)(1-\delta_2).
  \end{align}
  
Now, we apply Lemma~\ref{lem_sobo} with $ \sig=p$, $s=\sdel p$, $r^*=2\sdel$ and $\vartheta=\tfrac{\sdel}{\delta}$ for $\delta\in(\delta_2,1)$ to the second term on the right hand side of \eqref{eq_915}. Denoting $r=2\delta$, we have along with the fact $\sdel p= \vartheta \delta p$ that
	\begin{align*}
    \begin{split}
        \fiint_{Q }\frac{|u-u_{Q  }|^{\de p}}{\rho^{\de p}}\,dz
        &\le \fint_{I  }\left( \fint_{B  }\frac{|u-u_{Q  }|^{ p}}{\rho^{ p}}\,dx\right)^{\delta} \, dt\\
        & \le c\fint_{I  }\left(\fint_{B }\left(\frac{|u-u_{Q }|^{\vartheta (\de p)}}{\rho^{\vartheta (\de p)}}+|\na u|^{\vartheta (\de p)}\right)\,dx\right. \\
        &\qquad\times \left. \left(\fint_{B }\frac{|u-u_{Q }|^r}{\rho^r}\,dx\right)^{\frac{(1-\vartheta)(\de p)}{r}}\right)\,dt,
    \end{split}
	\end{align*}
    where $c=c(n,N,p)$.
Moreover, recall that $2-q(1-\delta)>0$.
We apply H\"older's inequality with $\nu=\tfrac{2}{2-p(1-\vartheta)(1-\de)}>1$ and its conjugate $\nu'=\tfrac{2}{p(1-\vartheta)(1-\de)}$, to obtain
	\begin{align*}
    \begin{split}
        &\fiint_{Q }\frac{|u-u_{Q }|^{\de p}}{\rho^{\de p}}\,dz\\
        &\le c\fint_{I } \left(\fint_{B }\left(\frac{|u-u_{Q }|^{\vartheta (\de p)}}{ \rho^{\vartheta (\de p)}}+|\na u|^{\vartheta (\de p)}\right)\,dx\left(\fint_{B }\frac{|u-u_{Q }|^r}{\rho^r}\,dx\right)^{\frac{(1-\vartheta)(\de p)}{r}}\right)\,dt\\
        &\le c\left(\fint_{I }\left(\fint_{B}\left(\frac{|u-u_{Q }|^{\vartheta (\de p)}}{ \rho^{\vartheta (\de p)}}+|\na u|^{\vartheta (\de p)}\right)\,dx\right)^\nu\,dt\right)^\frac{1}{\nu}\\
        &\qquad\times \left(\fint_{I }\left(\fint_{B }\frac{|u-u_{Q }|^r}{\rho^r}\,dx\right)^{\frac{\nu'(1-\vartheta)(\de p)}{r}}\,dt\right)^\frac{1}{\nu'}.
    \end{split}
	\end{align*}
Again, applying H\"older's inequality, we have    
\begin{align*}
    \begin{split}
        \fiint_{Q }\frac{|u-u_{Q }|^{\de p}}{\rho^{\de p}}\,dz
        & \le c\left(\fiint_{Q }\left(\frac{|u-u_{Q }|^{\vartheta\nu (\de p)}}{ \rho^{\vartheta\nu (\de p)}}+|\na u|^{\vartheta\nu (\de p)}\right)\,dz\right)^\frac{1}{\nu} \\
        &\qquad\times \left(\fint_{I}\left(\fint_{B}\frac{|u-u_{Q }|^r}{\rho^r}\,dx\right)^{\frac{\nu'(1-\vartheta)(\de p)}{r}}\,dt\right)^\frac{1}{\nu'}.
    \end{split}
	\end{align*}
    Here, note that $ \vartheta\nu<1$ for any $\vartheta < 1$ as there holds
\[
\vartheta < \frac{1}{\nu} \iff (2-p(1-\delta))\vartheta < 2-p(1-\de).
\]
Therefore, we employ Lemma~\ref{sec4:lem:2} to have
\begin{align*}
    \begin{split}
        \fiint_{Q }\frac{|u-u_{Q }|^{\de p}}{\rho^{\de p}}\,dz
        & \le c\left(\fiint_{Q }  \left( H(z,|\na u|)+ H(z,|F|) \right)^{\vartheta\nu \delta } \,dz\right)^\frac{1}{\nu} \\
        &\qquad\times \left(\fint_{I}\left(\fint_{B}\frac{|u-u_{Q }|^r}{\rho^r}\,dx\right)^{\frac{\nu'(1-\vartheta)(\de p)}{r}}\,dt\right)^\frac{1}{\nu'}.
    \end{split}
\end{align*}
To estimate the second integral, we recall that $r=2\delta$ and apply Hölder's inequality with the conjugate pair $(\tfrac{1}{\de}, \tfrac{1}{1-\de})$ to obtain
\begin{align*}
    \begin{split}
        \fint_{B}\frac{|u-u_{Q}|^r}{\rho^r}\,dx 
        &= \fint_{B}\frac{|u-u_{Q}|^{2\de}}{\rho^{2\de}\mh^{\de(1-\de)}}\mh^{\de(1-\de)}\,dx\\
        &\leq \left(\fint_{B }\frac{|u-u_{Q }|^2}{ \rho^2\mh^{1-\de}}\,dx\right)^\de\left(\fint_{B }\mh^\de\,dx\right)^{1-\de},
    \end{split}
\end{align*}
where $\mh$ is defined in \eqref{def_mh}. Therefore, we have 
\begin{align*}
    \begin{split}
    &\fint_{I}\left(\fint_{B}\frac{|u-u_{Q}|^r}{\rho^r}\,dx\right)^\frac{\nu'(1-\vartheta)\de p}{r}\,dt\\ 
    &\leq \fint_{I}\left(\left(\fint_{B}\frac{|u-u_{Q}|^2}{\rho^2\mh^{1-\de}}\,dx\right)^\frac{\nu'(1-\vartheta)\de p}{2}\left(\fint_{B}\mh^\de\,dx\right)^\frac{\nu'(1-\vartheta)(1-\de) p}{2}\right)\,dt\\
    &=\fint_{I} \left(\left(\fint_{B }\frac{|u-u_{Q }|^2}{\rho^2\mh^{{1-\de}}}\,dx\right)^\frac{\de}{1-\de}\fint_{B }\mh^\de\,dx\right)\,dt\\
    &\leq \sup_{I }\left(\fint_{B}\frac{|u-u_{Q }|^2}{ \rho^2\mh^{{1-\de}}}\,dx\right)^\frac{\de}{1-\de}\fiint_{Q}\mh^\de\,dz.
    \end{split}
\end{align*}
Applying Lemma~\ref{lem_KKM_4_3} and \eqref{eq_est_mz} , we arrive at
\[
\fint_{I}\left(\fint_{B}\frac{|u-u_{Q}|^r}{\rho^r}\,dx\right)^\frac{\nu'(1-\vartheta)\de p}{r}\,dt\le c\pi^\frac{2\delta}{1-\delta},
\]
and thus
\begin{align}\label{eq_927}
    \begin{split}
        \fiint_{Q }\frac{|u-u_{Q }|^{\de p}}{\rho^{\de p}}\,dz 
        &\le c \pi^\frac{2\delta}{(1-\delta)\nu'} \left(\fiint_{Q }  \left( H(z,|\na u|)+ H(z,|F|) \right)^{\vartheta\nu \delta } \,dz\right)^\frac{1}{\nu}\\
        & \le c \pi^{\delta p (1-\vartheta)} \left(\fiint_{Q }  \left( H(z,|\na u|)+ H(z,|F|) \right)^{\vartheta\nu \delta } \,dz\right)^\frac{1}{\nu}\\
        &\le \frac{1}{8}\pi^{\delta p} +c\left(\fiint_{Q }  \left( H(z,|\na u|)+ H(z,|F|) \right)^{\vartheta\nu \delta } \,dz\right)^\frac{1}{\vartheta\nu}.
    \end{split}
\end{align}

Then we consider the third term in \eqref{eq_915}. From \eqref{eq_holder_reg} and triangle inequality we obtain
\begin{align}\label{eq_929}
    \begin{split}
        c\fiint_{Q }a(z)^\de \frac{ |u-u_{Q}|^{\de q} }{\rho^{\de q}} \,dz
        &\leq c\fiint_{Q }\inf_{w\in Q}a(w)^\de \frac{|u-u_{Q}|^{\de q}}{\rho^{\de q}}\,dz \\
        &\qquad +c\fiint_{Q }\rho^{\de \al} \frac{|u-u_{Q}|^{\de q}}{\rho^{\de q}}\,dz.
    \end{split}
    \end{align}
Note that we estimate the second term as in \eqref{radius_absorb_1} to have
\[
c\fiint_{Q }\rho^{\de \al} \frac{|u-u_{Q}|^{\de q}}{\rho^{\de q}}\,dz\le \frac{1}{8}\pi^{\delta p}.
\]

Then we consider the first term on the right hand side of \eqref{eq_929}.
    Again, the conditions of Lemma~\ref{lem_sobo} are satisfied with $\sig=q$, $s=\sdel q$, $r^*=2\sdel$ and $\vartheta=\tfrac{\sdel}{\delta}$ for $\delta\in(\delta_2,1)$, where $\delta_2$ is retaken, if necessary, to satisfy
    \[
    \frac{n}{2\sdel} -\frac{n}{q} \le \frac{\sdel}{\delta_2} \left(  1-\frac{n}{\sdel q} +\frac{n}{2\sdel}  \right).
    \]
    We have $\sdel q=\vartheta (\delta q)$. Denoting $r=2\delta$, we get
    \begin{align*}
    \begin{split}
        \fiint_{Q}\inf_{w\in Q }a(w)^\de\frac{|u-u_{Q }|^{\de q}}{\rho^{\de q}}\,dz
        &\le \fint_{I} \left( \inf_{w\in Q }a(w) \fint_{B}  \frac{|u-u_{Q }|^{ q}}{\rho^{ q}}\,dx \right)^\delta \,dt    \\
        &\le c\fint_{I }   \left(    \fint_{B } \inf_{w\in Q}a(w)^{\vartheta \delta}\left(\frac{|u-u_{Q }|^{\vartheta (\de q)}}{ \rho^{\vartheta (\de q)}}+|\na u|^{\vartheta (\de q)}\right)\,dx \right.\\
        &\qquad\times \left.\left(\inf_{w\in Q}a(w)^\frac{r}{q}\fint_{B}\frac{|u-u_{Q}|^r}{\rho^r}\,dx\right)^{\frac{(1-\vartheta)(\de q)}{r}}\right)\,dt
    \end{split}
	\end{align*}
    Moreover, we denote $\ka=\tfrac{2}{2-q(1-\vartheta)(1-\de)}>1$, the Hölder conjugate of which is $\ka'=\tfrac{2}{q(1-\vartheta)(1-\de)}$. Then we have
	\begin{align*}
    \begin{split}
        &\fiint_{Q}\inf_{w\in Q }a(w)^\de\frac{|u-u_{Q }|^{\de q}}{\rho^{\de q}}\,dz\\
        &\le c\left(\fint_{I }\left(\fint_{B }\inf_{w\in Q }a(w)^{\vartheta\de}\left(\frac{|u-u_{Q }|^{\vartheta (\de q)}}{ \rho ^{\vartheta (\de q)}}+|\na u|^{\vartheta (\de q)}\right)\,dx\right)^\ka\,dt\right)^\frac{1}{\ka}\\
        &\qquad \times \left(\fint_{I }\left(\fint_{B  }\inf_{w\in Q }a(w)^\frac{r}{q}\frac{|u-u_{Q }|^r}{ \rho ^r}\,dx\right)^{\frac{\ka'(1-\vartheta)(\de q)}{r}}\,dt\right)^\frac{1}{\ka'}\\
        & \le c\left(\fiint_{Q }\inf_{w\in Q }a(w)^{\vartheta\de\ka}\left(\frac{|u-u_{Q }|^{\vartheta\ka (\de q)}}{ \rho ^{\vartheta\ka (\de q)}}+|\na u|^{\vartheta\ka (\de q)}\right)\,dz\right)^\frac{1}{\ka}\\
        &\qquad\times \left(\fint_{I }\left(\fint_{B }\pi^\frac{r(p-q)}{q}\frac{|u-u_{Q}|^r}{\rho^r}\,dx\right)^{\frac{\ka'(1-\vartheta)(\de q)}{r}}\,dt\right)^\frac{1}{\ka'},
    \end{split}
	\end{align*}
    where we used H\"older's inequality and \eqref{eq_pcase} in obtaining the last inequality.
    Here, note that again we have $\vartheta \kappa<1$. It follows from Lemma~\ref{sec4:lem:3} that
    \begin{align*}
        \begin{split}
            &\fiint_{Q}\inf_{w\in Q }a(w)^\de\frac{|u-u_{Q }|^{\de q}}{\rho^{\de q}}\,dz\\
            &\le c\left(\fiint_{Q } (  H(z,|\na u|) + H(z,|F|)  )^{\vartheta \delta\kappa}  \,dz\right)^\frac{1}{\ka}\\
        &\qquad\times \left(\fint_{I }\left(\fint_{B }\pi^\frac{r(p-q)}{q}\frac{|u-u_{Q}|^r}{\rho^r}\,dx\right)^{\frac{\ka'(1-\vartheta)(\de q)}{r}}\,dt\right)^\frac{1}{\ka'}.
        \end{split}
    \end{align*}
We start by considering the second integral. Again by Hölder's inequality with the conjugate pair $(\tfrac{1}{\de}, \tfrac{1}{1-\de})$, we obtain
\begin{align*}
    \begin{split}
        \fint_{B }\pi^\frac{r(p-q)}{q}\frac{|u-u_{Q }|^r}{ \rho ^r}\,dx
        &= \fint_{B  }\pi^\frac{2\de(p-q)}{q}\frac{|u-u_{Q  }|^{2\de}}{ \rho ^{2\de}\mh^{\de(1-\de)}}\mh^{\de(1-\de)}\,dx\\
        &\leq \left(\fint_{B }\pi^\frac{2(p-q)}{q}\frac{|u-u_{Q }|^2}{\rho^2\mh^{1-\de}}\,dx\right)^\de\left(\fint_{B}\mh^\de\,dx\right)^{1-\de}.
    \end{split}
\end{align*} 
Along with Lemma~\ref{lem_KKM_4_3} and \eqref{eq_est_mz}, we therefore get
\begin{align*}
    \begin{split}
    &\fint_{I}\left(\fint_{B}\pi^\frac{r(p-q)}{q}\frac{|u-u_{Q }|^r}{ \rho^r}\,dx\right)^\frac{\ka'(1-\vartheta)\de q}{r}\,dt\\ 
    &\leq \fint_{I }\left(\left(\fint_{B }\pi^\frac{2(p-q)}{q}\frac{|u-u_{Q }|^2}{ \rho ^2\mh^{1-\de}}\,dx\right)^\frac{\ka'(1-\vartheta)\de q}{2}\left(\fint_{B }\mh^\de\,dx\right)^\frac{\ka'(1-\vartheta)(1-\de) q}{2}\right)\,dt\\
    &=\fint_{I } \left(\left(\fint_{B }\pi^\frac{2(p-q)}{q}\frac{|u-u_{Q }|^2}{ \rho ^2\mh^{{1-\de}}}\,dx\right)^\frac{\de}{1-\de}\fint_{B }\mh^\de\,dx\right)\,dt\\
    &\leq \pi^\frac{2\delta(p-q)}{(1-\de)q} \sup_{I }\left(\fint_{B }\frac{|u-u_{Q }|^2}{ \rho^2\mh^{{1-\de}}}\,dx\right)^\frac{\de}{1-\de} \fiint_Q \mh^\delta \,dz\\
    &\le c \pi^\frac{2\delta(p-q)}{(1-\de)q} \pi^{(\delta p +2-p)\frac{\delta}{1-\delta}} \pi^{\delta p}\\
    &= c\pi^\frac{2\delta p}{(1-\delta)q}.
    \end{split}
\end{align*}
Recalling $\ka'$, we obtain from  Young's inequality that
\begin{align*}
\begin{split}
    \fiint_{Q}\inf_{w\in Q }a(w)^\de\frac{|u-u_{Q }|^{\de q}}{\rho^{\de q}}\,dz
    &\le c\pi^{(1-\vartheta)\delta p}\left(\fiint_{Q } (  H(z,|\na u|) + H(z,|F|)  )^{\vartheta \delta\kappa}  \,dz\right)^\frac{1}{\ka}\\
    &\le \frac{1}{8}\pi^{\delta p} + c\left(\fiint_{Q } (  H(z,|\na u|) + H(z,|F|)  )^{\vartheta \delta\kappa}  \,dz\right)^\frac{1}{\vartheta\ka}.
\end{split}
\end{align*}

Finally, we consider the first term on the right hand side of \eqref{eq_915}. By Hölder's inequality, the inequality in \eqref{eq_927} and  Young's inequality, we have that
\begin{align*}
\begin{split}
    \fiint_{Q} \frac{|u-u_{Q}|^2}{\pi^{2-p}\Pi^{1-\de}\rho^2 }\,dz 
    &\leq \pi^{\de p-2}\left(\fiint_{Q}\frac{|u-u_{Q }|^{\de p}}{\rho^{\de p} }\,dz\right)^\frac{2}{\de p}\\
    &\leq  c\pi^{\de p-2+ 2(1-\vartheta) }\left(\fiint_{Q} (H(z,|\na u|) + H(z,|F|))^{\vartheta\nu\de }\,dz\right)^\frac{2}{\nu\de p}\\
    &\le \frac{1}{8} \pi^{\delta p} +  c\left(\fiint_{Q} (H(z,|\na u|) + H(z,|F|))^{\vartheta\nu\de }\,dz\right)^\frac{1}{\vartheta\nu}.
\end{split}
\end{align*}

Combining the estimates for the terms on the right hand side of \eqref{eq_915}, we have 
\begin{align*}
    \begin{split}
        \pi^{\de p}
        &\leq \frac{1}{2}\pi^{\de p}+ c\left(\fiint_{Q }  (H(z,|\na u|) + H(z,|F|)  )^{\vartheta\nu\de }\,dz\right)^\frac{1}{\vartheta\nu}\\
        &\qquad +c\left(\fiint_{Q  } (H(z,|\na u|) + H(z,|F|)  )^{\vartheta\ka\de }\,dz\right)^\frac{1}{\vartheta\ka},
    \end{split}
\end{align*}
where $c = c(\data)$. On the other hand, note that
\[
1<\nu<\kappa<\frac{2}{2-q(1-\sdel)(1-\delta_2)}.
\]
We set
\[
\theta_0= \frac{2\sdel}{2-q(1-\sdel)(1-\delta_2)}.
\]
Then $\vartheta\kappa \delta <\theta_0$ holds and moreover, it follows from \eqref{def_del_2_2} that $\theta_0<\delta_2$. We obtain
\[
        \pi^{\de p}
        \leq c\left(\fiint_{Q }  (H(z,|\na u|) + H(z,|F|)  )^{\theta_0 }\,dz\right)^\frac{\delta}{\theta_0}.\\
\]
We have finished the proof of the first inequality in the statement.

In order to prove the second statement, we have from the first statement that
\begin{align*}
    \begin{split}
        \pi^{\delta p}
        &\le  c\pi^{ (\delta  - \theta_0)p }  \fiint_Q H(z,|\na u|)^{\theta_0
        } \,dz + c \fiint_Q H(z,|F|)^{\delta } \,dz.
    \end{split}
\end{align*}
To estimate further, we observe that
\begin{align*}
    \begin{split}
        c\pi^{  (\delta-\theta_0)p }  \fiint_Q H(z,|\na u|)^{\theta_0} \,dz
        &\le \frac{ 1 }{8} \pi^{ (\delta -\theta_0 ) p } \Pi^{\theta_0}  \\
        &\qquad +  c\frac{ \pi^{ (\delta -\theta_0 ) p  } }{|Q|} \iint_{Q\cap \Psi(  (8c)^{-\frac{1}{\theta_0}} \Pi   ) } H(z,|\na u|)^{\theta_0 } \,dz\\
        &\le \frac{1}{8}\Pi^{\delta}  +  c\frac{\Pi^{\delta- \theta_0 }}{|Q|} \iint_{Q\cap \Psi(  (8c)^{-\frac{1}{\theta_0}}  \Pi   ) } H(z,|\na u|)^{\theta_0   } \,dz.
    \end{split}
\end{align*}
Similarly, we also have
\begin{align*}
    \begin{split}
        c \fiint_Q H(z,|F|)^{\delta } \,dz
        &\le \frac{1}{8}\Pi^\delta +  c\frac{1}{|Q|} \iint_{Q\cap \Phi( (8c)^{-\frac{1}{\delta} }\Pi ) } H(z,|F|)^{\delta }   \,dz.
    \end{split}
\end{align*}
Combining these estimates and using \eqref{eq_pcase}, we get
\begin{align*}
    \frac{1}{2}\Pi^\delta 
    &\le \frac{1}{4}\Pi^\delta +  c\frac{\Pi^{\delta-  \theta_0}}{|Q|} \iint_{Q\cap \Psi(  (8c)^{-\frac{1}{\theta_0}} \Pi   ) } H(z,|\na u|)^{ \theta_0  } \,dz\\
    &\qquad  +  c\frac{1}{|Q|} \iint_{Q\cap \Phi( (8c)^{-\frac{1}{\delta} }\Pi ) } H(z,|F|)^{\delta }   \,dz.
\end{align*}

Absorbing the first term on the right hand side to the left hand side and then using \eqref{eq_p1}, we get
\begin{align*}
    \begin{split}
        \fiint_{8Q} H(z,|\na u|)^\delta  \,dz
        &\le 4 c\frac{\Pi^{\delta - \theta_0 }}{|Q|} \iint_{Q\cap \Psi(  (8c)^{-\frac{1}{\theta_0}} \Pi   ) } H(z,|\na u|)^{ \theta_0 } \,dz\\
    &\qquad  + 4 c\frac{1}{|Q|} \iint_{Q\cap \Phi( (8c)^{-\frac{1}{\delta} }\Pi ) } H(z,|F|)^{\delta }   \,dz,
    \end{split}
\end{align*}
or equivalently,
\begin{align*}
    \begin{split}
        \iint_{8Q} H(z,|\na u|)^\delta  \,dz
        &\le 8^{n+2} 4 c  \Pi^{\delta -\theta_0  } \iint_{Q\cap \Psi(  (8c)^{-\frac{1}{\theta_0}} \Pi   ) } H(z,|\na u|)^{\theta_0} \,dz\\
    &\qquad  + 8^{n+2} 4 c \iint_{Q\cap \Phi( (8c)^{-\frac{1}{\delta} }\Pi ) } H(z,|F|)^{\delta }   \,dz.
    \end{split}
\end{align*}
Recalling that $\theta_0<\delta$, we replace the above constants in the upper level sets by $ (8^{n+2} 4c )^\frac{1}{\theta_0} $ to get the conclusion.

\end{proof}

\subsection{$(p,q)$-intrinsic case}

We start with a $(p,q)$-intrinsic parabolic Poincar\'e inequality. Although we are considering very weak solutions, the proof of the following lemmas is the same as in \cite{KKM}.
\begin{lemma}\label{sec4:lem:5}
	Let $u$ be a very weak solution to \eqref{11}. Then, for $\theta\in((q-1)/p,\de]$ and $s\in[2\rho,4\rho]$, 
	there exists a constant $c=c(n,N,p,q,L_2)$ such that 
	\begin{align*}
	 \begin{split}
	     \fiint_{G_{s}^\pi(z_0)}H\left(z_0,\frac{|u-u_{G_{s}^\pi(z_0)}|}{s}\right)^{\theta} \,dz
			&\le c\fiint_{G_{s}^\pi(z_0)}H(z_0,|\na u|)^{\theta}\,dz\\
            &\qquad+c\fiint_{G_{s}^\pi(z_0)}H(z_0,|F|)^{\theta}\,dz,
	 \end{split}
	\end{align*}
  whenever $G_{4\rho}^{\pi}(z_0)\subset\Omega_T$ satisfies \eqref{eq_pqcase}-\eqref{eq_pq4}.
\end{lemma}

\begin{lemma}\label{sec4:lem:6}
	Let $u$ be a very weak solution to \eqref{11}. Then, for $\theta\in((q-1)/p,\de]$ and $s\in[2\rho,4\rho]$, there exists a constant $c=c(n,N,p,q,L_2)$ such that 
	\begin{align*}
			\begin{split}
			    \fiint_{G_{s}^\pi(z_0)} \frac{|u-u_{G_{s}^\pi(z_0)}|^{\theta p}}{s^{\theta p}}\,dz
			\le c\fiint_{G_{s}^\pi(z_0)} (|\na u|+|F| )^{\theta p}\,dz,
			\end{split}
	\end{align*}
whenever $G_{4\rho}^{\pi}(z_0)\subset\Omega_T$ satisfies \eqref{eq_pqcase}-\eqref{eq_pq4}.
\end{lemma}

Note that by \eqref{eq_pqcase}, \eqref{eq_pq2} and \eqref{eq_pq4}, we have
\[
    \fiint_{G_{4\rho}^\pi(z_0)} \frac{a(z_0)^\delta }{2^\delta}   (|\na u| + |F|  )^{\de q}\,dz \leq 2a(z_0)^\delta\pi^{\de q},
\]
and therefore 
\begin{align}\label{eq_847}
    \fiint_{G_{4\rho}^\pi(z_0)}(|\na u| +  |F|  )^{\de q}\,dz \leq 4\pi^{\de q}.
\end{align}

We denote
	\begin{align*}
		S(u,G_{\rho}^\la(z_0))=\sup_{J_{\rho}^\la(t_0)}\fint_{B_{\rho}(x_0)}\frac{|u-u_{G_{\rho}^\la(z_0)}|^2}{\rho^2\mh^{1-\de}}\,dx.
	\end{align*}
 
\begin{lemma}\label{lem_KKM_4_3_pq}
	Let $u$ be a very weak solution to \eqref{11}. Then there exists a constant $c=c(n,N,p,q,L_1,L_2)$ such that 
	\begin{align*}
		S(u,G_{2\rho}^\pi(z_0))=\sup_{J_{2\rho}^\pi(t_0)}\fint_{B_{2\rho}(x_0)}\frac{|u-u_{G_{2\rho}^\pi(z_0)}|^2}{(2\rho)^2\mh^{1-\de}}\,dx\le c\pi^2\Pi^{\de-1},
	\end{align*}	
 whenever $G_{16\rho}^{\pi}(z_0)\subset\Omega_T$ satisfies \eqref{eq_pqcase}-\eqref{eq_pq4}.
\end{lemma}
\begin{proof}
	By Lemma~\ref{lem_caccioppoli}, there exists a constant $c=c(n,p,q,L_1,L_2)$ such that
 \begin{equation}\label{sec5:61}
		\begin{split}
			&\frac{\Pi}{\pi^2}\sup_{J_{2\rho}^\pi(t_0)}\fint_{B_{2\rho}(x_0)}\frac{|u-u_{G_{2\rho}^\pi(z_0)}|^2}{(2\rho)^2\mh^{1-\de}}\,dx\\
			&\le c\fiint_{G_{4\rho}^\pi(z_0)}\left(\frac{|u-u_{G_{4\rho}^\pi(z_0)}|^{\de p}}{(4\rho)^{\de p}}+a(z)^\delta\frac{|u-u_{G_{4\rho}^\pi(z_0)}|^{\de q}}{(4\rho)^{\de q}}\right)\,dz\\
            &\qquad +c\frac{\Pi^\de}{\pi^2}\fiint_{G_{4\rho}^\pi(z_0)}\frac{|u-u_{G_{4\rho}^\pi(z_0)}|^2}{(4\rho)^2}\,dz + c\fiint_{G_{4\rho}^\pi(z_0)} H(z,|F|)^\delta \,dz.
		\end{split}
  \end{equation}
	
	For the first term on the right-hand side of \eqref{sec5:61}, we apply Lemma~\ref{sec4:lem:5} together with \eqref{eq_pq2} and \eqref{eq_pq3} to obtain
	\begin{align*}
		\begin{split}
			&\fiint_{G_{4\rho}^\pi(z_0)}\left(\frac{|u-u_{G_{4\rho}^\pi(z_0)}|^{\de p}}{4\rho^{\de p}}+a(z)^\delta\frac{|u-u_{G_{4\rho}^\pi(z_0)}|^{\de q}}{4\rho^{\de q}}\right)\,dz\\
			&\le 2\fiint_{G_{4\rho}^\pi(z_0)}H\left(z_0,\frac{|u-u_{G_{4\rho}^\pi(z_0)}|}{4\rho}\right)^\de\,dz\\
   &\le c\fiint_{Q_{4\rho}^\pi(z_0)} (H(z_0,|\na u|) + H(z_0,|F|)    )^\de\,dz\\
   &\le c\Pi^\de.
		\end{split}
	\end{align*}
    
For the second term on the right-hand side of \eqref{sec5:61} we obtain with Hölder's inequality, Lemma~\ref{sec4:lem:6} and \eqref{eq_847} that
	\begin{align*}
		\begin{split}
			\frac{\Pi^\de}{\pi^2}\fiint_{G_{4\rho}^\pi(z_0)}\frac{|u-u_{G_{4\rho}^\pi(z_0)}|^2}{(4\rho)^2}\,dz
            &\leq \frac{\Pi^\de}{\pi^2}\left(\fiint_{G_{4\rho}^\pi(z_0)}\frac{|u-u_{G_{4\rho}^\pi(z_0)}|^{\de p}}{(4\rho)^{\de p}}\,dz\right)^\frac{2}{\de p}\\
			&\le \frac{c\Pi^\de}{\pi^2}\left(\fiint_{G_{4\rho}^\pi(z_0)}
             (|\na u| + |F|   )^{\de p}\,dz\right)^\frac{2}{\de p}\\
            &\le \frac{c\Pi^\de}{\pi^2}\left(\fiint_{G_{4\rho}^\pi(z_0)}
            (|\na u| + |F| )^{\de q}\,dz\right)^\frac{2}{\de q} \\
            &\leq c\Pi^\de.
		\end{split}
	\end{align*}

    Finally we apply \eqref{eq_pq4} to the last term in \eqref{sec5:61} and conclude from \eqref{sec5:61} that
 \[
 \sup_{J_{2\rho}^\pi(t_0)}\fint_{B_{2\rho}(x_0)}\frac{|u-u_{G_{2\rho}^\pi(z_0)}|^2}{(2\rho)^2\mh^{1-\de}}\,dx \leq c\pi^2\Pi^{\de-1},
 \]
\end{proof}
which finishes the proof.

\begin{lemma}\label{lem_reverse_Hölder_pq}
	Let $u$ be a very weak solution to \eqref{11}.
There exist constants $c=c(\data)$, $\delta_2=\delta_2(\data)\in(\delta_1,1)$ with $2-q(1-\delta_2)>0$, $\theta_0=\theta_0(\data)\in(0,\delta_2) $ such that
	\begin{align*}
		\begin{split}
			\Pi^\de\le c\left(\fiint_{G_{2\rho}^\pi(z_0)} H(z,|\na u|)^{\theta_0}\,dz\right)^\frac{\delta}{\theta_0} + c \fiint_{G_{2\rho}^\pi(z_0)}H(z,|F|)^{\de}\,dz,
		\end{split}
	\end{align*}	
	whenever $G_{16\rho}^{\pi}(z_0)\subset\Omega_T$ satisfies \eqref{eq_pqcase} - \eqref{eq_pq4} and $\delta\in(\delta_2,1)$. Moreover, we have 
    \begin{align}\label{eq_revH_Psi_pq}
		\begin{split}
			\iint_{G_{16\rho}^\pi(z_0)}H(z,|\na u|)\,dz
			&\le c\Pi^{\delta -\theta_0 }\iint_{G_{2\rho}^\pi(z_0)\cap \Psi(c^{-1}\Pi)}H(z,|\na u|)^{\theta_0}\,dz\\
            &\qquad +c\iint_{G_{2\rho}^\pi(z_0)\cap \Phi(c^{-1}\Pi)}H(z,|F|)^\delta\,dz
		\end{split}
	\end{align}
 where $\Psi(\cdot)$ and $\Phi(\cdot)$ are defined in \eqref{sec2:2}
\end{lemma}

\begin{proof}
For simplicity, we write $G=G_{2\rho}^\pi(z_0)$, $B=B_{2\rho}(x_0)$, $J=J^\pi_{2\rho}(t_0)$.
By the Caccioppoli inequality in Lemma~\ref{lem_caccioppoli} and \eqref{eq_pq2} we have
\begin{align}\label{eq_915_pq}
	\begin{split}
	\Pi^\de 
    &\leq c\fiint_{G} \frac{|u-u_{G}|^2}{\pi^2\Pi^{-\de}\rho^2 }\,dz +c\fiint_{G}\left(\frac{ | u-u_{G } |^{\de p} }{\rho^{\de p}}+a(z_0)^\de\frac{ | u-u_{G }|^{\de q} }{\rho^{\de q}} \right)\,dz\\
    &\qquad+ c\fiint_{G} H(z,|F|)^\delta\,dz.
	\end{split}
\end{align}
 We estimate the second term first. As in the proof of Lemma~\ref{lem_reverse_Hölder}, we take $\sdel$, $\delta_2$ and $\vartheta=\tfrac{\sdel}{\delta}$ for $\delta\in(\delta_2,1)$. Then for
$\nu=\tfrac{2}{2-p(1-\vartheta)(1-\delta)}>1$ and $\nu'=\tfrac{2}{p(1-\vartheta)(1-\delta)}$, there holds
 \begin{align*}
    \begin{split}
        \fiint_{G}   \frac{|u-u_{G}|^{\de p}}{\rho^{\de p}}\,dz
        & \le c\left(\fiint_{G }\left(\frac{|u-u_{G }|^{\vartheta\nu (\de p)}}{\rho^{\vartheta\nu (\de p)}}+|\na u|^{\vartheta\nu (\de p)}\right)\,dz\right)^\frac{1}{\nu} \\
        &\qquad \times \left(  \sup_J \left(  \fint_B \frac{|u-u_G|^2}{\rho^2\hat{\Upsilon}^{1-\delta}} \,dx \right)^\frac{\delta}{1-\delta} \fiint_G \hat{\Upsilon}^\delta\,dz
        \right)^\frac{1}{\nu'},
    \end{split}
	\end{align*}
 where $\vartheta\nu<1$. 
To estimate further, we apply \eqref{eq_est_mz}, Lemma~\ref{sec4:lem:6}  and Lemma~\ref{lem_KKM_4_3_pq} to the right hand side and obtain 
\begin{align}\label{pq_est_p}
    \begin{split}
        \fiint_{G}   \frac{|u-u_{G}|^{\de p}}{\rho^{\de p}}\,dz
        & \le c   \pi^{\delta p (1-\vartheta)}  \left(\fiint_{G }  (|\na u|  +  |F|)   ^{\vartheta\nu (\de p)} \,dz\right)^\frac{1}{\nu}\\
        &\le \frac{1}{4} \pi^{\delta p } + \left(\fiint_{G }  (|\na u|  +  |F|)   ^{\vartheta\nu (\de p)} \,dz\right)^\frac{1}{\vartheta\nu} .
    \end{split}
\end{align}
Similarly, we also have
\begin{align*}
    \begin{split}
        & a(z_0)^\delta \fiint_{G} \frac{|u-u_{G}|^{\de q}}{\rho^{\de q}}\,dz\\
        &\le \left( \fiint_G a(z_0)^{\vartheta \delta \kappa} \left( \frac{|u-u_G|^{\vartheta \kappa (\delta q)}}{\rho^{\vartheta \kappa (\delta q)}} + |\na u|^{\vartheta \kappa (\delta q)}  \right)\,dz \right)^\frac{1}{\kappa} \\
        &\qquad \times a(z_0)^{  (1-\vartheta ) \delta }\left(   \sup_J \left(  \fint_B \frac{|u-u_G|^2}{\rho^2\hat{\Upsilon}^{1-\delta}} \,dx \right)^\frac{\delta}{1-\delta} \fiint_G \hat{\Upsilon}^\delta\,dz   \right)^\frac{1}{\kappa'},
    \end{split}
\end{align*}
where $\kappa=\tfrac{2}{2-q(1-\vartheta)(1-\delta)}>1$ and $\kappa'=\tfrac{2}{q(1-\vartheta)(1-\delta)}$. Again applying \eqref{eq_est_mz}, Lemma~\ref{sec4:lem:5} and Lemma~\ref{lem_KKM_4_3_pq}, it follows that
\begin{align*}
    \begin{split}
        a(z_0)^\delta \fiint_{G} \frac{|u-u_{G}|^{\de q}}{\rho^{\de q}}\,dz
        &\le a(z_0)^{ (1-\vartheta)  \delta } \pi^{(1-\vartheta)(\delta q)} \\
        &\qquad\times \left( \fiint_G   (H(z,|\na u|)  + H(z,|F|) )^\delta  \,dz \right)^\frac{1}{\kappa}\\
        &\le \frac{1}{4}a(z_0)^{\delta}\pi^{\delta q} \\
        &\qquad + \left( \fiint_G   (H(z,|\na u|)  + H(z,|F|) )^{\kappa\vartheta \delta}  \,dz \right)^\frac{1}{\kappa \vartheta}.
    \end{split}
\end{align*}

Finally, we consider the first term on the right hand side of \eqref{eq_915_pq}. Using \eqref{eq_pqcase}, \eqref{pq_est_p} and \eqref{eq_847}, we obtain
\begin{align*}
\begin{split}
    \fiint_{G }\frac{|u-u_{G }|^2}{\pi^2\Pi^{-\de}\rho^2 }\,dz 
    &\leq \Pi^\de\pi^{-2}\left(\fiint_{G }\frac{|u-u_{G  }|^{\de p}}{\rho^{\de p} }\,dz\right)^\frac{2}{\de p}\\
    &\le c \Pi^\de\pi^{-2+2(1-\vartheta)} \left(  \fiint_G ( |\na u| + |F|  )^{\vartheta \nu (\delta p)} \, dz \right)^\frac{2}{\nu (\delta p)}\\
    &\le c  (a(z_0)\pi^q)  ^\delta  \pi^{-\vartheta} \left(  \fiint_G ( |\na u| + |F|  )^{\vartheta \nu (\delta q)} \, dz \right)^\frac{1}{\nu (\delta q)}\\
    &= c a(z_0)^{\delta -\frac{\vartheta }{q}} \pi^{\delta q -\vartheta } \left(  \fiint_G a(z_0)^{\vartheta \nu \delta }  ( |\na u| + |F|  )^{\vartheta \nu (\delta q)} \, dz \right)^\frac{1}{\nu (\delta q)}.
\end{split}
\end{align*}
Therefore, Young's inequality gives
\[
\fiint_{G }\frac{|u-u_{G }|^2}{\pi^2\Pi^{-\de}\rho^2 }\,dz \le 
\frac{1}{4} a(z_0)^\delta \pi^{\delta q} + c\left(  \fiint_G a(z_0)^{\vartheta \nu \delta }  ( |\na u| + |F|  )^{\vartheta \nu (\delta q)} \, dz \right)^\frac{1}{ \vartheta \nu }
\]

Combining the estimates for the terms on the right hand side of \eqref{eq_915_pq} we have that
\begin{align*}
    \begin{split}
        \Pi^{\de p}
        &\leq \frac{1}{2}\Pi^{\de p}+ \left(\fiint_{G }  (H(z,|\na u|) + H(z,|F|)  )^{\vartheta\nu\de }\,dz\right)^\frac{1}{\vartheta\nu}\\
        &\qquad +c\left(\fiint_{G  } (H(z,|\na u|) + H(z,|F|)  )^{\vartheta\ka\de }\,dz\right)^\frac{1}{\vartheta\ka}.
    \end{split}
\end{align*}
The proof of the first statement in this lemma follows as in the proof of Lemma~\ref{lem_reverse_Hölder} by taking $\theta_0$.
Meanwhile, the second statement follows as in the proof of Lemma~\ref{lem_reverse_Hölder}.

\end{proof}

\section{Proof of Theorem~\ref{HI_theorem}}\label{sec_HI}
At the end of this section, we will take $\delta_0\in(\delta_2,1)$ satisfying Theorem~\ref{HI_theorem} and
\begin{align}\label{de_0_1}
    q < p+\frac{\alpha(2-q(1-\de_0))}{n+2},\qquad 2-q(1-\delta_0)>1,\qquad 2<\delta_0 p,
\end{align}
where $\delta_2\in(0,1)$ is in Lemma~\ref{lem_reverse_Hölder}. Note that the existence of a number $\de_ 0$ satisfying the above inequalities is guaranteed as \eqref{eq_range_q} holds with a strict inequality. We recall $\delta\in(\delta_0,1)$. Moreover, $R_0\in(0,1)$ will be taken in Lemma~\ref{lem_KKM_4_3}, Lemma~\ref{lemma_decay} and Lemma~\ref{p_q_comp} depending on $\data, \delta ,\| H(z,|\na u|) \|_{L^\delta(\Om_T)}, \| H(z,|F|) \|_{L^\delta(\Om_T)}$ and $\|u\|_{C(0,T;L^2(\Om,\mathbb{R}^N))}$.

\subsection{Stopping time argument}\label{sec_stopping_time}
First, we show that the intrinsic cylinders defined in Section~\ref{sec_setup} exist. Recall $\mfc>1$ defined in \eqref{def_mfc}.
Let $z_0\in \Om_T$ be such that $Q_{2R}(z_0)\subset \Om_T$ for $R\le R_0$. For $\delta\in(\delta_0,1)$, we denote
\begin{align}\label{sec6:01}
      \pi_0^{2}(\pi_0^p+\sup_{z\in Q_{2R}(z_0)}a(z)\pi_0^q)^{\de-1}=\fiint_{Q_{2R}(z_0)} ( H(z,|\na u|) + H(z,|F|)  )^\de\,dz+\mfc^\frac{2}{p}
\end{align}
and 
\[
\Pi_0=\pi_0^p+\sup_{z\in Q_{2R}(z_0)}a(z)\pi_0^q.
\]
It is easy to see that $\pi_0>\mfc^\frac{1}{p}$ and $\Pi_0>\mfc$. Therefore, the condition $\Pi>\mfc$ in Section~\ref{sec_setup} follows. 
Note that $2+q(\delta-1)>0$  by \eqref{de_0_1}.
For $\La>1$ and $\rho\in[R,2R]$, recall upper level sets $\Psi(\La)$ and $\Phi(\La)$ in \eqref{sec2:2}. We denote
\begin{align*}
    \begin{split}
        \Psi(\La,\rho)&=\Psi(\La)\cap Q_{\rho}=\{ z\in Q_{\rho}(z_0): H(z,|\na u(z)|)>\La\},\\
        \Phi(\La,\rho)&=\Phi(\La)\cap Q_{\rho}=\{ z\in Q_{\rho}(z_0): H(z,|F(z)|)>\La\}.
    \end{split}
\end{align*}

For the application of a stopping time argument, we set $R\le R_1<R_2
\le 2R$ and consider
\begin{align}\label{sec6:4}
	\Pi>\left(\frac{32 R}{R_2-R_1}\right)^\frac{q(n+2)}{2+p(\delta-1)}\Pi_0.
\end{align}
For any $w\in \Psi(\Pi,R_1)$, let $\pi_w>0$ be such that
\begin{align}\label{sec6:5}
    \Pi = \pi_w^p+a(w)\pi_w^q.
\end{align}
Note that since $0\le s\mapsto s^p+a(w)s^q$ is strictly increasing, $\pi_w$ is uniquely determined.
We claim that 
\begin{align}\label{eq_96}
    \pi_w > \left(\frac{ 32 R}{R_2-R_1}\right)^\frac{n+2}{2+p(\delta-1)}\pi_0.
\end{align}
For a contradiction, assume that \eqref{eq_96} does not hold. Then we have
\begin{align*}
	\Pi=\pi_w^p+a(w)\pi_w^q \leq \left(\frac{32 R}{R_2-R_1}\right)^{\frac{q(n+2)}{2+p(\delta-1)}}\left(\pi_0^p+a(w)\pi_0^q\right)\le \left(\frac{ 32 R}{R_2-R_1}\right)^{\frac{q(n+2)}{2+p(\delta-1)}}\Pi_0,
\end{align*}
which is a contradiction with \eqref{sec6:4} and therefore \eqref{eq_96} holds.

We now verify the properties of $p$-intrinsic cylinders in \eqref{eq_pcase}-\eqref{eq_rho_decay}. 

\begin{lemma}\label{p_stopping}
    For any $w\in \Psi(\Pi,R_1)$, there exists $\rho_{w}\in(0,(R_2-R_1)/16)$ such that
\begin{align}\label{sec6:6}
	\fiint_{Q_{\rho_{w}}^{\pi_w}(w)}  (H(z,|\na u|) +H(z,|F|)  )^\de\,dz= \pi_w^{\de p}
\end{align}
and 
\[
\fiint_{Q_s^{\pi_w}(w)}(H(z,|\na u|) +H(z,|F|)  )^\de\,dz < \pi_w^{\de p}
\]
for every $s\in(\rho_{w},R_2-R_1)$.
\end{lemma}

\begin{proof}
For any $s\in [(R_2-R_1)/16,R_2-R_1)$, we have by \eqref{sec6:01} and \eqref{eq_96} that
\begin{align*} 
	\begin{split}
		&\fiint_{Q_s^{\pi_w}(w)}( H(z,|\na u|) + H(z,|F|) )^\de \,dz\\
		&\le\pi_w^{p-2}\left(\frac{2R}{s}\right)^{n+2}\fiint_{Q_{2R}(z_0) } ( H(z,|\na u|) + H(z,|F|) )^\de \,dz\\
		&\le\left(\frac{ 32 R}{R_2-R_1}\right)^{n+2}\pi_w^{p-2}\pi_0^2(\pi_0^p+\sup_{z\in Q_{2R}(z_0) }a(z)\pi_0^q)^{\de-1}\\
        &\leq \left(\frac{ 32 R}{R_2-R_1}\right)^{n+2}\pi_w^{p-2}\pi_0^{2+p(\de -1)}\\
        &\leq \pi_w^{p-2}  \left(   \left(\frac{ 32 R}{R_2-R_1}\right)^\frac{n+2}{2+p(\delta-1)}    \pi_0  \right)^{2+p(\de -1)}\\
        &< \pi_w^{\de p}.
	\end{split}
\end{align*}
Here, to obtain the third inequality, we drop the term $\sup a(z)\pi_0^q$ since there is an exponent $\delta-1<0$, while to obtain the last inequality, we use \eqref{eq_96}.
On the other hand, since $ w\in \Psi(\pi_w^p,R_1)$, the conclusion holds by the Lebesgue differentiation theorem.
\end{proof}

The above lemma provides \eqref{eq_p1}-\eqref{eq_p2}.
We write condition \eqref{eq_pcase} in this setting as
\begin{align}\label{p_phase_w}
    \pi_w^p\ge a(w)\pi_w^q.
\end{align}
This condition proves the following estimate appearing in \eqref{eq_rho_decay}.

\begin{lemma}\label{lemma_decay}
    Suppose $w\in \Psi(\Pi,R_1)$ and \eqref{p_phase_w}. Then there exists $R_0\in(0,1)$ depending on $n,p,q,\alpha,[a]_\alpha,\delta,\| H(z,|\na u|) \|_{L^\delta(\Om_T)}, \| H(z,|F|) \|_{L^\delta(\Om_T)}$ such that    
    \[
    \rho_w^\alpha\pi_w^{q-p}\le \frac{1}{1+10[a]_\alpha}.
    \]
\end{lemma}
\begin{proof}
    We have by \eqref{sec6:6} that
    \[
    \pi_w^{\delta p} =\frac{\pi_w^{p-2}}{2|B_1|\rho_w^{n+2}} \iint_{Q_{\rho_w}^{\la_w}(w)} ( H(z,|\na u|) + H(z,|F|) )^\delta \, dz
    \]
    or equivalently, 
    \[
    \rho_w^{n+2} =\frac{\pi_w^{p(1-\delta)-2} }{2|B_1|} \iint_{Q_{\rho_w}^{\la_w}(w)} ( H(z,|\na u|) + H(z,|F|) )^\delta \, dz.
    \]
    Taking both sides to power $\tfrac{1}{n+2}$, we have
    \[
    \rho_w \le \pi_w^\frac{p(1-\delta)-2}{n+2} \left(  \frac{1}{2|B_1|} \iint_{Q_{\rho_w}^{\la_w}(w)} ( H(z,|\na u|) + H(z,|F|) )^\delta \, dz \right)^\frac{1}{n+2}.
    \]
    Now, let $\beta\in(0,\alpha)$ be chosen later and use  \eqref{sec6:6} and the above inequality to obtain
    \begin{align*}
        \begin{split}
            \rho_w^\alpha \pi_w^{q-p}
            &=\rho_w^{\alpha-\beta} \rho_w^\beta \left(  \frac{\pi_w^{p-2}}{2|B_1|\rho_w^{n+2}} \iint_{Q_{\rho_w}^{\la_w}(w)} ( H(z,|\na u|) + H(z,|F|) )^\delta \, dz   \right)^\frac{q-p}{\delta p}\\
            &\le \rho_w^{\alpha-\beta} \rho_w^\beta \left(  \frac{\pi_w^{p-2}}{2|B_1|\rho_w^{n+2}} \iint_{ \Om_T } ( H(z,|\na u|) + H(z,|F|) )^\delta \, dz   \right)^\frac{q-p}{\delta p}\\
            &\le \rho_w^{\alpha-\beta} \rho_w^{\beta -\frac{ (q-p)(n+2)    }{\delta p} } \pi_w^\frac{(p-2)(q-p)}{\delta p}  \\
            &\qquad \times \left(  \frac{ 1 }{2|B_1| } \iint_{ \Om_T } ( H(z,|\na u|) + H(z,|F|) )^\delta \, dz   \right)^\frac{q-p}{\delta p}\\ 
            &\le \rho_w^{\alpha-\beta} \pi_w ^{   \left( \frac{p(1-\delta) -2 }{n+2}  \right) \left( \beta -\frac{ (q-p)(n+2) }{\delta p}  \right) +  \frac{ (p-2) (q-p) } {\delta p}   }   \\
            &\qquad \times \left(  \frac{1}{2|B_1|} \iint_{\Om_T} ( H(z,|\na u|) + H(z,|F|) )^\delta \, dz   \right)^{  \frac{\beta}{n+2} }.
        \end{split}
    \end{align*}
    Looking at the exponent of $\pi_w$ with $\beta$ replaced by $\alpha$ for a moment, we observe that
    \begin{align*}
        \begin{split}
           & \left( \frac{p(1-\delta) -2 }{n+2}  \right)  \left( \alpha -\frac{(q-p)(n+2)}{\delta p}  \right) +  \frac{ (p-2) (q-p)  }{\delta p}    \\
           &= \frac{ (p(1-\delta) -2) \alpha }{n+2}  + (q-p)\\
           &\le \frac{ (q(1-\delta_0) -2) \alpha }{n+2}  + q - p\\
           &<0,
        \end{split}
    \end{align*}
   where to obtain the last inequality we used \eqref{de_0_1}. Furthermore, recalling $p(1-\delta_0)-2<0$ and $\tfrac{(p-2)(q-p)}{\delta p}>0$, we take $\beta\in(0,\alpha)$  so that
   \[
   \left( \frac{p(1-\delta) -2 }{n+2}  \right)  \left( \beta -\frac{(q-p)(n+2)}{\delta p}  \right) +  \frac{ (p-2) (q-p)  }{\delta p} =0 .
   \]
   Then we get
   \[
   \rho_w^\alpha \pi_w^{q-p}\le \rho_w^{\alpha -\beta} \left(  \frac{1}{2|B_1|} \iint_{\Om_T} ( H(z,|\na u|) + H(z,|F|) )^\delta \, dz   \right)^{     \frac{\beta}{n+2}  }.
   \]
   Finally, since $\rho_w\le R_0$, we take $R_0$ small enough so that 
   \[
   \rho_w^\alpha \pi_w ^{q-p}\le \frac{1}{1+10[a]_\alpha}.
   \]
   This completes the proof.
\end{proof}

Next, we move on to showing the properties \eqref{eq_pq2}-\eqref{eq_pq4} of $(p,q)$-intrinsic cylinder. In the next lemma, we denote the stopping time of the radius as $\rho_w$. Although the same notation appears in \eqref{sec6:6} for $p$-intrinsic cylinder, it is clear from the context which radius is referred to.

\begin{lemma}\label{pq_stopping}
    For any $w\in \Psi(\Pi,R_1)$, there exists $\rho_{w}\in(0,(R_2-R_1)/16)$ such that
    \begin{align}\label{pq_sec6:6}
	\fiint_{G_{\rho_{w}}^{\pi_w}(w)}  ( H(z,|\na u|) + H(z,|F|)  )^\de\,dz= \Pi^\de
\end{align}
and 
\[
\fiint_{G_s^{\pi_w}(w)}  ( H(z,|\na u|) + H(z,|F|)  )^\de\,dz < \Pi^\de
\]
for every $s\in(\rho_{w},R_2-R_1)$.
\end{lemma}

\begin{proof}
For any $s\in [(R_2-R_1)/16,R_2-R_1)$ we have by \eqref{sec6:01}, \eqref{eq_96} and \eqref{sec6:5} that
\begin{align*} 
	\begin{split}
		&\fiint_{G_s^{\pi_w}(w)} ( H(z,|\na u|) + H(z,|F|)  )^\de \,dz\\
		&\le\frac{\Pi}{\pi_w^2}\left(\frac{2R}{s}\right)^{n+2}\fiint_{Q_{2R} (z_0) } ( H(z,|\na u|) + H(z,|F|)  )^\de\,dz\\
		&\le\left(\frac{32 R}{R_2-R_1}\right)^{n+2}\frac{\Pi}{\pi_w^2}\pi_0^{2}(\pi_0^p+\sup_{z\in Q_{2R}(z_0) }a(z)\pi_0^q)^{\de-1}\\
        &\le\left(\frac{32 R}{R_2-R_1}\right)^{n+2}\frac{\Pi}{\pi_w^2}\pi_0^{2}(\pi_0^p+a(w)\pi_0^q)^{\de-1}\\
        &<\left(\frac{32 R}{R_2-R_1}\right)^{n+2- \frac{(n+2)(2+p(\delta-1)) }{2+p(\delta-1)} }\frac{\Pi}{\pi_w^2}\pi_w^{2}(\pi_w^p+a(w)\pi_w^q)^{\de-1}\\
        &= \Pi^\de.
	\end{split}
\end{align*}
As in the $p$-intrinsic case, the conclusion follows from the Lebesgue differentiation theorem since $w\in \Psi(\Pi,R_1)$.
\end{proof}

The previous lemma proves \eqref{eq_pq3}-\eqref{eq_pq4}. Next, we consider the counterpart of \eqref{p_phase_w} appearing in \eqref{eq_pqcase}, i.e., the $(p,q)$-intrinsic case
\begin{align}\label{pq_phase_w}
    \pi_w^p< a( w)\pi_w^q.
\end{align}
The above condition along with the previous lemma implies \eqref{eq_pq2}.
\begin{lemma}\label{p_q_comp}
    Suppose $w\in \Psi(\Pi,R_1)$ and \eqref{pq_phase_w}. Then there exists $R_0\in(0,1)$ depending on $n,\alpha,[a]_\alpha, \delta, \| H(z,|\na u|) \|_{L^\delta(\Om_T)}, \| H(z,|F|) \|_{L^\delta(\Om_T)}$ such that for any $z\in Q_{10\rho_w}(w)$, we have
    \[
    \frac{a(w)}{2}\le a(z)\le 2a(w).
    \]
\end{lemma}
\begin{proof}
    It is enough to prove
    \begin{align}\label{eq_518}
    2[a]_\alpha(10\rho_w)^\alpha \le a(w).    
    \end{align}
    Indeed, if the above inequality holds, then it follows that
    \[
    2[a]_\alpha(10\rho_w)^\alpha \le a(w)\le \inf_{z\in Q_{10\rho_w}(w)}a(z)+[a]_\alpha(10\rho_w)^\alpha
    \]
    and thus we have
    \[
    [a]_\alpha(10\rho_w)^\alpha\le \inf_{z\in Q_{10\rho_w}(w)}a(z)
    \]
    and
    \[
    \sup_{z\in Q_{10\rho_w}(w)}a(z)\le \inf_{z\in Q_{10\rho_w}(w)}a(z)+[a]_\alpha(10\rho_w)^\alpha \le 2\inf_{z\in Q_{10\rho_w}(w)}a(z).
    \]
    The statement of the lemma follows from the above inequality.

    Now we prove \eqref{eq_518} by contradiction. Assume we have
    \begin{align}\label{false}
        a(w) < 2[a]_\alpha(10\rho_w)^\alpha.
    \end{align}
    Recall $\Pi=H(w,\pi_w)$.  Then along with \eqref{pq_sec6:6} and \eqref{pq_phase_w}, we get
    \begin{align*}
        \begin{split}
        (a(w)\pi_w^q)^\delta
            &\le \Pi^\delta \\
            &= \fiint_{G_{\rho_w}^{\pi_w}(w)} (H(z,|\na u|) + H(z,|F|) )^\delta\,dz\\
            &\le \frac{\pi_w^{-2}(\pi_w^p+a(w)\pi^q_w)}{2|B_1|\rho_w^{n+2}}\iint_{ \Om_T } (H(z,|\na u|) + H(z,|F|) )^\delta\,dz\\
            &\le \frac{a(w)\pi_w^{q-2}}{|B_1|\rho_w^{n+2}}\iint_{ \Om_T } (H(z,|\na u|) + H(z,|F|) )^\delta\,dz.
        \end{split}
    \end{align*}
    Multiplying both sides by $(a(w)\pi_w^q)^{-\delta} \rho_w^{n+2}$ and then using \eqref{false}, it follows that
    \begin{align*}
        \begin{split}
            \rho_w^{n+2}
            &\le \pi_w^{(1-\delta)q-2}\frac{a(w)^{1-\delta}}{|B_1|}\iint_{ \Om_T } (H(z,|\na u|) + H(z,|F|) )^\delta\,dz\\
            &\le \pi_w^{(1-\delta)q-2}\frac{(20[a]_\alpha \rho_w)^{1-\delta}}{|B_1|}\iint_{ \Om_T } (H(z,|\na u|) + H(z,|F|) )^\delta\,dz\\
            &\le \pi_w^{(1-\delta)q-2}\frac{(1+20[a]_\alpha) \rho_w^{1-\delta}}{|B_1|}\iint_{ \Om_T } (H(z,|\na u|) + H(z,|F|) )^\delta\,dz.
        \end{split}
    \end{align*}
   Since $\rho_w\le R_0$, we choose $R_0$ depending on $n,\alpha,[a]_\alpha, \delta,\| H(z,|\na u|) \|_{L^\delta(\Om_T)}, \linebreak \| H(z,|F|) \|_{L^\delta(\Om_T)}$ small enough to have
   \[
   \rho_w^{n+2}\le \left(\frac{1}{20[a]_\alpha}\right)^\frac{n+2}{\alpha} \pi_w^{(1-\delta)q -2},
   \]
   or equivalently,
   \[
   \rho_w^\al \le \frac{1}{20[a]_\alpha}\pi_w^\frac{((1-\delta)q -2)\alpha}{n+2}.
   \]
    Applying the above inequality to \eqref{pq_phase_w} along with \eqref{false}, we get
    \begin{align*}
            \pi_w^p 
            < a(w)\pi_w^q
            \le 20[a]_\alpha\rho_w^\alpha \pi_w^q
            \le \pi_w^{q+\frac{((1-\delta_0)q-2)\alpha}{n+2}}\le \pi_w^p,
    \end{align*}
    where to obtain the last inequality, we have used the choice of $\de_0$ in \eqref{de_0_1}. This is a contradiction and thus \eqref{eq_518} holds. The proof is completed.
\end{proof}

\subsection{Vitali type covering argument}
For each $ w\in \Psi(\Pi,R_1)$, we consider
\[
U ( w)= \begin{cases}
Q_{2\rho_{ w}}^{\pi_w}( w)&\text{ in $p$-intrinsic case,}\\    G_{2\rho_{ w}}^{\pi_w}( w)&\text{ in $(p,q)$-intrinsic case.}
\end{cases}
\]
We prove a Vitali type covering lemma for this collection of intrinsic cylinders.
We denote
\begin{align}\label{def_mathcal_F}
    \mathcal{F}=\left\{U ( w):  w\in \Psi(\Pi,R_1)\right\}.
\end{align}
\begin{lemma}
    For $U(v),U(w)\in\mathcal{F}$, suppose $U(v)\cap U(w)\ne \emptyset$ and $\rho_w\le 2\rho_v$. Then we have
    \[
    \pi_v\le 2^{\frac{1}{p}} \pi_w.
    \]
\end{lemma}
\begin{proof}
    Note that by the assumptions in the statement and the standard Vitali covering lemma, we have
    \begin{align}\label{vitali_inclusion}
        Q_{2\rho_w}(w) \subset Q_{10\rho_v}(v).
    \end{align}
    
    To prove the lemma, we assume for a contradiction that
    \[
     \pi_w \le 2^{-\frac{1}{p}} \pi_v.
    \]
    We divide into cases.

    \textit{Case}~1: Suppose $U(v)$ is $p$-intrinsic. We observe that by \eqref{vitali_inclusion}
    \[
    \Pi=\pi_w^p+a(w)\pi_w^q \le \pi_w^p+a(v)\pi_w^q + [a]_\alpha(10\rho_v)^\alpha \pi_w^q.
    \]
    Using the counterassumption and Lemma~\ref{lemma_decay}, we get
    \[
    \Pi\le \frac{1}{2} \left( \pi_v^p+a(v)\pi_v^q+ \frac{10[a]_\alpha}{1+10[a]_\alpha}\pi_v^p \right)<\pi_v^p+a(v)\pi_v^q=\Pi.
    \]
    This is a contradiction and the conclusion follows in this case.

    \textit{Case}~2: Suppose $U(w)$ is $(p,q)$-intrinsic. Employing \eqref{vitali_inclusion} and Lemma~\ref{p_q_comp}, we have
    \[
    \Pi=\pi_w+a(w)\pi_w^q\le \pi_w+2a(v)\pi_w^q\le \frac{1}{2}\left( \pi_v^p+2a(v)\pi_v^q \right)<\Pi,
    \]
    where to obtain the last inequality, we used the counterassumption. Again, this is a contradiction and the proof is completed.
\end{proof}

After this the argument to show that the following Vitali property holds is as in \cite{KKM}. We omit the proof.
\begin{lemma}\label{vitali_lem}
    Let $\mathcal{F}$ be defined as in \eqref{def_mathcal_F}. Then there exists a pairwise disjoint countable subcollection $\{U_i\}_{i\in \mathbb{N}}$ of $\mathcal{F}$ such that for any $U(w)\in \mathcal{F} $, there exists $i\in\mathbb{N}$ satisfying
\[
U(w) \subset 8U_i.
\]
\end{lemma}

\subsection{Final proof of the gradient estimate}

\begin{proof}[Proof of Theorem~\ref{HI_theorem}]

  Let $\{U_i\}_{i\in \mathbb{N}}$ be as in Lemma~\ref{vitali_lem}.
  We have shown in Section~\ref{sec_stopping_time} that these cylinders satisfy the properties required in Lemma~\ref{lem_reverse_Hölder} and Lemma~\ref{lem_reverse_Hölder_pq}. Therefore, we obtain from \eqref{eq_revH_Psi} and \eqref{eq_revH_Psi_pq} that
	\begin{align}\label{eq_930}
		\begin{split}
			\iint_{8 U_i}H(z,|\na u|)^\de\,dz
			&\le c\Pi^{\de-\theta_0}\iint_{ U_i\cap \Psi(c^{-1}\Pi)}H(z,|\na u|)^{\theta_0}\,dz\\
            &\qquad+c\iint_{ U_i\cap \Phi(c^{-1}\Pi)}H(z,|F|)^{\de}\,dz
		\end{split}
	\end{align}
for some $\theta_0\in(0,\delta)$ and any $i\in\NN$ and $\Pi>\left(\tfrac{32 R}{R_2-R_1}\right)^\frac{q(n+2)}{2+p(\delta-1)}\Pi_0$. Here $\Pi_0$ is defined in \eqref{sec6:01} and $R\le R_1<R_2\le 2R$. By summing over $i$ and applying the fact that the cylinders in $\{U_i\}_{i\in\mathbb{N}}$ are pairwise disjoint, we obtain from \eqref{eq_930} that
\begin{equation}\label{sec6:25}
\begin{split}
\iint_{\Psi(\Pi,R_1)}H(z,|\na u|)^\de\,dz
&\le\sum_{j=1}^\infty\iint_{8 U_{i}}H(z,|\na u|)^\de\,dz\\
&\le c\Pi^{\de -\theta_0}\iint_{\Psi(c^{-1}\Pi,R_2)}H(z,|\na u|)^{\theta_0}\,dz\\
&\qquad +c \iint_{\Phi(c^{-1}\Pi,R_2)}H(z,|F|)^{\de}\,dz.
\end{split}
\end{equation}
Moreover, since
\[
\iint_{\Psi(c^{-1}\Pi,R_1)\setminus \Psi(\Pi,R_1)}H(z,|\na u|)^\de\,dz
\le\Pi^{\de - \theta_0}\iint_{\Psi(c^{-1}\Pi,R_2)}H(z,|\na u|)^{\theta_0}\,dz,
\]
we conclude from \eqref{sec6:25} that
\begin{equation}\label{sec6:26}
	\begin{split}
		\iint_{\Psi(c^{-1}\Pi,R_1)}H(z,|\na u|)^\de\,dz
		&\le c\Pi^{\de - \theta_0 }\iint_{\Psi(c^{-1}\Pi,R_2)}H(z,|\na u|)^{ \theta_0 }\,dz\\
        &\qquad + c \iint_{\Phi(c^{-1}\Pi,R_2)}H(z,|F|)^{\de}\,dz.
	\end{split}
\end{equation}

For $k\in\mathbb N$, let
\[		
H(z,|\na u|)_k=\min\{H(z,|\na u|),k\}
\]
and
\[
\Psi_k(\Pi,\rho)=\{z\in Q_{\rho}(z_0):H(z,|\na u|)_k>\Pi\}.
\]
It is easy to see that if $\Pi>k$, then $\Psi_k(\Pi,\rho)=\emptyset$ and if $\Pi\le k$, then $\Psi_k( \Pi,\rho)=\Psi(\Pi,\rho)$. Therefore, we deduce from \eqref{sec6:26} that
\begin{align*}
	\begin{split}
		&\iint_{\Psi_k(c^{-1}\Pi,R_1)}H(z,|\na u|)_k^{\de - \theta_0}H(z,|\na u|)^{\theta_0}\,dz\\
		&\le c\Pi^{\de- \theta_0}\iint_{\Psi_k(c^{-1}\Pi,R_2)}H(z,|\na u|)^{\theta_0}\,dz 
        +  c \iint_{\Phi(c^{-1}\Pi,R_2)}H(z,|F|)^{\de}\,dz.
	\end{split}
\end{align*}
With the constant $c$ as above, we denote
\begin{align}\label{def_pi_1}
	\Pi_1=c^{-1}\left(\frac{32 R}{R_2-R_1}\right)^\frac{q(n+2)}{2+p(\de-1)}\Pi_0.
\end{align}
We conclude from the above that for any $\Pi>\Pi_1$ we have
\begin{align*}
	\begin{split}
		&\iint_{\Psi_k(\Pi,R_1)}H(z,|\na u|)_k^{\de - \theta_0}H(z,|\na u|)^{\theta_0}\,dz\\
		&\le c\Pi^{\de - \theta_0}\iint_{\Psi_k(\Pi,R_2)}H(z,|\na u|)^{\theta_0}\,dz 
        +  c \iint_{\Phi(\Pi,R_2)}H(z,|F|)^{\de}\,dz.
	\end{split}
\end{align*}

We multiply the above inequality by $\Pi^{-\delta}$ and integrate each term over $(\Pi_1,\infty)$, which implies
\begin{align*}
	\begin{split}
		\mathrm{I}
        &=\int_{\Pi_1}^{\infty}\Pi^{-\delta}\iint_{\Psi_k(\Pi,R_1)}H(z,|\na u|)_k^{\de - \theta_0 }H(z,|\na u|)^{\theta_0}\,dz\,d\Pi\\
		&\le c\int_{\Pi_1}^{\infty}\Pi^{-\theta_0}\iint_{\Psi_k(\Pi,R_2)}H(z,|\na u|)^{\theta_0}\,dz\,d\Pi \\
        &\qquad+ c \int_{\Pi_1}^{\infty} \Pi^{-\delta} \iint_{\Phi(\Pi,R_2)}H(z,|F|)^{\de}\,dz\,d\Pi\\
        &= \mathrm{II} + \mathrm{III}.
	\end{split}
\end{align*}

We apply Fubini's theorem to $\mathrm{I}$ and obtain
\begin{align*}
	\begin{split}
		\mathrm{I}
		&= \iint_{\Psi_k(\Pi_1,R_1)}H(z,|\na u|)_k^{\de- \theta_0}H(z,|\na u|)^{\theta_0}\int_{\Pi_1}^{H(z,|\na u|)_k}\Pi^{-\delta}\,d\Pi\,dz\\
        &=\frac{1}{1-\de}\iint_{\Psi_k(\Pi_1,R_1)}H(z,|\na u|)_k^{1-\theta_0 }H(z,|\na u|)^{\theta_0}\,dz\\
		&\qquad-\frac{1}{1-\de}\Pi_1^{1-\de}\iint_{\Psi_k(\Pi_1,R_1)}H(z,|\na u|)_k^{\delta-\theta_0}H(z,|\na u|)^{\theta_0}\,dz.
	\end{split}
\end{align*}
Also, since we have the following estimate
\begin{align*}
	\begin{split}
		&\iint_{Q_{R_1}(z_0) \setminus \Psi_k(\Pi_1,R_1)}H(z,|\na u|)_k^{1-\theta_0}H(z,|\na u|)^{\theta_0}\,dz\\
		&\le \Pi_1^{1-\de}\iint_{Q_{2R}(z_0)}H(z,|\na u|)_k^{\de -  \theta_0 }H(z,|\na u|)^{\theta_0}\,dz,
	\end{split}
\end{align*}
we have
\begin{align*}
	\begin{split}
		\mathrm{I}
        &\ge \frac{1}{1-\de}\iint_{Q_{R_1}(z_0)} H(z,|\na u|)_k^{1-\theta_0}H(z,|\na u|)^{\theta_0}\,dz\\
		&\qquad-\frac{2}{1-\de}\Pi_1^{1-\de}\iint_{Q_{2R}(z_0)}H(z,|\na u|)_k^{\de - \theta_0 }H(z,|\na u|)^{\theta_0}\,dz.
	\end{split}
\end{align*}

We estimate $\mathrm{II}$ and $\mathrm{III}$ as above and obtain 
\begin{align*}
		\begin{split}
		\mathrm{II}+ \mathrm{III}
        &\le\frac{c}{1-\theta_0 }\iint_{Q_{R_2}(z_0)}H(z,|\na u|)_k^{1-\theta_0}H(z,|\na u|)^{\theta_0} \,dz\\
        &\qquad + \frac{c}{1-\delta} \iint_{Q_{2R}(z_0)} H(z,|F|)\,dz.
		\end{split}
\end{align*}
Therefore, we have
\begin{align*}
    \begin{split}
        &\iint_{Q_{R_1}(z_0)} H(z,|\na u|)_k^{1-\theta_0}H(z,|\na u|)^{\theta_0}\,dz\\
        &\le \frac{c(1-\delta)}{1-\theta_0} \iint_{Q_{R_2}(z_0)} H(z,|\na u|)_k^{1-\theta_0}H(z,|\na u|)^{\theta_0}\,dz  + c\iint_{Q_{2R}(z_0)} H(z,|F|)\,dz\\
        & \qquad + \Pi_1^{1-\delta}\iint_{Q_{2R}(z_0)} H(z,|\na u|)_k^{\delta -\theta_0} H(z,|\na u|)^{\theta_0} \,dz.
    \end{split}
\end{align*}
Taking $\delta_0=\delta_0(\data)$ close enough to 1 so that for $\delta\in(\delta_0,1)$ we have
\[
\frac{c(1-\delta)}{1-\theta_0}\le \frac{1}{2},
\]
we obtain
\begin{align*}
    \begin{split}
        &\iint_{Q_{R_1}(z_0)} H(z,|\na u|)_k^{1-\theta_0}H(z,|\na u|)^{\theta_0}\,dz\\
        &\le \frac{1}{2} \iint_{Q_{R_2}(z_0)} H(z,|\na u|)_k^{1-\theta_0}H(z,|\na u|)^{\theta_0}\,dz  + c\iint_{Q_{2R}(z_0)} H(z,|F|)\,dz\\
                & \qquad + \Pi_1^{1-\delta}\iint_{Q_{2R}(z_0)} H(z,|\na u|)_k^{\delta -\theta_0} H(z,|\na u|)^{\theta_0} \,dz.
    \end{split}
\end{align*}
Recall the definition of $\Pi_1$ in \eqref{def_pi_1}. By applying Lemma~\ref{lem_iter} we get
\begin{align*}
    \begin{split}
        &\iint_{Q_{R}(z_0)} H(z,|\na u|)_k^{1-\theta_0}H(z,|\na u|)^{\theta_0}\,dz\\
        &\le  c\Pi_0^{1-\delta}\iint_{Q_{2R}(z_0)} H(z,|\na u|)_k^{\delta -\theta_0} H(z,|\na u|)^{\theta_0} \,dz+ c\iint_{Q_{2R}(z_0)} H(z,|F|)\,dz.
    \end{split}
\end{align*}
Letting $k\longrightarrow\infty$ and recalling \eqref{sec6:01}, we obtain 
\begin{align}\label{eq_final}
    \begin{split}
        \iint_{Q_{R}(z_0)} H(z,|\na u|)\,dz
        &\le c\Pi_0^{1-\delta}\iint_{Q_{2R}(z_0)} H(z,|\na u|)^{\delta} \,dz+ c\iint_{Q_{2R}(z_0)} H(z,|F|)\,dz.
    \end{split}
\end{align}
Finally, note from \eqref{sec6:01} that $\pi_0>1$ and thus, 
\[
\pi_0^{2+q(\delta-1)}\le c\fiint_{Q_{2R}(z_0)} (H(z,|\na u|) + H(z,|F|) )^\delta\,dz+ \mfc^\frac{2}{p} 
\]
and
\[
\Pi_0\le (1+\|a\|_{L^\infty(\Om_T)}) \pi_0^q \le c \left( \fiint_{Q_{2R}(z_0)} (H(z,|\na u|) + H(z,|F|) )^\delta\,dz+ \mfc^\frac{2}{p}  \right)^\frac{q}{2-q(1-\delta)},
\]
where $c=c(\data,\|a\|_{L^\infty(\Om_T)})$.
Furthermore, recalling $\mfc$ depending on $\data$ and $\| (H(z,|\na u|)+H(z,|F|)+1 )\|_{L^\delta(\Om_T)}$,
the desired estimate follow by substituting this inequality to \eqref{eq_final}.
\end{proof}

\section*{Acknowledgement}
 W. Kim has been supported by KIAS individual Grant (HP105501). L. Särkiö has been supported by a doctoral training grant from Vilho, Yrj\"o and Kalle V\"ais\"al\"a Foundation. Part of this work was done during his visit at KIAS in 2025.

\section*{Declarations}
\textbf{Data availability} No data were used to support this study.

\textbf{Competing interests}  The authors have no competing interests to declare that are relevant to the content of this
article.

\end{document}